\documentclass[10pt]{amsart}
\usepackage{amsmath}
\usepackage{amsxtra}
\usepackage{amscd}
\usepackage{amsthm}
\usepackage{amsfonts}
\usepackage{amssymb}
\usepackage{eucal}
\usepackage[all]{xy}
\usepackage{graphicx}

\newtheorem{conj}[subsubsection]{Conjecture}
\newtheorem{prop}[subsubsection]{Proposition}
\newtheorem{propconstr}[subsubsection]{Proposition-Construction}

\newtheorem{cor}[subsubsection]{Corollary}
\newtheorem{lem}[subsubsection]{Lemma}
\newtheorem{defn}[subsubsection]{Definition}

\newtheorem{thm}[subsubsection]{Theorem}
\newtheorem{quest}[subsubsection]{Question}

\numberwithin{equation}{section}

\theoremstyle{remark}
\newtheorem{rem}[subsubsection]{Remark}

\newcommand{\conjref}[1]{Conjecture~\ref{#1}}

\newcommand{\lemref}[1]{Lemma~\ref{#1}}
\newcommand{\thmref}[1]{Theorem~\ref{#1}}
\newcommand{\secref}[1]{Sect.~\ref{#1}}
\newcommand{\corref}[1]{Corollary~\ref{#1}}
\newcommand{\propref}[1]{Proposition~\ref{#1}}

\newcommand{\nc}{\newcommand}

\nc{\ssec}{\subsection}
\nc{\sssec}{\subsubsection}

\nc{\renc}{\renewcommand}
\nc{\on}{\operatorname}
\nc\ol{\overline}
\nc\wt{\widetilde}
\nc{\Loc}{\on{Loc}}
\nc{\Bun}{\on{Bun}}
\nc{\BQ}{{\mathbb{Q}}}
\nc{\BA}{{\mathbb{A}}}
\nc{\BC}{{\mathbb{C}}}
\nc{\BH}{{\mathbb{H}}}
\nc{\BG}{{\mathbb{G}}}
\nc{\BK}{{\mathbb{K}}}
\nc{\BN}{{\mathbb{N}}}
\nc{\BD}{{\mathbb{D}}}
\nc{\BE}{{\mathbb{E}}}
\nc{\BV}{{\mathbb{V}}}
\nc{\BZ}{{\mathbb{Z}}}
\nc{\BL}{{\mathbb{L}}}
\nc{\CA}{{\mathcal{A}}}
\nc{\CC}{{\mathcal{C}}}
\nc{\CG}{{\mathcal{G}}}
\nc{\CI}{{\mathcal{I}}}
\nc{\CJ}{{\mathcal{J}}}
\nc{\CO}{{\mathcal{O}}}
\nc{\CP}{{\mathcal{P}}}
\nc{\CR}{{\mathcal{R}}}
\nc{\CU}{{\mathcal{U}}}
\nc{\CV}{{\mathcal{V}}}
\nc{\CW}{{\mathcal{W}}}
\nc{\CK}{{\mathcal{K}}}
\nc{\CM}{{\mathcal{M}}}
\nc{\CN}{{\mathcal{N}}}
\nc{\CL}{{\mathcal{L}}}
\nc{\CT}{{\mathcal{T}}}
\nc{\CE}{{\mathcal{E}}}
\nc{\CF}{{\mathcal{F}}}
\nc{\CX}{{\mathcal{X}}}
\nc{\CY}{{\mathcal{Y}}}
\nc{\CZ}{{\mathcal{Z}}}

\nc{\D}{{\mathcal{D}}}
\nc{\fd}{{\mathfrak{d}}}
\nc{\fg}{{\mathfrak{g}}}
\nc{\fD}{{\mathfrak{D}}}
\nc{\fh}{{\mathfrak{h}}}
\nc{\fl}{{\mathfrak{l}}}
\nc{\fn}{{\mathfrak{n}}}
\nc{\sM}{{\mathsf M}}
\nc{\ppart}{(\!(t)\!)}
\nc{\qqart}{[\![t]\!]}
\nc{\hg}{{\widehat\fg}}
\nc{\sA}{{\mathsf A}}
\nc{\sB}{{\mathsf B}}
\nc{\sF}{{\mathsf F}}
\nc{\sG}{{\mathsf G}}
\nc{\sH}{{\mathsf H}}
\nc{\sK}{{\mathsf K}}
\nc{\sS}{{\mathsf S}}
\nc{\sT}{{\mathsf T}}
\nc{\sk}{{\mathsf k}}
\nc{\sj}{{\mathsf j}}

\nc{\bC}{{\mathbf{C}}}
\nc{\bZ}{{\mathbf{Z}}}
\nc{\bD}{{\mathbf{D}}}
\nc{\bO}{{\mathbf{O}}}
\nc{\bU}{{\mathbf{U}}}
\nc{\bc}{{\mathbf{c}}}
\nc{\be}{{\mathbf{e}}}
\nc{\bo}{{\mathbf{o}}}
\nc{\br}{{\mathbf{r}}}
\nc{\bM}{{\mathbf{M}}}
\nc{\bA}{{\mathbf{A}}}
\nc{\bK}{{\mathbf{K}}}

\nc{\oS}{\overset{\circ}S{}}
\nc{\oY}{\overset{\circ}Y{}}
\nc{\oX}{\overset{\circ}X{}}
\nc{\of}{\overset{\circ}f{}}
\nc{\oCX}{\overset{\circ}\CX{}}
\nc{\opi}{\overset{\circ}\pi{}}
\nc{\obc}{\overline\bc}

\nc{\fW}{{\mathfrak{W}}}
\nc{\reg}{{\text{\rm reg}}}
\nc{\nilp}{{\text{\rm nilp}}}
\nc{\cG}{{\check{G}}}
\nc{\cB}{{\check{B}}}
\nc{\cg}{{\check{\fg}}}
\nc{\cb}{{\check{\fb}}}
\nc{\cn}{{\check{\fn}}}
\nc{\mer}{{\on{mer}}}
\nc{\Const}{\mathsf{Const}}
\nc{\Whit}{\on{Whit}}
\nc{\KL}{\on{KL}}
\nc{\FS}{\on{FS}}
\nc{\LocSys}{\on{LocSys}}
\nc{\QCoh}{\on{QCoh}}
\nc{\Coh}{\on{Coh}}
\nc{\IndCoh}{{\on{IndCoh}}}
\nc{\Cat}{\on{Cat}}
\nc{\Op}{\on{Op}}
\nc{\Gr}{\on{Gr}}
\nc{\Fl}{\on{Fl}}
\nc{\Rep}{\on{Rep}}
\renc{\mod}{{\on{-mod}}}
\nc{\Conn}{\on{Conn}}
\nc{\unit}{{\mathbf{1}}}

\nc{\Hom}{\on{Hom}}
\nc{\End}{\on{End}}
\nc{\Vect}{\on{Vect}}
\nc{\Av}{\on{Av}}
\nc{\Ind}{\on{Ind}}
\nc{\Spec}{\on{Spec}}
\nc{\KG}{K\backslash G}
\nc{\comult}{{co\text{-}mult}}
\nc{\counit}{{co\text{-}unit}}
\nc{\uHom}{{\underline{\Hom}}}
\nc{\dgSch}{\on{DGSch}}
\nc{\dgindSch}{\on{DGindSch}}
\nc{\indSch}{\on{indSch}}
\nc{\Sch}{\on{Sch}}
\nc{\affdgSch}{\on{DGSch}^{\on{aff}}}
\nc{\affSch}{\on{Sch^{\on{aff}}}}

\nc{\Groupoids}{\on{Grpd}}
\nc{\inftygroup}{\infty\on{-Grpd}}
\nc{\inftyPicgroup}{\infty\on{-PicGrpd}}
\nc{\inftyCat}{\infty\on{-Cat}}
\nc{\StinftyCat}{\on{DGCat}}
\nc{\MoninftyCat}{\infty\on{-Cat}^{Mon}}
\nc{\SymMoninftyCat}{\infty\on{-Cat}^{SymMon}}
\nc{\SymMonStinftyCat}{\on{DGCat}^{SymMon}}
\nc{\MonStinftyCat}{\on{DGCat}^{Mon}}
\nc{\inftystack}{\on{Stk}}
\nc{\inftystackalg}{Stk^{1\text{-}alg}}
\nc{\inftyprestack}{\on{PreStk}}
\nc{\inftydgnearstack}{\on{NearStk}}
\nc{\inftydgstack}{\on{Stk}}
\nc{\inftydgstackalg}{DGStk^{1\text{-}alg}}
\nc{\inftydgprestack}{\on{PreStk}}

\nc{\mmod}{\on{-}{\mathbf{mod}}}
\nc{\wh}{\widehat}

\nc{\nDG}{^{\leq n}\!\on{DG}}

\nc{\Maps}{\on{Maps}}
\nc{\CMaps}{{\mathcal Maps}}
\nc{\bLoc}{\mathbf{Loc}}
\nc{\bGamma}{\mathbf{\Gamma}}
\nc{\bDelta}{\mathbf{\Delta}}
\nc{\dr}{{\on{dR}}}
\nc{\comod}{\on{-comod}}
\nc{\Oblv}{{\mathbf{oblv}}}
\nc{\oblv}{{\on{oblv}}}
\nc{\Iinv}{{\mathbf{inv}}}
\nc{\ccoinv}{{\mathbf{coinv}}}
\nc{\Iind}{{\mathbf{ind}}}
\nc{\coIind}{{\mathbf{coind}}}
\nc{\ind}{{\on{ind}}}
\nc{\coind}{{\on{coind}}}
\nc{\Rrec}{{\mathbf{rec}}}
\nc{\Rres}{{\mathbf{res}}}
\nc{\coRres}{{\mathbf{cores}}}
\nc{\Rreg}{{\mathbf{Reg}}}
\nc{\commod}{\on{-}{\mathbf{comod}}}
\nc{\oCY}{\overset{\circ}{\CY}}
\nc{\osF}{\overset{\circ}{\sF}}
\nc{\osG}{\overset{\circ}{\sG}}
\nc{\oU}{\overset{\circ}{U}}
\nc{\obC}{\overset{\circ}{\bC}}
\nc{\obD}{\overset{\circ}{\bD}}
\nc{\one}{\mathbf{1}}

\begin{document}

\title[Sheaves of categories]{Sheaves of categories and the notion of 1-affineness}

\author{Dennis Gaitsgory}

\address{Department of Mathematics, Harvard University, 1 Oxford street, Cambridge MA, 02138, USA}\

\email{gaitsgde@math.harvard.edu}

\date{\today}

\begin{abstract}
We define the notion of 1-affineness for a prestack, and prove an array of results that establish
1-affineness of certain types of prestacks.
\end{abstract}

\maketitle

\tableofcontents

\section*{Introduction}

\ssec{The notion of sheaf of categories}

\sssec{}

Before we define what we mean by a \emph{sheaf of categories}, let us specify what these are sheaves on:
we will consider sheaves of categories over arbitrary \emph{prestacks}.

\medskip

In this paper we work in the framework of derived algebraic geometry, as developed by J.~Lurie. For a brief summary
of our conventions, the reader is referred to the paper \cite{Stacks}. Throughout this paper we will be working over a
fixed ground field $k$ of characteristic $0$. 

\medskip

Let $\affdgSch$ be the category of affine DG schemes. By definition, a \emph{prestack} is an arbitrary functor \footnote{Technically,
we require our prestacks to be accessible as functors.} of $\infty$-categories
$$\CY:(\affdgSch)^{\on{op}}\to \inftygroup.$$

I.e., a prestack is given by its functor of points on affine DG schemes (with the only condition being of set-theoretic nature,
referred to in the footnote). 

\sssec{}

Informally, a sheaf of categories $\CC$ over a prestack $\CY$ is a functorial assignment for every affine DG scheme $S$, 
mapping to $\CY$, of a DG category $\bGamma(S,\CC)$, which is acted on by the monoidal DG category $\QCoh(S)$ of
quasi-coherent sheaves on $S$. I.e., 
\begin{equation} \label{e:shvcat prel}
(S\to \CY) \in \affdgSch_{/\CY}\rightsquigarrow \bGamma(S,\CC)\in \QCoh(S)\mmod,
\end{equation}
where we denote by 
$\QCoh(S)\mmod$ the $\infty$-category of $\QCoh(S)$-module categories.

\medskip

The assignment \eqref{e:shvcat prel} must be functorial in $S$ in the sense that for a map
$f:S_1\to S_2$ in $\affdgSch_{/\CY}$ we must be given an isomorphism in $\QCoh(S_1)\mmod$:
$$\QCoh(S_1)\underset{\QCoh(S_2)}\otimes \bGamma(S_2,\CC)\to \bGamma(S_1,\CC),$$
together with a homotopy-coherent system of compatibilities for compositions of morphisms.

\medskip

A precise definition of the $\infty$-category $\on{ShvCat}(\CY)$ is given in \secref{ss:defn shv-cat}.

\sssec{}

As is often the case with subjects such as the present one, a natural question to ask is why we should
care about the notion of sheaf of categories, and especially in such generailty.

\medskip

The author was led to the study of this notion by the (still highly conjectural) local
geometric Langlands program. Namely, the object of study of this program is the notion
of category, equipped with an action of the loop group $G\ppart$. 

\medskip

Now, the $\infty$-category of categories acted on by $G\ppart$
is (more or less by definition) the same as $\on{ShvCat}(\CY)$, for 
$$\CY:=B\left(G\ppart_\dr\right),$$
where $B(\CG)$ denotes the classifying prestack of a given group-prestack $\CG$,
and $(-)_\dr$ is the de Rham prestack of a given stack. 

\medskip

It turns out that the case of $B\left(G\ppart_\dr\right)$ contains the complexity of all the examples
considered in this paper combined (algebraic stacks, indschemes, classifying stacks of 
formal groups, de Rham prestacks).

\ssec{Quasi-coherent sheaves on a prestack}

We shall now take a slightly different approach to what a sheaf of categories over a prestack might mean. 

\sssec{}

For every prestack $\CY$ we have the DG category $\QCoh(\CY)$. By definition,
its objects are assigments
$$(S\to \CY) \in \affdgSch_{/\CY}\rightsquigarrow \CF_S\in \QCoh(S),$$
endowed with the data of 
$$f^*(\CF_{S_2})\simeq \CF_{S_1}, \quad (f:S_1\to S_2) \in \affdgSch_{/\CY},$$
together with a homotopy-coherent system of compatibilities for compositions of morphisms.
I.e., informally, a quasi-coherent sheaf on $\CY$ is a compatible family of quasi-coherent sheaves
on affine DG schemes mapping to $\CY$. 

\medskip

The DG category $\QCoh(\CY)$ has a natural (symmetric) monoidal structure given by term-wise
tensor product:
$$(\CF^1\otimes \CF^2)_S:=\CF^1_S\underset{\CO_S}\otimes \CF^2_S.$$

\sssec{}

Can consider the $\infty$-category $\QCoh(\CY)\mmod$ of module categories over $\QCoh(\CY)$.
The goal of this paper is to study the connection between the $\infty$-categories
$$\on{ShvCat}(\CY)  \text{ and } \QCoh(\CY)\mmod.$$

\medskip

The above two $\infty$-categories are tautologically equivalent if $\CY$ is an affine DG scheme. 

\sssec{}

The first observation is that the categories $\on{ShvCat}(\CY)$ and $\QCoh(\CY)\mmod$ are related by
a pair of adjoint functors:

\medskip

Given $\CC\in \on{ShvCat}(\CY)$, we can take the DG category $\bGamma(\CY,\CC)$
of its global sections over $\CY$. It will be naturally acted on by $\QCoh(\CY)$. Thus, we
obtain an object 
$$\bGamma^{\on{enh}}(\CY,\CC)\in \QCoh(\CY)\mmod,$$
and we obtain a functor
$$\bGamma^{\on{enh}}_\CY:=\bGamma^{\on{enh}}(\CY,-):\on{ShvCat}(\CY)\to \QCoh(\CY)\mmod.$$

\medskip

The functor $\bGamma^{\on{enh}}_\CY$ admits a left adjoint, denoted $\bLoc_\CY$, given by tensoring up.
Namely, for $\bC\in \QCoh(\CY)\mmod$ we let $\bLoc_\CY(\bC)$ be the sheaf of categories, whose value
on $S \in \affdgSch_{/\CY}$ is
$$\QCoh(S)\underset{\QCoh(\CY)}\otimes \bC.$$

\sssec{}

We can now give the definition central for this paper: we shall say that a prestack $\CY$ is 1-affine if
the functors $\bGamma^{\on{enh}}_\CY$ and $\bLoc_\CY$ are (mutually inverse) equivalences
of $\infty$-categories.

\medskip

Thus, $\CY$ is 1-affine, if and only if the category $\on{ShvCat}(\CY)$ can be completely recovered
from the monoidal DG category $\QCoh(\CY)$.

\medskip

As was mentioned above, an affine DG scheme is tautologically 1-affine. 

\sssec{}
The origin of the name is the following: let us say that a prestack $\CY$ is \emph{weakly 0-affine}
if the functor
$$\Gamma(\CY,-):\QCoh(\CY)\to \Gamma(\CY,\CO_\CY)\mod$$
is an equivalence of categories. Here $\Gamma(\CY,\CO_\CY)$ is the (DG) algebra of global sections of $\CO_\CY$,
and $\Gamma(\CY,\CO_\CY)\mod$ is DG category of its modules.

\medskip

Tautologically, an affine DG scheme is weakly 0-affine. However, the class of weakly 0-affine
prestacks is much larger than just affine DG schemes. For example, any quasi-affine DG scheme is 
weakly 0-affine. In addition, the algebraic stack $\on{pt}/\BG_a$ is also weakly 0-affine. 

\medskip

The notion of 1-affineness is a higher-categorical analog, where instead of modules over DG algebras, we
consider module categories over monoidal DG categories. 

\ssec{Main results}

This paper aims to determine which prestacks are 1-affine. 

\begin{rem}
Let us say right away that it is ``much easier'' for a prestack to be 1-affine than weakly 0-affine.
We shall see multiple manifestations of this phenomenon below (however, it is \emph{not} true that
every weakly $0$-affine prestack is 1-affine).
\end{rem}

\sssec{}

First, one shows that any (quasi-compact, quasi-separated) DG scheme is 1-affine 
(\thmref{t:alg space}). 

\medskip

Furthermore, we show that algebraic stacks (under some not too restrictive technical conditions) are also
1-affine (\thmref{t:alg}). \footnote{Unfortunately, the only proof of this result that we could come up with
for arbitrary algebraic stacks is rather complicated. On the other hand, a much simpler proof can be
given for algebraic stacks that are global quotients (\thmref{t:quotient}). The core idea of both proofs
is due to J.~Lurie. }

\medskip

One of the technical conditions for 1-affineness of algebraic stacks is that the inertia group
of points be of finite type. This condition turns out to be necessary. Namely, we show (\thmref{t:thick})
that the classifying stack of a group-scheme of infinite type is typically \emph{not} 1-affine. 

\sssec{}

One can wonder whether it is reasonable to expect higher Artin stacks to be 1-affine. Unfortunately,
we did not find a principle that governs the answer:

\medskip

Consider the iterated classifying spaces $B\BG_a$, $B^2\BG_a$, $B^3\BG_a$. We prove (\thmref{t:iterated B})
that they are all 1-affine. However, we also prove that $B^4\BG_a$ is \emph{not} 1-affine.

\sssec{}

Another class of prestacks of interest for us is (DG) indschemes. These turn out \emph{not} to be 1-affine, 
even in the nicest cases, such as $\BA^\infty$.

\sssec{}

A third class of primary interest is prestacks of the form $Z_\dr$, where $Z$ is a scheme of finite type.
We can think of $\on{ShvCat}(Z_\dr)$ as the category of \emph{crystals of categories} over $Z$. 

\medskip

We prove (\thmref{t:main DR}) that $Z_\dr$ is 1-affine. 

\medskip

Note, however, that if $Z$ is not a scheme but an algebraic stack, then $Z_\dr$ is no longer 1-affine. 

\sssec{Methods}

Let us say a few words about what goes into proving that a given class of prestacks is or is not 1-affine. Invariably,
this question reduces to that of whether a certain functor between two very concrete DG categories is \emph{monadic}
(see \secref{sss:intr monadic} for what this means). 

\medskip

Usually, the monadicity of a functor is established using the Barr-Beck-Lurie theorem (\cite[Theorem 6.2.2.5]{Lu2}).
In order to apply this theorem, one needs to check two conditions. One is that the functor in question is
conservative (usually, this is fairly easy). The second condition is that the functor commutes with certain
geometric realizations. This condition is much harder to check in practice, unless our functor happens to
commute with \emph{all} colimits (i.e., is continuous), while the latter is not always the case.  

\medskip

Verifying this second condition constitutes the bulk of the technical work in this paper. Let us emphasize again
that, although our main assertions are initially about \emph{continuous} functors between DG categories
(i.e., functors that commute with all colimits), the core of the proofs involves non-continuous functors. 

\medskip

So, one can say that at the end of the day, the proofs consist of showing that certain colimits commute
with certain limits, i.e., we deal with convergence problems. 
In this sense, what we do in this paper can be called ``functional analysis within homological algebra." 

\ssec{Organization of the paper}

The paper can be loosely divided into three parts. 

\sssec{}

In Part I we give the definitions and discuss some general constructions.

\medskip

In \secref{s:shvcat} we define sheaves of categories, the property of 1-affineness, and discuss
some basic results.

\medskip

In \secref{s:results} we state the main results of this paper pertaining to 1-affineness and non 1-affineness
of certain classes of prestacks. 

\medskip

In \secref{s:functors} we discuss the functors of direct and inverse image of sheaves of categories,
and study how these functors interact with the functors $\bGamma^{\on{enh}}$ and $\bLoc$
mentioned earlier.

\medskip

In \secref{s:formal compl} we show that the property of 1-affineness survives the operation of taking
the formal completion of a prestack along a closed subset. 

\sssec{}

In Part II we consider the question of 1-affineness of algebraic spaces and algebraic stacks.

\medskip

In \secref{s:alg space} we prove that (quasi-compact, quasi-separated) algebraic spaces
are 1-affine. In addition, we single out a class of prestacks (we call them \emph{passable}; this
class includes algebraic stacks satisfying certain technical hypotheses) for which the functor 
$\bGamma^{\on{enh}}$ is fully faithful. 

\medskip

In \secref{s:crit aff} we give several equivalent conditions for a (passable) algebraic stack to be
1-affine. Essentially, these conditions reduce the verification of 1-affineness of a given 
algebraic stack to the question of monadicity of a certain very concrete functor. 

\medskip

In \secref{s:BG} we show that the classifying stack of a (classical) algebraic group of finite type
is 1-affine. The proof is based on the criterion of 1-affineness developed in \secref{s:crit aff}.
The idea of the proof belongs to J.~Lurie. From the 1-affineness of the classifying stack
we (easily) deduce the 1-affineness of algebraic stacks that are global quotients. 

\medskip

Finallly, in \secref{s:stacks} we prove that algebraic stacks (under certain technical hypotheses)
are 1-affine. As was mentioned above, the proof is rather long. It consists of checking the
monadacity of a functor when the conditions of the Barr-Beck-Lurie theorem could
not be checked directly (or, rather, the author did not find a way to do so). 

\sssec{}

In Part III we treat the question of 1-affineness of a host of cases: (DG) indschemes, 
classifying prestacks of general group-prestacks, classifying prestacks of formal groups, 
de Rham prestacks, and other related types of prestacks. 

\medskip

In \secref{s:indsch} we specify a class of (DG) indschemes, for which the functor $\bLoc$ is
fully faithful. This class includes formally smooth indcshemes locally almost of finite type. 
We also show that (DG) indschemes are typically not 1-affine. 

\medskip

In \secref{s:classifying} we study sheaves of categories over prestacks of the form $B\CG$, where
$\CG$ is a group-object in the category of prestacks. We explain how the theory of sheaves of
categories over such prestacks can be viewed as ``higher representation theory," i.e., as the
theory of categories acted on by $\CG$.

\medskip

In \secref{s:formal groups} we show that prestacks of the form $B\CG$, where $\CG$ is a formal group,
which as a formal scheme is isomorphic to $\on{Spf}(k[t_1,...,t_n])$, is 1-affine. 

\medskip

In \secref{s:DR} we study the question of 1-affineness of prestacks of the form $Z_\dr$, where $Z$
is a scheme or algebraic stack.  The proof of 1-affineness in the case of schemes relies on 1-affineness
of formal classifying spaces, developed in the previous section.

\medskip

In \secref{s:inf loop} we study the following problem: we start with a DG scheme $Z$ with a point $z\in Z(k)$, and
consider the (derived) inertia group $\Omega(Z,z)$ (a.k.a. the infinitesimal loop group of $Z$ at $z$). 
We study the connection between sheaves of categories on the prestack $B(\Omega(Z,z))$ and 
sheaves of categories on the formal completion of $Z$ at $z$.

\medskip

In \secref{s:coaffine} we study the question of 1-affineness of iterated classifying prestacks of the form 
$B^k(\BG_a)$, and of classifying prestacks of iterated loop groups $\Omega^k(Z,z)$. 

\sssec{}

This paper contains several appendices, included for the reader's convenience in order to make the exposition
more self-contained.

\medskip

In \secref{s:proof of descent} we reproduce the proof of the result of J.~Lurie that the assigment
$$\CY\rightsquigarrow \on{ShvCat}(\CY)$$
is itself a sheaf in the fppf topology.  

\medskip

In \secref{s:quasi-affine} we reproduce proofs of several statements from \cite{QCoh} pertaining to
the behavior of quasi-affine morphisms from the point of view of tensor products of categories. 

\medskip

In \secref{s:Beck-Chevalley} we review the (monadic and co-monadic) Beck-Chevalley conditions for co-simplicial categories.
These conditions make the totalization of the given co-simplicial category calculable: namely the forgetful
functor of evaluation on $0$-simplices turns out to be monadic (resp., co-monadic), and the corresponding monad
(resp., co-monad) can be described explicitly.

\medskip 

In \secref{s:rigid} we review the notion of rigidity for a monoidal DG category. \footnote{For compactly generated
monoidal DG categories, the condition of rigidity is equivalent to requiring that every compact object admit a left
and right monoidal duals.} This notion turns out to be very
convenient, as it allows for explicit control of the operation of tensor product of module categories over
our monoidal DG category. 

\medskip

In \secref{s:Hopf} we prove a certain basic result about commutative Hopf algebras in symmetric monoidal $\infty$-categories
(its version in ordinary categories is easy to prove by hand, and so is often passed by, without being stated 
explicitly).

\ssec{Conventions}

\sssec{}  \label{sss:intr monadic}

This paper relies on the theory of $\infty$-categories as developed by J.~Lurie in \cite{Lu1} and \cite{Lu2}. 
By a slight abuse of terminology we shall sometimes say ``category", when we actually mean $\infty$-category.

\sssec{}

The following terminology is used throughout the paper. If $\sM$ is a monad acting on an $\infty$-category $\bC$,
we let $\sM\mod(\bC)$ denote the category of $\sM$-modules (sometimes also called $\sM$-algebras) in $\bC$.
We let
$$\on{ind}_\sM:\bC\rightleftarrows \sM\mod(\bC):\on{oblv}_\sM$$
the resulting adjoint pair of functors (``oblv" stands for the forgetful functor, and ``ind" for the induction functor). 

\medskip 

Let
$$\bC\leftarrow \bD:\sG$$
be a functor between $\infty$-categories. We shall say that $\sG$ is \emph{monadic} if it admits
a left adjoint, denoted $\sF$, and when we view the composition $\sG\circ \sF$ as a monad
acting on $\bC$, the resulting functor
$$(\sG\circ \sF)\mod(\bC)\leftarrow \bD:\sG^{\on{enh}}$$
is an equivalence. 

\medskip

Replacing ``left" by ``right", we obtain the notion of \emph{co-monadic} functor. 

\sssec{}

Our conventions regarding DG categories follow those adopted in \cite{DGCat}. 

\medskip

We let $\Vect$
denote the DG category of chain complexes of $k$-vector spaces. \footnote{The reader can substitute the 
notion of DG category by a better documented notion of presentable stable $\infty$-category, tensored
over $\Vect$.} In this paper all DG categories will be assumed presentable (in particular, cocomplete,
i.e., closed under arbitrary direct sums). 

\medskip

We let $\StinftyCat$ denote the $\infty$-category of
DG categories and accessible exact functors. We let $\StinftyCat_{\on{cont}}$ denote the category with
the same objects, but where we restrict 1-morphisms to be continuous (i.e., commuting with all
direct sums, equivalently, with all colimits). 

\medskip

In multiple places of the paper we will use the result of \cite[Corollary 5.5.3.3]{Lu1} that says the colimit
of a diagram in $\StinftyCat_{\on{cont}}$, can be computed as a DG category, as the limit in 
$\StinftyCat$ of the diagram obtained by passage to right adjoint functors. For a sktech of the
proof of this result the reader is referred to \cite[Lemma  1.3.3]{DGCat}.

\sssec{}

The $\infty$-category $\StinftyCat_{\on{cont}}$ carries a natural symmetric monoidal structure given
by tensor product of DG categories. (We emphasize that we live in the world of cocomplete DG categories
and continuous functors.) 

\medskip

If $\bO$ is an algebra object in $\StinftyCat_{\on{cont}}$, i.e., a monoidal DG category, we let $\bO\mmod$
denote the category of $\bO$-modules in $\StinftyCat_{\on{cont}}$, i.e., the $\infty$-category of $\bO$-module
categories.  

\medskip

In general, throughout the paper, we use boldface symbols for ``higher" objects and functors. E.g., if $\CG$ is an
affine DG group-scheme, we use $\on{inv}^\CG$ to denote the functor 
$$\Rep(\CG)\to \Vect,$$
of invariants on the category of $\CG$-representations, and we use $\Iinv^\CG$ to denote the functor 
$$\CG\mmod\to \StinftyCat_{\on{cont}}$$
that sends a DG category acted on by $\CG$ to the category of $\CG$-equivariant objects.

\medskip

For a pair of DG categories $\bD_1,\bD_2$, we let $\uHom(\bD_1,\bD_2)$ denote their ``internal Hom",
i.e., the DG category of continuous functors $\bD_1\to \bD_2$. Similarly, for $\bD_1,\bD_2\in \bO\mmod$
we will use the notation $\uHom_{\bO}(\bD_1,\bD_2)$ for the DG category of functors compatible
with the $\bO$-module structure.

\sssec{}

Our conventions regarding derived algebraic geometry follow those of \cite{Stacks}. We let $\affdgSch$
denote the $\infty$-category of affine DG schemes, which is by definition the opposite category to that
of connective $k$-algebras.

\medskip

We let $\on{PreStk}$ denote the $\infty$-category of all prestacks, i.e., the category of accessible functors
$$(\affdgSch)^{\on{op}}\to \inftygroup,$$
where $\inftygroup$ is the $\infty$-ategory of $\infty$-groupoids (a.k.a., spaces). 

\medskip

In the main body of the paper, we will need the notion of fppf morphism between DG schemes, for
which the reader is referred to \cite[Sect. 2.1]{Stacks}\footnote{In this paper we use the more common
``fppf" rather than ``fpppf."}. 

\medskip

We will need the notion of what it means for a DG scheme or Artin stack to be classical
(resp., eventually coconnective). By definition,
an affine DG scheme is classical (resp., eventually coconnective) if its DG ring of functions 
has no (resp., finitely many) non-zero cohomology groups. An Artin stack is classical 
(resp., eventually coconnective)
if it admits an fppf cover by an affine DG scheme which is classical (resp., eventually coconnective). 
For further details see \cite[Sects. 1.1, 2.4 and 4.6]{Stacks}.

\medskip

We let $\on{PreStk}_{\on{laft}}$ denote the full subcategory of $\on{PreStk}$ formed by prestacks
that are \emph{locally almost of finite type}, see \cite[Sect. 1.3.9]{Stacks}.

\sssec{}

In some proofs, we will need to use the category of \emph{ind-coherent sheaves}, developed
in \cite{IndCoh}. This category is defined on prestacks that belong to $\on{PreStk}_{\on{laft}}$. 

\ssec{Acknowledgements}

The problems such as those addressed in this paper were brought to the author's awareness by J.~Lurie,
so this paper can be regarded as a research project carried out under his guidance, and the
influence of his ideas is evident everywhere in the text.  

\medskip

The author is grateful to V.~Drinfeld for collaboration on \cite{DrGa}, which supplied the key ideas for
our main result on algebraic stacks, \thmref{t:alg}.

\medskip

The author is grateful to N.~Rozenblyum and S.~Raskin for numerous helpful discussions of various topics
related to the contents of this paper. 

\medskip

The author is grateful to D.~Beraldo for posing the question of 1-affineneness of de Rham prestacks, which
prompted the writing of this paper.

\medskip

The author is supported by NSF grant DMS-1063470.

\bigskip

\centerline{\bf Part I: Generalities}

\section{Quasi-coherent sheaves of categories}  \label{s:shvcat}

\ssec{Definition of a quasi-coherent sheaf of categories} \label{ss:defn shv-cat}

\sssec{}

Consider the functor
$$\on{ShvCat}_{\affdgSch}:(\affdgSch)^{\on{op}}\to \inftyCat,\quad S\mapsto \QCoh(S)\mmod,$$
that assigns to an affine DG scheme the $\infty$-category of module categories
over the monoidal DG category $\QCoh(S)$. 

\medskip

Let 
$$\on{ShvCat}_{\on{PreStk}}:(\on{PreStk})^{\on{op}}\to  \inftyCat$$
be the right Kan extension of $\on{ShvCat}_{\affdgSch}$ along the Yoneda embedding
$$(\affdgSch)^{\on{op}}\hookrightarrow \on{PreStk}^{\on{op}}.$$

\medskip

For a prestack $\CY$, we let 
$$\on{ShvCat}(\CY)\in \inftyCat$$
denote the value of $\on{ShvCat}_{\on{PreStk}}$ on $\CY\in \on{PreStk}$.

\medskip

We shall refer to objects of $\on{ShvCat}(\CY)$ as a ``quasi-coherent sheaves of DG categories on $\CY$."

\sssec{}

In other words, for $\CY\in \on{PreStk}$, an object $\CC\in \on{ShvCat}(\CY)$ is an assignment 
$$S\in \affdgSch_{/\CY}\rightsquigarrow \bGamma(S,\CC)\in \QCoh(S)\mmod,$$
and for an arrow $g:S_1\to S_2$ in $\affdgSch_{/\CY}$ of an equivalence
$$\QCoh(S_1)\underset{\QCoh(S_2)}\otimes \bGamma(S_2,\CC)\simeq \bGamma(S_1,\CC),$$
along with a homotopy-coherent system of compatibilities.

\medskip

Morphisms between sheaves of categories are defined naturally. 

\medskip

From the definition of $\on{ShvCat}(-)$ as the right Kan extension we obtain: 

\begin{lem}
The functor $\on{ShvCat}(-)$ takes colimits in $\on{PreStk}$ to limits in $\inftyCat$.
\end{lem}

\sssec{}

The basic example of an object of $\on{ShvCat}(\CY)$ is $\QCoh_{/\CY}$, whose value on
$S\in  \affdgSch_{/\CY}$ is $\QCoh(S)$. 

\medskip

The category $\on{ShvCat}(\CY)$ carries a symmetric monoidal structure given by component-wise
tensor product, and $\QCoh_{/\CY}$ is its unit object. 

\sssec{}

The category $\on{ShvCat}(\CY)$ contains colimits that are computed value-wise.

\medskip

The category $\on{ShvCat}(\CY)$ contains limits, which are computed by
$$\bGamma\left(S,\underset{i}{\underset{\longleftarrow}{lim}}\,(\CC_i)\right)\simeq \underset{i}{\underset{\longleftarrow}{lim}}\,\bGamma(S,\CC_i).$$

Indeed, this follows from the fact that for a morphism $f:S_1\to S_2$ in $\affdgSch_{/\CY}$, the functor
$$ \QCoh(S_1)\underset{\QCoh(S_2)}\otimes -:\QCoh(S_2)\mmod\to \QCoh(S_1)\mmod$$
commutes with \emph{limits}, which in turn follows from the fact that the category
$\QCoh(S_1)$ is dualizable as an object of $\QCoh(S_2)\mmod$, see \lemref{l:dualizable in rigid}. 

\ssec{Global sections}

\sssec{}

For a given $\CY$ and $\CC\in \on{ShvCat}(\CY)$, we can right-Kan-extend the functor 
$$\bGamma(-,\CC):(\affdgSch_{/\CY})^{\on{op}}\to \StinftyCat_{\on{cont}}$$ to a functor
$$(\on{PreStk}_{/\CY})^{\on{op}}\to \StinftyCat_{\on{cont}}; \quad \CZ\mapsto \bGamma(\CZ,\CC).$$
I.e.,
$$\bGamma(\CZ,\CC):=\underset{S\in \affdgSch_{/\CZ}}{\underset{\longleftarrow}{lim}}\, \bGamma(S,\CC).$$ 

\medskip

For example, it is clear that
$$\bGamma(\CZ,\QCoh_{/\CY})\simeq \QCoh(\CZ).$$

\medskip

In particular, we obtain a DG category $\bGamma(\CY,\CC)$. 

\sssec{}

It is clear that the functor 
$$\CZ\mapsto \bGamma(\CZ,\CC)$$
takes colimits in $\on{PreStk}_{/\CY}$ to limits in $\StinftyCat_{\on{cont}}$. 

\sssec{}

The functor 
$$\bGamma(\CZ,-):\on{ShvCat}(\CY)\to  \StinftyCat_{\on{cont}}$$
is lax symmetric monoidal. 

\medskip

In particular, we obtain that it natually upgrades to a functor
$$\on{ShvCat}(\CY)\to  \bGamma(\CZ,\QCoh_{/\CY})\mmod\simeq \QCoh(\CZ)\mmod.$$

We shall denote the resulting functor
$$\on{ShvCat}(\CY)\to \QCoh(\CZ)\mmod$$
by $\bGamma^{\on{enh}}(\CZ,-)$. 

\medskip 

When $\CZ=\CY$, we shall sometimes write 
$$\bGamma^{\on{enh}}_\CY:\on{ShvCat}(\CY)\to \QCoh(\CY)\mmod$$
instead of $\bGamma^{\on{enh}}(\CY,-)$.

\ssec{Posing the problem}

\sssec{}

We note that the functor
$$\bGamma^{\on{enh}}_\CY:\on{ShvCat}(\CY)\to \QCoh(\CY)\mmod$$
admits a left adjoint; we denote it by $\bLoc_\CY$.

\medskip

Namely, for $\bC\in \QCoh(\CY)\mmod$ we have
$$\bGamma(S,\bLoc_\CY(\bC))=\QCoh(S)\underset{\QCoh(\CY)}\otimes \bC,\quad S\in \affdgSch_{/\CY}.$$

\medskip

It clear from the construction that the functor 
$$\bLoc_\CY:\QCoh(\CY)\mmod\to \on{ShvCat}(\CY)$$
is symmetric monoidal. 

\sssec{}

The questions that we want to address in this paper are the following: 

\begin{quest}  \label{q:main} \hfill

\begin{enumerate}

\item

Under what conditions, for $\CC\in \on{ShvCat}(\CY)$ is the co-unit map
\begin{equation} \label{e:counit}
\bLoc_\CY\left(\bGamma^{\on{enh}}(\CY,\CC)\right)\to \CC
\end{equation}
an equivalence? 

\medskip

\item Under what conditions on $\CY$ is \eqref{e:counit} an equivalence for \emph{all}
$\CC\in \on{ShvCat}(\CY)$? I.e., when is $\bGamma^{\on{enh}}_\CY$ fully faithful? 

\medskip

\item Under what conditions, for $\bC\in  \QCoh(\CY)\mmod$ is the unit map
\begin{equation} \label{e:unit}
\bC\to \bGamma^{\on{enh}}(\CY,\bLoc_\CY(\bC))
\end{equation}
an equivalence? 

\medskip

\item Under what conditions on $\CY$ is \eqref{e:unit} an equivalence for \emph{all}
$\bC \in \QCoh(\CY)\mmod$? I.e., when is $\bLoc_\CY$ fully faithful? 

\end{enumerate}

\end{quest}

\sssec{}

In some cases, the answer is very easy: 

\begin{lem}
Suppose that $\bC$ is dualizable as an object of $\QCoh(\CY)\mmod$. Then the adjunction map
$$\bC\to \bGamma^{\on{enh}}(\CY,\bLoc_\CY(\bC))$$
is an equivalence.
\end{lem}

\begin{proof}
This follows from the fact that for $\bC\in \QCoh(\CY)\mmod$ dualizable, the functor
$$-\underset{\QCoh(\CY)}\otimes \bC:\QCoh(\CY)\mmod\to \StinftyCat_{\on{cont}}$$
commutes with limits. Indeed,
\begin{multline*} 
\bGamma(\CY,\bLoc_\CY(\bC))\simeq \underset{S\in \affdgSch_{/\CY}}{\underset{\longleftarrow}{lim}}\, (\QCoh(S)\underset{\QCoh(\CY)}\otimes \bC)\simeq \\
\left(\underset{S\in \affdgSch_{/\CY}}{\underset{\longleftarrow}{lim}}\, \QCoh(S)\right) \underset{\QCoh(\CY)}\otimes \bC\simeq
\QCoh(\CY) \underset{\QCoh(\CY)}\otimes \bC\simeq \bC.
\end{multline*}

\end{proof}

\sssec{}

We give the following definition: 

\begin{defn}
We shall say that $\CY$ is 1-affine if the functors $\bGamma_\CY^{\on{enh}}$ and $\bLoc_\CY$ are
mutually inverse equivalences.
\end{defn}

The main results of this paper will amount to saying that certain classes of prestacks are (or are not)
1-affine. 

\ssec{Dualizability and compact generation}

\sssec{}

Recall that in any symmetric monoidal $\infty$-category we can talk about the property of an
object to be dualizable. 

\medskip

Since the functor $\bLoc_\CY$ is symmetric monoidal, it automatically sends dualizable objects in  
$\QCoh(\CY)\mmod$ to dualizable objects in $\on{ShvCat}(\CY)$.

\medskip

The following is also tautological:

\begin{lem}
If $\CY$ is 1-affine, then the lax symmetric monoidal structure on $\bGamma^{\on{enh}}_\CY$ is strict
(i.e., non-lax).
\end{lem}

From here we obtain:

\begin{cor}
If $\CY$ is 1-affine, and $\CC\in \on{ShvCat}(\CY)$ is dualiable, then $\bGamma(\CY,\CC)$ is dualizable
as an object of $\QCoh(\CY)\mmod$.
\end{cor}

\sssec{}

Let us make the notion of being dualizable as an object of $\on{ShvCat}(\CY)$ more explicit:

\begin{prop}
An object $\CC\in \on{ShvCat}(\CY)$ is dualizable if and only if for every $S\in \affdgSch_{/\CY}$,
the category $\bGamma(S,\CC)$ is dualizable as a plain DG category.
\end{prop}

\begin{proof}
The proof follows from the combination of the following two lemmas:

\begin{lem}[Lurie] \label{l:dualizable in limit}
Let a symmetric monoidal category $\bO$ be equal to the limit
$$\underset{\alpha}{\underset{\longleftarrow}{lim}}\, \bO_\alpha$$
of a diagram $\alpha\mapsto \bO_\alpha$ of symmetric monoidal categories. Then an object
$\bo\in \bO$ is dualizable if and only if its projection $\bo_\alpha\in \bO_\alpha$ is dualizable
for every index $\alpha$.
\end{lem}

For the next lemma recall the notion of \emph{rigid} monoidal DG category, see \secref{ss:rigidity}.
For example, for an affine DG scheme $S$, the monoidal DG category $\QCoh(S)$ is rigid (this
is the trivial case of \lemref{l:pass rigid}). 

\medskip

We have (see \secref{sss:dualizable in rigid}):

\begin{lem} \label{l:dualizable in rigid}
Let $\bA$ is a symmetric monoidal DG category, which is rigid as a monoidal DG category. Then $\bC\in \bA\mmod$ is dualizable
as an object of the symmetric monoidal category $\bA\mmod$ if and only if $\bC$ is dualizable as a plain DG category
\emph{(}i.e., as an object of $\StinftyCat_{\on{cont}}$\emph{)}. 
\end{lem}

\end{proof}

We also notice the following corollary of \lemref{l:dualizable in rigid}:

\begin{cor} \label{c:Loc commutes with limits rigid}
Let $\CY$ be such that $\QCoh(\CY)$ is rigid. Then:

\smallskip

\noindent{\em(a)} An object $\bC\in \QCoh(\CY)\mmod$ is dualizable if and only if
it is dualizable as a plain category.

\smallskip

\noindent{\em(b)} Then the functor $\bLoc_\CY$ commutes with limits. 

\end{cor}

\sssec{}

One can also ask the following questions:

\begin{quest}  \label{q:compact} \hfill

\begin{enumerate}

\item
Suppose that $\CC\in \on{ShvCat}(\CY)$ is such that for all $S\in \affdgSch_{/\CY}$, the category
$\bGamma(S,\CC)$ is compactly generated. When can we guarantee that $\bGamma(\CY,\CC)$ is compactly generated
as a plain category? 

\medskip

\item 
Let $\bC\in \QCoh(\CY)\mmod$ be compactly generated as a plain DG category. When can we guarantee 
that $\QCoh(S)\underset{\QCoh(\CY)}\otimes \bC$ is compactly generated for any $S\in \affdgSch_{/\CY}$. 

\end{enumerate}

\end{quest}

These questions appear to be more subtle. For example, to the best of our knowledge, it is not known
whether the category $\QCoh(\CY)$ is compactly generated when $\CY$ is an algebraic stack (when 
$\CY$ is neither smooth nor a global quotient).

\ssec{Descent}

Before we proceed to the discussion of main results of this paper, let us remark that the questions such as
those in Question \ref{q:main} are insensitive to fppf sheafification: 

\sssec{}

First, we recall that following result, which is essentially established in \cite[Theorem 5.4]{Lu3}:

\medskip

\begin{thm} \label{t:descent 1}
Let $Y$ be an affine DG scheme, and let $\bC$ be an object $\QCoh(Y)\mmod$. Then the functor
$$(\affdgSch_{/Y})^{\on{op}}\to \StinftyCat_{\on{cont}},  \quad S\rightsquigarrow \bGamma(S,\bLoc_Y(\bC))=\QCoh(S)\underset{\QCoh(Y)}\otimes \bC$$
satisfies fppf descent.
\end{thm}

This theorem is proved in \cite{Lu3} when instead of the fppf topology one considers the \'etale topology. The
remanining step is easy (and well-known), since 

\medskip

\hskip3cm``fppf descent"=``Nisnevich descent"+``finite flat descent." 

\medskip

For the sake of completeness,
we will prove \thmref{t:descent 1} in Appendix \ref{s:proof of descent}.

\sssec{}

As a formal consequence, we obtain:

\begin{cor}
For a prestack $\CY$ and $\CC\in \on{ShvCat}(\CY)$, the functor
$$(\affdgSch_{/\CY})^{\on{op}}\to \StinftyCat_{\on{cont}},  \quad S\rightsquigarrow \bGamma(S,\CC)$$
satisfies fppf descent.
\end{cor}

From here:

\begin{cor} \label{c:Cech descent}
Let $\CC$ be an object of $\on{ShvCat}(\CY)$.  

\smallskip

\noindent{\em(a)}
If $\CZ\to \CW$ is an fppf surjection in $\on{PreStk}_{/\CY}$, then the pullback functor
$$\bGamma(\CW,\CC) \to \on{Tot}(\bGamma(\CZ^\bullet/\CW,\CC))$$
is an equivalence, where $\CZ^\bullet/\CW$ is the \v{C}ech nerve of
the map $\CZ\to \CW$.

\smallskip

\noindent{\em(b)} For $\CZ\in \on{PreStk}_{/\CY}$, the pullback functor
$$\bGamma(L_{/\CY}(\CZ),\CC)\to \bGamma(\CZ,\CC)$$
is an equivalence, where $L_{/\CY}(-)$ is fppf sheafification in the category
$\on{PreStk}_{/\CY}$.

\smallskip

\noindent{\em(b')} If $\CY$ is an fppf stack, then for $\CZ\in \on{PreStk}_{/\CY}$, the pullback functor
$$\bGamma(L(\CZ),\CC)\to \bGamma(\CZ,\CC)$$
is an equivalence, where $L(-)$ is fppf sheafification in the category
$\on{PreStk}$.

\end{cor}

\sssec{}

Next, from \thmref{t:descent 1}, one formally deduces the next result (this is \cite[Theorem 5.13]{Lu3}):

\begin{thm} \label{t:descent 2}
The functor $\on{ShvCat}_{\on{PreStk}}:(\on{PreStk})^{\on{op}}\to  \inftyCat$
satisfies fppf descent.
\end{thm}

For the reader's convenience, we will supply the derivation \thmref{t:descent 1} $\Rightarrow$ \thmref{t:descent 2} in
Appendix \ref{s:proof of descent}.

\medskip

As a formal consequence of \thmref{t:descent 2}, we obtain:

\begin{cor}  \hfill  \label{c:shv via Cech}

\smallskip

\noindent{\em(a)}
If $\CZ\to \CY$ is an fppf surjection in $\on{PreStk}$, the pullback functor
$$\on{ShvCat}(\CY)\to \on{Tot}(\on{ShvCat}(\CZ^\bullet/\CY))$$
is an equivalence, where $\CZ^\bullet/\CY$ is the \v{C}ech nerve of
the map $\CZ\to \CY$.

\smallskip

\noindent{\em(b)} For $\CY\in \on{PreStk}$, the pullback functor
$$\on{ShvCat}(L(\CY))\to \on{ShvCat}(\CY)$$
is an equivalence, where $L(-)$ is fppf sheafification. 

\end{cor}

\section{Statements of the results}  \label{s:results}

\ssec{Algebraic spaces and schemes}

The following will not be difficult (see \secref{ss:alg space}):

\begin{thm} \label{t:alg space}
Let $\CY$ be a quasi-compact quasi-separated algebraic space. Then $\CY$ is 1-affine.
\end{thm}

\ssec{Algebraic stacks}  \label{ss:alg stacks}

Our conventions regarding algebraic stacks follow those of \cite[Sect. 1.3.3]{DrGa}. In particular,
we assume that the diagonal morphism is representable, quasi-compact and quasi-separated. 

\sssec{}

In \secref{s:BG} we will prove:

\begin{thm} \label{t:BG}
The stack $\on{pt}/G$, where $G$ is a classical affine algebraic group of finite type, is 1-affine.
\end{thm}

The assumption that the group $G$ be of finite type is important. Namely, in \secref{ss:pro infty} we will prove:

\begin{thm} \label{t:thick}
The stack $\on{pt}/G$ for $G=\underset{n}{\underset{\longleftarrow}{lim}}\, (\BG_a)^{\times n}$
is \emph{not} 1-affine.
\end{thm}

In \secref{sss:global quotient}, from \thmref{t:BG} we will deduce:

\begin{thm} \label{t:quotient}
An algebraic stack that can be realized as $Z/G$, where $Z$ is a quasi-compact quasi-separated algebraic space
and $G$ is a classical affine algebraic group of finite type,  is $1$-affine.
\end{thm}

\sssec{}

Finally, in \secref{s:stacks} we will prove:

\begin{thm} \label{t:alg}
An eventually coconnective quasi-compact algebraic stack locally almost of finite type
with an affine diagonal is 1-affine.
\end{thm}

We conjecture that in \thmref{t:alg}, the assumption that $\CY$
be eventually coconnective is superfluous. 

\ssec{Formal completions}

In \secref{s:formal compl} we will prove:

\begin{thm} \label{t:main formal} 
Suppose that $\CY$ is obtained as a formal completion \footnote{The definition of formal completion will
be recalled in \secref{sss:formal compl}.} of a $1$-affine prestack
along a closed subfunctor such that the embedding of its complement is
quasi-compact. Then $\CY$ is $1$-affine. 
\end{thm}

In particular, combining with \thmref{t:alg space}, we obtain:

\begin{cor}
Let $\CY$ be the formal completion of a quasi-compact quasi-separated algebraic space, 
along a closed subset whose complement is quasi-compact. Then $\CY$ is $1$-affine. 
\end{cor}

\ssec{(DG) ind-schemes}

We refer the reader to \cite{IndSch} for our conventions regarding DG indschemes. In particular, we will need the notions
of a \emph{weakly $\aleph_0$} DG indscheme (see \cite[1.4.11]{IndSch}) and of \emph{formally smooth}
DG indscheme (see \cite[8.1.3]{IndSch}).

\sssec{}

In \secref{ss:indsch} we will prove:

\begin{thm}  \label{t:indsch}
Let $\CY$ be a weakly $\aleph_0$ formally smooth 
DG indscheme localy almost of finite type. Then the functor $\bLoc_\CY$ is fully faithful.
\end{thm}

From here in \secref{sss:loop group} we will deduce:

\begin{thm} \label{t:loop group}
Let $G$ be a classical affine algebraic group of finite type. Then for the DG indscheme $G\ppart$,
the functor $\bLoc_{G\ppart}$ is fully faithful.
\end{thm}

\sssec{}

We should note that even the ``nicest" DG indschemes are typically \emph{not} 1-affine. As a manifestation of this,
in \secref{ss:A-infty} we will prove:

\begin{thm} \label{t:A-infty}
Let $\CY=\BA^\infty:=\underset{n}{\underset{\longrightarrow}{colim}}\, \BA^n$. Then $\CY$ is \emph{not} $1$-affine.
\end{thm}

\sssec{}

We do not know whether the functor $\bLoc_\CY$ is fully faithful for more general DG indschemes. For example,
we do not know this in the example of
$$\CY:=\on{pt}\underset{\BA^\infty}\times \on{pt}.$$

\ssec{Classifying prestacks}

\sssec{}

Let $\CG$ be a group-object in $\on{PreStk}$. We let $B^\bullet\CG$ denote the standard simplicial object
of $\on{PreStk}$ associated with $\CG$, i.e., the usual simplicial model for the classifying space.

\medskip

We let $B\CG$ denote the geometric realization,
$$B\CG:=|B^\bullet\CG|\in \on{PreStk}.$$

We remark that if $G$ is an algebraic group, then the algebraic stack $\on{pt}/G$, mentioned earlier, is
by definition the fppf sheafification of $BG$.

\medskip

If $\CG$ is such that $\bLoc_{\CG}$ is fully faithful, the category $\on{ShvCat}(B\CG)$ can be described more explicitly, see \secref{ss:groups}.

\begin{rem} \label{r:BG}
Note that \thmref{t:BG} (combined with \corref{c:shv via Cech}(b)) implies that if $G$ is a classical an affine algebraic
group of finite type, then $BG$ is 1-affine. Note also that according to \thmref{t:thick}, if $G$ is an affine
group-scheme of infinite type, then $BG$ is typically not 1-affine. Below we will discuss several
more cases when $B\CG$ is (or is not) 1-affine.
\end{rem} 

\sssec{}

In \secref{ss:formal groups} we will prove:

\begin{thm} \label{t:main formal groups}
Let $\CG$ be a group-object in $\on{PreStk}$, which as a prestack is a weakly $\aleph_0$ formally smooth DG indscheme 
locally almost of finite type with $({}^{cl}\CG)_{red}=\on{pt}$. In this case: 

\smallskip

\noindent{\em(a)} The functor $\bLoc_{B\CG}$ is fully faithful.

\smallskip

\noindent{\em(b)} The prestack $B\CG$ is 1-affine if and only if the tangent space of $\CG$ at the origin is 
finite-dimensional.
\end{thm}

In addition, in \secref{ss:HCh}, we will prove:

\begin{thm} \label{t:HCh}
Let $G$ be a classical affine algebraic group of finite type, and $H\subset G$ a subgroup. Let $\CG$ be
the formal completion of $G$ along $H$. Then $B\CG$ is 1-affine.
\end{thm}

\sssec{}

Let $0\neq V\in \Vect^\heartsuit$ be finite-dimensional, regarded as a commutative (i.e., $E_\infty$) group-object
of $\on{PreStk}$.   

\medskip

Recall that by Remark \ref{r:BG}, the prestack $BV$ is 1-affine. As another series of examples of group-objects $\CG$ for 
which $B\CG$ is (or is not) 1-affine, in \secref{s:coaffine} we will prove:

\begin{thm}  \label{t:iterated B} \hfill

\smallskip

\noindent{\em(a)} The prestack $B^2(V)$ is 1-affine.

\smallskip

\noindent{\em(b)} The prestack $B^3(V)$ is 1-affine.

\smallskip

\noindent{\em(c)} The prestack $B^4(V)$ is \emph{not} 1-affine.

\smallskip

\noindent{\em(d)} The prestack $B^2(V^\wedge_0)$ is \emph{not} 1-affine, where 
$V^\wedge_0$ is the completion of $V$ at the origin, regarded as a commutative 
group-object of $\on{PreStk}$. 

\end{thm}

\sssec{}

Let now $\CG$ be a group-object of $\affdgSch_{\on{aft}}$. (Note that the fppf sheafification of $B\CG$
is not an algebraic stack, so the discussion in \secref{ss:alg stacks} is not applicable to it.)

\medskip

We propose the following conjecture: 

\begin{conj} \label{conj:BG der} 
The classifying prestack $B\CG$ is 1-affine. 
\end{conj}

As a piece of evidence toward this conjecture, in \secref{ss:vector} we will prove:

\begin{thm} \label{t:vector}
Let $V\in \Vect^\heartsuit$ be finite-dimensional, and for $n\in \BZ^+$ let us view 
$\CG=\Spec(\on{Sym}(V[n]))$ as a group-object of $\affdgSch_{\on{aft}}$. Then
$B\CG$ is 1-affine.
\end{thm}

\sssec{}  \label{sss:quotient by A-infty}

Finally, let us take 
$$\CG=\underset{n}{\underset{\longrightarrow}{colim}}\, (\BG_a)^{\times n}.$$

In \secref{sss:quotient by A-infty proof} we will show that $B\CG$ is not 1-affine. In this example, 
the functor $\bGamma^{\on{enh}}_{B\CG}$ fails to be fully faithful. We do not
know whether $\bLoc_{B\CG}$ is fully faithful. 

\ssec{De Rham prestacks}

\sssec{}

Let $\CZ$ be a prestack. Recall that the prestack $\CZ_\dr$ is defined by
$$\Maps(S,\CZ_\dr):=\Maps(({}^{cl}S)_{red},\CZ).$$

\sssec{}

In \secref{ss:DR} we will prove:

\begin{thm} \label{t:main DR}  
Let $\CY$ be of the form $\CZ_\dr$, where $\CZ$ is an indscheme \footnote{We say ``indscheme" instead of ``DG indscheme",
because for a prestack $\CZ$, the prestack $\CZ_\dr$ only depends on the underlying classical prestack.} locally
of finite type.  Then $\CY$ is $1$-affine. 
\end{thm}

\sssec{}

However, in \secref{ss:DR stack} we will show:

\begin{prop} \label{p:dr stack}
Let $\CY=\CZ_\dr$, where $\CZ=\on{pt}/\BG_a$. Then $\CY$ is \emph{not} 1-affine.
\end{prop}

Hence, if $\CY=\CZ_\dr$, where $\CZ$ is quasi-compact algebraic stack locally 
of finite type with an affine diagonal, it is not in general true that $\CY$ is 1-affine. 

\medskip

However, we propose:

\begin{conj}  \label{conj:dr stack}
Let $\CZ$ is quasi-compact algebraic stack locally of finite type with an affine diagonal. Then the functor 
$\bLoc_{\CZ_\dr}$ is fully faithful. 
\end{conj}

\medskip

In fact, we conjecture that in the situation of \conjref{conj:dr stack} it should be possible
to describe the essential image of the functor $\bLoc_{\CZ_\dr}$ in terms of the groups
of automorphisms of geometric points of $\CZ$. 

\section{Direct and inverse images for sheaves of categories}  \label{s:functors}

\ssec{Definition of functors}

Let $f:\CY_1\to \CY_2$ be a morphism between prestacks.

\sssec{}

The monoidal functor
$$f^*:\QCoh(\CY_2)\to \QCoh(\CY_1)$$
defines a forgetful functor
$$\Rres_f:\QCoh(\CY_1)\mmod\to \QCoh(\CY_2)\mmod.$$

It has a left adjoint, denoted $\Iind_f$, given by
$$\bC_2\mapsto \QCoh(\CY_1)\underset{\QCoh(\CY_2)}\otimes \bC_2.$$

\sssec{}

We also have a tautological functor
$$\coRres_f:\on{ShvCat}(\CY_2)\to \on{ShvCat}(\CY_1).$$

Namely, for $\CC_2\in \on{ShvCat}(\CY_2)$, we restrict the assignment $S\mapsto \bGamma(S,\CC_2)$
from $\affdgSch_{/\CY_2}$ to $\affdgSch_{/\CY_1}$. 

\medskip

We shall sometimes denote this functor by
$$\CC_2\in \on{ShvCat}(\CY_2)\mapsto \CC_2|_{\CY_1}\in \on{ShvCat}(\CY_1).$$

\sssec{}  \label{sss:dir and inv}

We claim that the functor $\coRres_f$ admits a right adjoint (to be denoted $\coIind_f$).
Namely, for $\CC_1\in \on{ShvCat}(\CY_2)$ and $S\in \affdgSch_{/\CY_2}$ we set
$$\bGamma(S,\coIind_f(\CC_1)):=\bGamma(S\underset{\CY_2}\times \CY_1,\CC_1).$$

The fact that this is indeed a quasi-coherent sheaf of categories follows from the next lemma:

\begin{lem}
For a map $S'\to S$ in $\affdgSch_{/\CY_2}$, the functor
$$\QCoh(S')\underset{\QCoh(S)}\otimes \bGamma(S\underset{\CY_2}\times \CY_1,\CC_1)\to 
\bGamma(S'\underset{\CY_2}\times \CY_1,\CC_1)$$
is an equivalence.
\end{lem}

\begin{proof}

By definition, we calculate $\bGamma(S\underset{\CY_2}\times \CY_1,\CC_1)$ as 
$$\underset{T\in \affdgSch_{/S\underset{\CY_2}\times\CY_1}}{\underset{\longleftarrow}{lim}}\, \bGamma(T,\CC_1).$$

\medskip

Since $\QCoh(S')$ is dualizable as a $\QCoh(S)$-module (by \lemref{l:dualizable in rigid}), tensoring with it commutes with limits
in the category $\QCoh(S)\mmod$. Hence, we obtain: 
$$\QCoh(S')\underset{\QCoh(S)}\otimes \bGamma(S\underset{\CY_2}\times \CY_1,\CC_1)\simeq 
\underset{T\in \affdgSch_{/S\underset{\CY_2}\times\CY_1}}{\underset{\longleftarrow}{lim}}\, 
\left(\QCoh(S')\underset{\QCoh(S)}\otimes \bGamma(T,\CC_1)\right).$$

Note that the functor
$$\affdgSch_{/S\underset{\CY_2}\times\CY_1}\to \affdgSch_{/S'\underset{\CY_2}\times\CY_1},\quad T\mapsto T\underset{S}\times S'$$
is cofinal. Hence, we can calculate $\bGamma(S'\underset{\CY_1}\times \CY_1,\CC_1)$ as
$$\underset{T\in \affdgSch_{/S\underset{\CY_2}\times\CY_1}}{\underset{\longleftarrow}{lim}}\, \left(\bGamma(S'\underset{S}\times T,\CC_1)\right),$$
and the two expressions are manifestly isomorphic.

\end{proof}

\sssec{}

Let 
$$\CY_1\overset{f_{1,2}}\longrightarrow \CY_2 \overset{f_{2,3}}\longrightarrow \CY_3$$
be a pair of morphisms. We have an obvious isomorphism
$$\coRres_{f_{1,2}}\circ \coRres_{f_{2,3}}\simeq \coRres_{f_{1,3}}.$$

By passing to right adjoints we obtain a canonical isomorphism
$$\coIind_{f_{2,3}}\circ \coIind_{f_{1,2}}\simeq \coIind_{f_{1,3}}.$$

\sssec{}

We note that the functor 
$$\bGamma(\CY,-):\on{ShvCat}(\CY)\to \StinftyCat_{\on{cont}}$$
is a particular case of $\coIind$, namely, for for the morphism $p_\CY:\CY\to \on{pt}$.

\medskip

In particular, we obtain:

\begin{lem}
For a morphism $f:\CY_1\to \CY_2$ and $\CC\in \on{ShvCat}(\CY_1)$
there is a canonical isomorphism
$$\bGamma(\CY_2,\coIind_f(\CC))\simeq \bGamma(\CY_1,\CC).$$
\end{lem}

\sssec{}

Let us note the following property of prestacks for which $\bGamma^{\on{enh}}_\CY$ is fully faithful:

\begin{prop}  \label{p:nice base change}
Suppose $\CY$ is such that $\bGamma^{\on{enh}}_\CY$ is fully faithful. Then for $S\in \affdgSch_{/\CY}$ and
$f:\CY'\to \CY$, the map
$$\QCoh(S)\underset{\QCoh(\CY)}\otimes \QCoh(\CY')\to \QCoh(S\underset{\CY}\times \CY')$$
is an equivalence.
\end{prop}

\begin{proof}
Consider the object $$\CC:=\coIind_f(\QCoh_{/\CY'})\in \on{ShvCat}(\CY).$$
Now, the two sides in the proposition are obtained by evaluation the two sides in 
$$\bLoc_\CY(\bGamma^{\on{enh}}(\CY,\CC))\to \CC$$
on $S\in \affdgSch_{/\CY}$. 
\end{proof}

\ssec{Commutation of diagrams}

Let $f:\CY_1\to \CY_2$ be a morphism of presracks. 

\sssec{}

We note that the following diagram is commutative by construction
\begin{equation} \label{e:commute 1}
\CD
\QCoh(\CY_1)\mmod  @>{\bLoc_{\CY_1}}>>  \on{ShvCat}(\CY_1)  \\
@A{\Iind_f}AA    @AA{\coRres_f}A   \\
\QCoh(\CY_2)\mmod  @>{\bLoc_{\CY_2}}>>  \on{ShvCat}(\CY_2).
\endCD
\end{equation} 

By adjunction, the following diagram is commutative as well:
\begin{equation} \label{e:commute 2}
\CD
\QCoh(\CY_1)\mmod  @<{\bGamma^{\on{enh}}_{\CY_1}}<<  \on{ShvCat}(\CY_1)  \\
@V{\Rres_f}VV    @VV{\coIind_f}V   \\
\QCoh(\CY_2)\mmod  @<{\bGamma^{\on{enh}}_{\CY_2}}<<  \on{ShvCat}(\CY_2).
\endCD
\end{equation} 

\sssec{}

Hence, each of the following two diagrams
\begin{equation} \label{e:almost commute 1}
\CD
\QCoh(\CY_1)\mmod  @>{\bLoc_{\CY_1}}>>  \on{ShvCat}(\CY_1)  \\
@V{\Rres_f}VV    @VV{\coIind_f}V   \\
\QCoh(\CY_2)\mmod  @>{\bLoc_{\CY_2}}>>  \on{ShvCat}(\CY_2)
\endCD
\end{equation}
and
\begin{equation} \label{e:almost commute 2}
\CD
\QCoh(\CY_1)\mmod  @<{\bGamma^{\on{enh}}_{\CY_1}}<<  \on{ShvCat}(\CY_1)  \\
@A{\Iind_f}AA    @AA{\coRres_f}A   \\
\QCoh(\CY_2)\mmod  @<{\bGamma^{\on{enh}}_{\CY_2}}<<  \on{ShvCat}(\CY_2)
\endCD
\end{equation}
commutes \emph{up to a natural transformation}.

\sssec{}

For future use, let us record the following:

\begin{lem} \label{l:morphism between 1-affine}
Let $f:\CY_1\to \CY_2$ be a morphism between 1-affine prestacks. Then for $\CC\in \on{ShvCat}(\CY_2)$, 
the functor
$$\QCoh(\CY_1)\underset{\QCoh(\CY_2)}\otimes \bGamma^{\on{enh}}(\CY_2,\CC)\to \bGamma^{\on{enh}}(\CY_1,\coRres_f(\CC))$$
is an equivalence, i.e., the natural transformation in the diagram \eqref{e:almost commute 2} is an isomorphism.
\end{lem}

\begin{proof}
Follows from the commutation of \eqref{e:commute 1}, as the horizontal arrows are equivalences.
\end{proof}

\sssec{}

The following assertion will be a key tool for many proofs: 

\begin{prop}  \label{p:almost commute 1}
Assume that for any $S\in \affdgSch_{/\CY_2}$ the map
$$\QCoh(S)\underset{\QCoh(\CY_2)}\otimes \QCoh(\CY_1)\to \QCoh(S\underset{\CY_2}\times \CY_1)$$
is an equivalence. 

\medskip

\noindent{\em(a)} Suppose that $f$ is such that its base change by every affine DG scheme yields a prestack 
for which $\bLoc$ is fully faithful. Then:

\smallskip

\noindent{\em(i)} The diagram \eqref{e:almost commute 1} commutes, 
i.e., the natural transformation is an isomorphism.

\smallskip

\noindent{\em(ii)} If $\QCoh(\CY_1)$ is dualizable as an object of $\QCoh(\CY_2)\mmod$, then the diagram
\eqref{e:almost commute 2} commutes, i.e., the natural transformation is an isomorphism.

\smallskip

\noindent{\em(iii)} If $\bLoc_{\CY_2}$ is fully faithful, then so is $\bLoc_{\CY_1}$. 

\medskip

\noindent{\em(b)} Suppose that $f$ is such that its base change by every affine DG scheme yields a 1-affine 
prestack. Then if $\bGamma^{\on{enh}}_{\CY_2}$ is fully faithful, then so is $\bGamma^{\on{enh}}_{\CY_1}$.
\end{prop}

\begin{cor}  \label{c:1-affine base and fiber}
Let $f:\CY_1\to \CY_2$ be a map, where $\CY_2$ is 1-affine, and the base change of $f$ by an affine DG scheme yields a 
1-affine prestack. Then $\CY_1$ is 1-affine.
\end{cor}

\begin{proof}
Follows from \propref{p:almost commute 1}, points (a,iii) and (b). The condition of the proposition
holds because of \propref{p:nice base change}.
\end{proof}
 
\begin{cor} \label{c:product 1-affine}
The product of two 1-affine prestacks is 1-affine.
\end{cor}
 
\ssec{Proof of \propref{p:almost commute 1}} 

\sssec{Proof of point \em{(a,i)}}  \hfill

\medskip
 
Fix $\bC_1\in \QCoh(\CY_1)\mmod$ and $S\in \affdgSch_{/\CY_2}$. By definition:
\begin{multline*}
\bGamma(S,\bLoc_{\CY_2}\circ \Rres_f(\bC_1))=\QCoh(S)\underset{\QCoh(\CY_2)}\otimes \bC_1\simeq \\
\simeq (\QCoh(S)\underset{\QCoh(\CY_2)}\otimes \QCoh(\CY_1)) \underset{\QCoh(\CY_1)}\otimes \bC_1,
\end{multline*}
while the latter maps isomorphically to 
$$\QCoh(S\underset{\CY_2}\times \CY_1)\underset{\QCoh(\CY_1)}\otimes \bC_1,$$
by the assumption of the proposition.

\medskip

Also,
$$\bGamma(S,\coIind_f\circ \bLoc_{\CY_1}(\bC_1))\simeq
\bGamma(S\underset{\CY_2}\times \CY_1,\bLoc_{\CY_1}(\bC_1)).$$

Set 
$$\bC:=\QCoh(S\underset{\CY_2}\times \CY_1)\underset{\QCoh(\CY_1)}\otimes \bC_1\in
\QCoh(S\underset{\CY_2}\times \CY_1)\mmod.$$

We have
$$\bLoc_{\CY_1}(\bC_1)|_{S\underset{\CY_2}\times \CY_1}\simeq \bLoc_{S\underset{\CY_2}\times \CY_1}(\bC),$$
and hence
$$\bGamma(S\underset{\CY_2}\times \CY_1,\bLoc_{\CY_1}(\bC_1))\simeq
\bGamma(S\underset{\CY_2}\times \CY_1,\bLoc_{S\underset{\CY_2}\times \CY_1}(\bC)).$$
Now, the assumtion in (a) implies that the latter is isomorphic to $\bC$ itself, as desired.  

\qed

\sssec{}

For the proof of point (a,ii) we will need the following assertion:

\begin{lem}  \label{l:sect on fiber}
Under the assumption of (a)  for  $S\in \affdgSch_{/\CY_2}$ we have:

\smallskip

\noindent{\em(1)}
For $\bC\in \QCoh(S)\mmod$,
the natural map
\begin{equation} \label{e:sect on fiber 1}
\QCoh(\CY_1)\underset{\QCoh(\CY_2)} \otimes \bC\to
\bGamma(S\underset{\CY_2}\times \CY_1,\bLoc_S(\bC))
\end{equation}
is an isomorphism. 

\smallskip

\noindent{\em(2)} For $\CC_2\in \on{ShvCat}(\CY_2)$, the natural map
\begin{equation} \label{e:sect on fiber 2}
\QCoh(\CY_1)\underset{\QCoh(\CY_2)} \otimes \bGamma(S,\CC_2)\to \bGamma(S\underset{\CY_2}\times \CY_1,\CC_2)
\end{equation}
is an isomorphism.
\end{lem} 
 
\begin{proof}

We rewrite the left-hand side in \eqref{e:sect on fiber 1} as
$$(\QCoh(\CY_1)\underset{\QCoh(\CY_2)} \otimes \QCoh(S))\underset{\QCoh(S)}\otimes \bC,$$
which by the assumption of the proposition maps isomorphically to 
$$\QCoh(S\underset{\CY_1}\times \CY_2)\underset{\QCoh(S)}\otimes \bC.$$

So, in order to prove that  \eqref{e:sect on fiber 1} is an isomorphism, we have to show that the natural map
\begin{equation} \label{e:sect on fiber 1a}
\QCoh(S\underset{\CY_1}\times \CY_2)\underset{\QCoh(S)}\otimes \bC\to
\bGamma(S\underset{\CY_2}\times \CY_1,\bLoc_S(\bC))
\end{equation}
is an isomorphism. 

\medskip

Set 
$$\bC':=\QCoh(S\underset{\CY_1}\times \CY_2)\underset{\QCoh(S)}\otimes \bC\in 
\QCoh(S\underset{\CY_1}\times \CY_2)\mmod.$$

We have:
$$\bLoc_S(\bC)|_{S\underset{\CY_2}\times \CY_1}\simeq \bLoc_{S\underset{\CY_2}\times \CY_1}(\bC'),$$
and the map in \eqref{e:sect on fiber 1a} identifies with
$$\bC'\to \bGamma(S\underset{\CY_2}\times \CY_1,\bLoc_{S\underset{\CY_2}\times \CY_1}(\bC')),$$
which is an isomorphism by the assumption in (a). This shows that  \eqref{e:sect on fiber 1} is an
isomorphism.

\medskip

To prove that \eqref{e:sect on fiber 2} is an isomorphism, we note that the two sides identify with
the corresponding sides in \eqref{e:sect on fiber 1} for 
$$\bC:=\bGamma(S,\CC_2),$$
using the fact that
$$\CC_2|_S\simeq \bLoc_S(\bGamma(S,\CC_2)).$$

\end{proof}

\sssec{Proof of point \em{(a,ii)}}

For $\CC_2\in \on{ShvCat}(\CY_2)$, we have 
\begin{equation}  \label{e:sections upstairs as limit}
\bGamma(\CY_1,\coRres_f(\CC_2))\simeq \bGamma(\CY_2,\coIind_f\circ \coRres_f(\CC_2))
\simeq \underset{S\in \affdgSch_{/\CY_2}}{\underset{\longleftarrow}{lim}}\, \bGamma(S\underset{\CY_2}\times \CY_1,\CC_2).
\end{equation}

By \lemref{l:sect on fiber}(2), we have 
$$\bGamma(S\underset{\CY_2}\times \CY_1,\CC_2)\simeq \QCoh(\CY_1)\underset{\QCoh(\CY_2)} \otimes \bGamma(S,\CC_2).$$

Hence, the expression in \eqref{e:sections upstairs as limit} identifies with 
\begin{equation}  \label{e:sections upstairs as limit 1}
\underset{S\in \affdgSch_{/\CY_2}}{\underset{\longleftarrow}{lim}}\,  \QCoh(\CY_1)\underset{\QCoh(\CY_2)} \otimes \bGamma(S,\CC_2).
\end{equation}

Now, the assumption that $\QCoh(\CY_1)$ is dualizable as an object of $\QCoh(\CY_2)\mmod$ implies that
$$\QCoh(\CY_1)\underset{\QCoh(\CY_2)} \otimes -:\QCoh(\CY_2)\mmod\to \StinftyCat_{\on{cont}}$$
commutes with limits, so we can rewrite the expression in \eqref{e:sections upstairs as limit 1} as
$$ \QCoh(\CY_1)\underset{\QCoh(\CY_2)} \otimes\left( \underset{S\in \affdgSch_{/\CY_2}}{\underset{\longleftarrow}{lim}}\, \bGamma(S,\CC_2)\right)\simeq
\QCoh(\CY_1)\underset{\QCoh(\CY_2)} \otimes \bGamma(\CY_2,\CC_2),$$
as desired.

\sssec{Proof of point \em{(a,iii)}}

We need to show that the unit map
$$\bC_1\to \bGamma^{\on{enh}}(\CY_1,\bLoc_{\CY_1}(\bC_1))$$
is an isomorphism. Note that the functor $\Rres_f$ is conservative. Hence, it suffices to show that
$$\Rres_f(\bC_1)\to \Rres_f\left(\bGamma^{\on{enh}}(\CY_1,\bLoc_{\CY_1}(\bC_1))\right)$$
is an isomorphism. However, we have a commutative diagram
$$
\CD
\Rres_f\left(\bGamma^{\on{enh}}(\CY_1,\bLoc_{\CY_1}(\bC_1))\right)  @>{\sim}>> \bGamma^{\on{enh}}\left(\CY_2,\coIind_f(\bLoc_{\CY_1}(\bC_1))\right)  \\
@AAA    @AAA    \\
\Rres_f(\bC_1)  @>>>  \bGamma^{\on{enh}}\left(\CY_2,\bLoc_{\CY_2}(\Rres_f(\bC_1))\right),
\endCD
$$
where the right vertical arrow is an isomorphism by point (a,i). Hence, if $\bLoc_{\CY_2}$ is fully faithful, the bottom
horizontal arrow is an isomorphism, and hence so is the left vertical arrow. 

\medskip

\sssec{Proof of point \em{(b)}}

By an argument similar that in point (a,iii), it suffices to show that under the assumption of point (b), the functor 
$\coIind_f$ is conservative.

\medskip

Let $\phi:\CC'_1\to \CC''_1$ be a morphism in $\on{ShvCat}(\CY_1)$, such that
$\coIind_f(\phi)$ is an isomorphism. We need to show that for every 
$T\in \affdgSch_{\CY_1}$, the resulting map
\begin{equation} \label{e:restr to fiber 1}
\bGamma(T,\CC'_1)\to \bGamma(T,\CC''_1)
\end{equation}
is an isomorphism (under the assumption of the proposition). 

\medskip

Set $\CZ:=T\underset{\CY_2}\times \CY_1$, considered as a prestack over $\CY_1$. Consider the corresponding map
\begin{equation} \label{e:restr to fiber 2}
\CC'_1|_{\CZ}\to \CC''_1|_{\CZ}.
\end{equation}

The assumption that $\coIind_f(\phi)$ is an isomorphism implies that the induced map
$$\bGamma(\CZ,\CC'_1|_{\CZ})\to \bGamma(\CZ,\CC''_1|_{\CZ})$$
is an isomorphism. Now, the fact that $\CZ$ is 1-affine implies that \eqref{e:restr to fiber 2}
is an isomorphism as well. 

\medskip

Evaluating \eqref{e:restr to fiber 2} on $T\in \affdgSch_{/\CZ}$, we obtain that \eqref{e:restr to fiber 1}
is an isomorphism, as required.

\qed

\section{The case of formal completions}  \label{s:formal compl}

Let $\CY$ be a $1$-affine prestack, $\CY'\overset{\iota}\to \CY$
a closed embedding, and let $\CY_0\overset{\jmath}\hookrightarrow \CY$ be the complementary open. 
Throughout this section, we will be assuming that $\jmath$ is quasi-compact. 

\ssec{$\QCoh$-modules on a formal completion}

\sssec{}  \label{sss:with supports}

We have an adjoint pair of functors
$$\jmath^*:\QCoh(\CY)\rightleftarrows \QCoh(\CY_0):\jmath_*.$$

The assumption that $\jmath$ is quasi-compact implies that $\jmath_*$ is continuous (see, e.g., \cite[Proposition 2.1.1]{QCoh}).

\medskip

Let $\QCoh(\CY)_{\CY'}$ be the full subcategory of $\QCoh(\CY)$ consisting of objects set-theoretically
supported on $\CY'$, i.e., $\QCoh(\CY)_{\CY'}=\on{ker}(\jmath^*)$. 

\medskip

Let
$$\wh\imath^{\QCoh}_!:\QCoh(\CY)_{\CY'}\hookrightarrow \QCoh(\CY)$$
denote the tautological embedding.

\medskip

The functor $\wh\imath^{\QCoh}_!$ admits a continuous right adjoint, denoted by $\wh\imath^{\QCoh,!}$, 
and given by 
$$\CF\mapsto \on{Cone}(\CF\to \jmath_*\circ \jmath^*(\CF))[-1].$$

We obtain a localization sequence:
\begin{equation} \label{e:localization seq}
\QCoh(\CY)_{\CY'}\underset{\wh\imath^{\QCoh,!}}{\overset{\wh\imath^{\QCoh}_!}\rightleftarrows} \QCoh(\CY) 
\underset{\jmath_*}{\overset{\jmath^*}\rightleftarrows} \QCoh(\CY_0).
\end{equation}

\sssec{}  \label{sss:formal compl}

Let $\CY^\wedge_{\CY'}$ the formal completion of $\CY$ along $\CY'$, see \cite[Defn. 6.1.2]{IndSch}.
I.e.,  $\CY^\wedge_{\CY'}$ is the prestack defined by
$$\Maps(S,\CY^\wedge_{\CY'}):=\Maps(S,\CY)\underset{\Maps(({}^{cl}S)_{red},\CY)}\times \Maps(({}^{cl}S)_{red},\CY').$$
Let $\wh{i}:\CY^\wedge_{\CY'}\to \CY$ denote the tautological map. 

\medskip

The following results from \cite[Proposition 7.1.3]{IndCoh} by base change:

\begin{prop} \label{p:compl and supp}
The functor
$$\wh\imath^*:\QCoh(\CY)\to \QCoh(\CY^\wedge_{\CY'})$$ factors as 
$$\QCoh(\CY) \overset{\wh\imath^{\QCoh,!}}\longrightarrow \QCoh(\CY)_{\CY'}\to \QCoh(\CY^\wedge_{\CY'}),$$
where the second arrow is an equivalence.
\end{prop}

\sssec{}

Consider now the adjoint functors
$$\Iind_{\wh{i}}:\QCoh(\CY)\mmod\rightleftarrows \QCoh(\CY^\wedge_{\CY'})\mmod:\Rres_{\wh{i}}.$$

\begin{prop} \label{p:X vs Y mod}
The functor $\Rres_{\wh{i}}$ is fully faithful; its essential image consists of those $\bC_\CY\in \QCoh(\CY)\mmod$,
on which $\on{ker}(\wh{i}^*)=\on{Im}(\jmath_*)\subset \QCoh(\CY)$ acts trivially (i.e., by zero). 
\end{prop}

The assertion of the proposition follows from  \propref{p:compl and supp} and the next general assertion:

\medskip

Let $\bO$ be a monoidal DG category, and let $\sF:\bO\to \bO'$ be a monoidal functor. 
Let $\bC$ be an $\bO$-module category, on which $\on{ker}(\sF)$ acts by zero.

\begin{lem} \label{l:tensor up local}
Assume that $\sF$ admits a fully faithful continuous right or left adjoint, which is a map of right $\bO$-module categories. 
Then the canonical map
$$\bC\simeq  \bO\underset{\bO}\otimes \bC \to \bO'\underset{\bO}\otimes \bC$$
is an equivalence. 
\end{lem}

\begin{proof}

Let $\sF$ have a fully continuous right adjoint.  The assumption on $\sF$ implies that in the localization sequence 
$$\bO'\overset{\sF} \rightleftarrows \bO \rightleftarrows \on{ker}(\sF),$$
all functors are maps of right $\bO$-module categories. Hence, tensoring up by $\bC$ over $\bO$ on the right, we
obtain a localization sequence of DG categories:
$$\bO'\underset{\bO}\otimes \bC \overset{\sF\otimes \on{Id}_\bC} \rightleftarrows \bC \rightleftarrows 
\on{ker}(\sF)\underset{\bO}\otimes \bC.$$

\medskip

However, the assumption on $\bC$ says that the functor $\bC\leftarrow \on{ker}(\sF)\underset{\bO}\otimes \bC$ is zero, and
since this functor is fully faithful, we obtain that $\on{ker}(\sF)\underset{\bO}\otimes \bC=0$. Hence, the functor
$$\bO'\underset{\bO}\otimes \bC \leftarrow \bC,$$
is an equivalence, as desired.

\medskip

The proof when $\sF$ admits a fully faithful left adjoint is similar.

\end{proof}

\ssec{Sheaves of categories on a formal completion}

\sssec{}

Consider now the pair of adjoint functors
$$\coRres_{\wh{i}}:\on{ShvCat}(\CY)\rightleftarrows \on{ShvCat}(\CY^\wedge_{\CY'}):\coIind_{\wh{i}},$$
see \secref{sss:dir and inv}.

\medskip

Since $\CY^\wedge_{\CY'}\underset{\CY}\times \CY^\wedge_{\CY'}\simeq \CY^\wedge_{\CY'}$, the adjunction map
$$\coRres_{\wh{i}}\circ \coIind_{\wh{i}}\to \on{Id}$$
is an isomorphism. Hence, the functor $\coIind_{\wh{i}}$ is fully faithful. We now claim:

\begin{prop}  \label{p:image formal}
The essential image of $\coIind_{\wh{i}}$ consists of those $\CC\in \on{ShvCat}(\CY)$,
for which $\CC|_{\CY_0}=0$. 
\end{prop}

\begin{proof}

The assertion readily reduces to the case when $\CY=S\in \affdgSch$. Let $S'\subset S$ be a closed DG subscheme
whose complement $S_0\overset{\jmath}\hookrightarrow S$ is quasi-compact. We need to show that for
$$\bC\in \QCoh(S)\mmod,$$
on which the action of $\QCoh(S)$ factors through the restriction functor $\QCoh(S)\to \QCoh(S^\wedge_{S'})$, the map
$$\bC=\bGamma(S,\bLoc_S(\bC))\to \bGamma(S^\wedge_{S'},\bLoc_S(\bC))$$
is an equivalence.

\medskip

By \cite[Proposition 6.7.4]{IndSch}, we can exhibit $S^\wedge_{S'}$ as
$$\underset{n}{\underset{\longrightarrow}{lim}}\, S_n,$$
where $S_n$ are closed subschemes of $S$, and the transition maps $\iota_{n_1,n_2}:S_{n_1}\to S_{n_2}$
are such that the functors $\iota_{n_1,n_2}^*$ admit \emph{left} adjoints. 

\medskip

In this case, by \cite[Lemma 1.3.3]{DGCat}, we calculate
$$\QCoh(S^\wedge_{S'}):=\underset{n}{\underset{\longleftarrow}{lim}}\, \QCoh(S_n)\simeq
\underset{n}{\underset{\longrightarrow}{colim}}\, \QCoh(S_n),$$
where the \emph{limit} is taken with respect to the transition functors $\iota_{n_1,n_2}^*$,
and the \emph{colimit} is taken with respect to the transition functors $(\iota_{n_1,n_2}^*)^L$.

\medskip

Similarly, 
\begin{multline*}
\bGamma(S^\wedge_{S'},\bLoc_S(\bC)):=
\underset{n}{\underset{\longleftarrow}{lim}}\, \left(\bC\underset{\QCoh(S)}\otimes \QCoh(S_n)\right)\simeq 
\underset{n}{\underset{\longrightarrow}{colim}}\, \left(\bC\underset{\QCoh(S)}\otimes \QCoh(S_n)\right)\simeq \\
\simeq \bC\underset{\QCoh(S)}\otimes \left(\underset{n}{\underset{\longrightarrow}{colim}}\, \QCoh(S_n)\right) \simeq 
\bC\underset{\QCoh(S)}\otimes \QCoh(S^\wedge_{S'})\simeq \bC,
\end{multline*}
where the latter is isomorphism holds by \propref{p:compl and supp} and \lemref{l:tensor up local}.

\end{proof}

\ssec{Proof of \thmref{t:main formal}}

\sssec{}  \label{sss:form loc}

Recall the commutative diagram:

$$
\CD
\on{ShvCat}(\CY^\wedge_{\CY'})   @>{\coIind_{\wh{i}}}>>  \on{ShvCat}(\CY)  \\
@V{\bGamma^{\on{enh}}_{\CY^\wedge_{\CY'}}}VV     @VV{\bGamma^{\on{enh}}_{\CY}}V  \\
\QCoh(\CY^\wedge_{\CY'})\mmod   @>{\Rres_{\wh{i}}}>> \QCoh(\CY)\mmod.
\endCD
$$

\medskip

By Propositions \ref{p:X vs Y mod} and \ref{p:image formal}, the horizontal arrows in this diagram are fully faithful.
Hence, we obtain that if $\bGamma^{\on{enh}}_{\CY}$ is fully faithful, then so is 
$\bGamma^{\on{enh}}_{\CY^\wedge_{\CY'}}$.

\sssec{}

We now claim that the map $$\wh{i}:\CY^\wedge_{\CY'}\to \CY$$
satisfies the conditions of \propref{p:almost commute 1}(a). 

\medskip

Indeed, the assumption of \propref{p:almost commute 1} is satisfied by the localization sequence
\eqref{e:localization seq} and \cite[Proposition 3.2.1]{QCoh}, applied to $\jmath$. 
The assumption of \propref{p:almost commute 1}(a) is satisfied by \secref{sss:form loc} above, applied
to an affine DG scheme and its formal completion. 

\medskip

Hence, by \propref{p:almost commute 1}(a,i), we obtain that the diagram
$$
\CD
\on{ShvCat}(\CY^\wedge_{\CY'})   @>{\coIind_{\wh{i}}}>>  \on{ShvCat}(\CY)  \\
@A{\bLoc_{\CY^\wedge_{\CY'}}}AA     @AA{\bLoc_{\CY}}A  \\
\QCoh(\CY^\wedge_{\CY'})\mmod   @>{\Rres_{\wh{i}}}>> \QCoh(\CY)\mmod.
\endCD
$$
commutes as well.  

\sssec{}

From the above diagram, we conlcude that if the functor $\bLoc_{\CY}$ is fully faithful, then so is 
$\bLoc_{\CY^\wedge_{\CY'}}$.

\qed

\bigskip

\centerline{\bf Part II: 1-Affinness of Algebraic Stacks} 

\section{Algebraic stacks: preparations}  \label{s:alg space}

\ssec{Passable prestacks}  \label{ss:passable}

\sssec{}

We give the following definition, taken from \cite[Sect. 3.3]{QCoh}:

\begin{defn}
A prestack $\CY$ is called \emph{passable} if it satisfies:
\begin{itemize}

\item The diagonal morphism $\CY\to \CY\times \CY$ is schematic, quasi-affine and quasi-compact;

\item $\CO_\CY\in \QCoh(\CY)$ is compact;

\item $\QCoh(\CY)$ is dualizable as a plain DG category.

\end{itemize}
\end{defn}

As our initial observation, we will prove:
\begin{prop} \label{p:alg}  Let $\CY$ be a passable prestack. Then the functor 
$$\bGamma^{\on{enh}}_\CY:\on{ShvCat}(\CY)\to \QCoh(\CY)\mmod$$
is fully faithful. 

\end{prop} 

\sssec{}

Many algebraic stacks are passable as prestacks. In particular, any algebraic stack, which in the  
terminology of \cite{BFN} is \emph{perfect}, is passable. 

\medskip

The following assertion is proved in \cite[Theorems 4.3.1 and 1.4.2]{DrGa}:

\begin{thm}  \label{t:ev coconn pass}
An eventually coconnective QCA \footnote{Under the ``locally almost of finite type" assumption, 
QCA means ``quasi-compact and the automorphism group of every field-valued point is affine."}
algebraic stack locally almost of finite type is passable.
\end{thm}

We conjecture that in \thmref{t:ev coconn pass} the hypothesis that $\CY$ should be eventually coconnective
is unnecessary. 

\sssec{}

The proof of \propref{p:alg} will use the following ingredient (see \cite[Proposition 2.3.2]{QCoh}; the proof is reproduced in 
\secref{sss:pass rigid} for the reader's convenience):

\begin{prop}  \label{p:passable rigid}
If $\CY$ is passable, the category $\QCoh(\CY)$ is rigid as a monoidal DG category.
\end{prop}

\begin{proof}[Proof of \propref{p:alg}]

We need to show that for $\CC\in \on{ShvCat}(\CY)$ and $T\in \affdgSch_{/\CY}$, the canonical map
$$\QCoh(T)\underset{\QCoh(\CY)}\otimes \bGamma(\CY,\CC)\to \bGamma(T,\CC)$$
is an isomorphism. We will prove this by applying \propref{p:almost commute 1}(a,ii) to the morphism $T\to \CY$.

\medskip

The condition of \propref{p:almost commute 1} holds by \cite[Proposition 3.3.3]{QCoh} (for the reader's convenience
we will reproduce the proof in \propref{p:base change pass}). 

\medskip

The condition of \propref{p:almost commute 1}(a) holds by \thmref{t:alg space} 
in the case of the quasi-affine schemes (which will be proved independently). 

\medskip

The condition of \propref{p:almost commute 1}(a,ii) holds by \lemref{l:dualizable in rigid}, using \propref{p:passable rigid}.

\end{proof} 

\ssec{A corollary of fully faithfulness of $\bLoc$}

\sssec{}

Let $\CY_1,\CY_2$ be two objects of $\on{PreStk}$, and recall that we have a canonically defined (symmetric monoidal) functor
\begin{equation} \label{e:product map}
\QCoh(\CY_1)\otimes \QCoh(\CY_2)\to \QCoh(\CY_1\times \CY_2).
\end{equation}

Recall also that the map \eqref{e:product map} is an equivalence if for one of the prestacks, the category 
$\QCoh(\CY_i)$ is dualizable (for the proof see, e.g., \cite[Proposition 1.4.4]{QCoh}). 

\sssec{}

We shall now prove:

\begin{prop} \label{p:Loc and prod} 
Suppose that $\CY_1$ is such that $\bLoc_{\CY_1}$ is fully faithful. Then \eqref{e:product map}
is an equivalence.
\end{prop}

\begin{proof}

Consider 
$$\bC:=\QCoh(\CY_1)\otimes \QCoh(\CY_2)\in \QCoh(\CY_1)\mmod.$$

The value of $\bLoc_{\CY_1}(\bC)$ on $S\in \affdgSch_{/\CY_1}$ is 
$$\QCoh(S)\underset{\QCoh(\CY_1)}\otimes \left(\QCoh(\CY_1)\otimes \QCoh(\CY_2)\right)\simeq
\QCoh(S)\otimes \QCoh(\CY_2),$$
and the latter is isomorphic to $\QCoh(S\times \CY_2)$, since $\QCoh(S)$ is dualiazable. 

\medskip

Hence, the category $\bGamma(\CY_1,\bLoc_{\CY_1}(\bC))$ is 
$$\underset{S\in \affdgSch_{/\CY_1}}{\underset{\longleftarrow}{lim}}\, \QCoh(S\times \CY_2),$$
and the latter is isomorphic to $\QCoh(\CY_1\times \CY_2)$. 

\medskip

Hence, if $\bC\to \bGamma(\CY_1,\bLoc_{\CY_1}(\bC))$ is an equivalence, then so is
$$\QCoh(\CY_1)\otimes \QCoh(\CY_2)\to \QCoh(\CY_1\times \CY_2).$$

\end{proof}

\sssec{}

As a corollary of \propref{p:Loc and prod}, we obtain:

\begin{cor} \label{c:Loc and prod}
Let $\CY_1,\CY_2\in \on{PreStk}$ such that the functors $\bLoc_{\CY_i}$ $i=1,2$ are fully faithful. Then
$\bLoc_{\CY_1\times \CY_2}$ is also fully faithful.
\end{cor}

\begin{proof}
Follows from \propref{p:almost commute 1}(a,iii), applied to the map
$\CY_1\times \CY_2\to \CY_1$.
\end{proof}

\sssec{}

As another corollary of \propref{p:Loc and prod}, we obtain:

\begin{cor} 
Let $\CY$ be prestack that satisfies:

\begin{itemize}

\item The diagonal morphism $\CY\to \CY\times \CY$ is representable, quasi-compact and quasi-separated;

\item $\CO_\CY\in \QCoh(\CY)$ is compact;

\item $\bLoc_\CY$ is fully faithful.

\end{itemize}
Then $\QCoh(\CY)$ is rigid as a monoidal DG category, and in particular, dualizable as a plain DG category. 
\end{cor}

\begin{proof}
Follows from \cite[Proposition 2.3.2]{QCoh}.
\end{proof}

\sssec{}

Hence, we we obtain:

\begin{cor}  \label{c:Loc and pass}
Let $\CY\in \on{PreStk}$ be such that: 

\begin{itemize}

\item The diagonal morphism $\CY\to \CY\times \CY$ is schematic, quasi-affine and quasi-compact;

\item $\CO_\CY\in \QCoh(\CY)$ is compact.

\end{itemize}

Then we have the following implications: \hfill

\smallskip

\hskip2cm $\bLoc_\CY$ is fully faithful $\Rightarrow$ $\CY$ is passable $\Rightarrow$ $\bGamma^{\on{enh}}_\CY$ is fully faithful.

\end{cor}

\ssec{Algebraic spaces: proof of \thmref{t:alg space}}   \label{ss:alg space}

We shall first prove \thmref{t:alg space}, assuming its validity in the case of quasi-compact quasi-affine schemes.
In particular, we assume the validity of \propref{p:alg}.

\sssec{}

First, we know that any quasi-compact quasi-separated algebraic space $\CY$ is passable as a prestack: this is
given by \cite[Propositions 2.2.2 and 2.3.6]{QCoh}. Hence, the functor $\bGamma^{\on{enh}}_\CY$ is fully faithful
by \propref{p:alg}. It remains to show that $\bLoc_\CY$ is fully faithful. 

\medskip

By \cite[Lemma 2.2.4]{QCoh}, we can exhibit $\CY$ is a finite union
$$\emptyset=\CY_0 \subset \CY_1\subset \ldots \subset \CY_{k-1}\subset \CY_k=\CY$$
of open subsets, such that for each $1\leq i\leq k$ there exists a quasi-affine quasi-compact
scheme $U_k$ equipped with \'etale map $f_k:U_k\to \CY_k$, which is one-to-one over  
$\CY_k-\CY_{k-1}$. 

\medskip

By induction, we can assume that $\CY_{k-1}$ is 1-affine. Thus, we can assume having a Cartesian
diagram of algebraic spaces
$$
\CD
U'  @>{j_U}>> U  \\
@V{f'}VV    @VV{f}V   \\
\CY'  @>{j_\CY}>>  \CY,
\endCD
$$
where the horizontal maps are open embeddings, the vertical maps are \'etale, $f$ is one-to-one over $\CY-\CY'$,
and $U,U',\CY'$ are 1-affine. We wish to deduce that the functor $\bLoc_\CY$ is fully faithful. 

\sssec{}

In \secref{sss:Nisn} it is shown that for $\bC\in \QCoh(\CY)\mmod$ the diagram
\begin{equation} \label{e:pullback 1}
\CD
\QCoh(U')\underset{\QCoh(\CY)}\otimes \bC  @<{(j_U)^*\otimes \on{Id}_\bC}<<  \QCoh(U)\underset{\QCoh(\CY)}\otimes \bC \\
@A{(f')^*\otimes \on{Id}_\bC}AA    @AA{f^*\otimes \on{Id}_\bC}A   \\ 
\QCoh(\CY')\underset{\QCoh(\CY)}\otimes \bC  @<{(j_\CY)^*\otimes \on{Id}_\bC}<<  \bC
\endCD
\end{equation}
is a pull-back diagram of DG categories.

\medskip

For $\CC\in \on{ShvCat}(\CY)$ consider the diagram
\begin{equation} \label{e:pullback 2}
\CD
\bGamma(U',\CC)   @<{(j_U)^*_\CC}<<  \bGamma(U,\CC) \\
@A{(f')^*_\CC}AA    @AA{f^*_\CC}A   \\
\bGamma(\CY',\CC)   @<{(j_\CY)^*_\CC}<<  \bGamma(\CY,\CC),
\endCD
\end{equation}
which is a pull-back diagram by \corref{c:Cech descent}(a) (applied to Nisnevich covers). 

\medskip

Taking $\CC:=\bLoc_\CY(\bC)$, and using the fact that $\CY'$ is 1-affine, we obtain that the map
$$\QCoh(\CY')\underset{\QCoh(\CY)}\otimes \bC\to  \bGamma\left(\CY',\bLoc_{\CY'}(\QCoh(\CY')
\underset{\QCoh(\CY)}\otimes \bC)\right)\simeq 
\bGamma(\CY',\CC)$$
is an equivalence, and the same holds for $\CY'$ replaced by $U'$ and $U$.

\medskip

Hence, we obtain a map from the diagram \eqref{e:pullback 1} to \eqref{e:pullback 2}, which induces equivalences
$$\QCoh(U)\underset{\QCoh(\CY)}\otimes \bC \to \bGamma(U,\CC),\quad 
\QCoh(\CY')\underset{\QCoh(\CY)}\otimes \bC \to \bGamma(\CY',\CC)$$ and 
$$\QCoh(U')\underset{\QCoh(\CY)}\otimes \bC \to \bGamma(U',\CC).$$

\medskip

Hence, the map
$$\bC\to \bGamma(\CY,\CC)$$
is also an equivalence, as required.

\sssec{}

To finish the proof of \thmref{t:alg space}, it remains to give an \emph{a priori} proof in the case when $\CY$
is a quasi-compact quasi-affine scheme. More generally, let us assume that $\CY$ is a quasi-compact separated
scheme. 

\medskip

Note that \propref{p:alg} is valid for $\CY$, because the diagonal of $\CY$ is a closed embedding, and hence is affine
(in the proof of \propref{p:alg} we used the fact that the base change of the diagonal morphism of $\CY$ by an affine
scheme yields a prestack which is 1-affine, which is tautological if the diagonal morphism is affine). 

\medskip

The proof proceeds by induction on the number of affine opens by which we can cover $\CY$.  Namely, we repeat
the proof of \thmref{t:alg space}, given above, replacing the word ``\'etale morphism" by ``open embedding",
and where the schemes $U_k$ are affine. In the induction step,  $\CY'$ and $U'$
are both 1-affine by the induction hypothesis. 

\qed

\ssec{Direct image for representable morphisms}  \label{ss:dir image}

\sssec{}

Let $f:\CY_1\to \CY_2$ be a representable, quasi-compact and quasi-separated 
morphism between prestacks. For $\CC\in \on{ShvCat}(\CY_2)$, 
we have a canonically defined functor
$$f^*_\CC:\bGamma(\CY_2,\CC)\to \bGamma(\CY_1,\CC)=\bGamma(\CY_1,\coRres_f(\CC)).$$

\medskip

We claim:

\begin{prop}  \label{p:direct image}
The above functor $f^*_\CC:\bGamma(\CY_2,\CC)\to \bGamma(\CY_1,\coRres_f(\CC))$
admits a continuous right adjoint (to be denoted $f_{\CC,*}$).  For a Cartesian diagram of prestacks 
$$
\CD
\CY'_1 @>{g_1}>>  \CY_1 \\
@V{f'}VV   @VV{f}V  \\
\CY'_2 @>{g_2}>>  \CY_2,
\endCD
$$
the diagram
$$
\CD
\bGamma(\CY'_1,\CC)   @<{(g_1)^*_\CC}<<   \bGamma(\CY_1,\CC) \\
@V{f'_{\CC,*}}VV   @VV{f_{\CC,*}}V   \\
\bGamma(\CY'_2,\CC)   @<{(g_2)^*_\CC}<<   \bGamma(\CY_2,\CC),
\endCD
$$
which \emph{a priori} commutes up to a natural transformation, actually commutes.
\end{prop}

\sssec{Proof of \propref{p:direct image}, Step 1}

Let us first consider the case when $\CY_2$ is a quasi-compact and quasi-separated algebrac space. In this case,
$\CY_1$ has the same properties.

\medskip

By \lemref{l:morphism between 1-affine}, we have
$$\bGamma(\CY_1,\coRres_f(\CC))\simeq \QCoh(\CY_1)\underset{\QCoh(\CY_2)}\otimes \bGamma(\CY_2,\CC).$$

In this case, the sought-for functor $f_{\CC,*}$ identifies with
$$\QCoh(\CY_1)\underset{\QCoh(\CY_2)}\otimes \bGamma(\CY_2,\CC) \overset{f_*\otimes \on{Id}_{\bGamma(\CY_2,\CC)}}
\longrightarrow 
\QCoh(\CY_2)\underset{\QCoh(\CY_2)}\otimes \bGamma(\CY_2,\CC)= \bGamma(\CY_2,\CC).$$

\sssec{Proof of \propref{p:direct image}, Step 2}

It suffices to show that for any $S\in \affdgSch_{/\CY_2}$, the functor
$$\bGamma(S,\CC)\to \bGamma(S\underset{\CY_2}\times \CY_1,\CC)$$
admits a right adjoint, and that for a map $g:S'\to S$, the corresponding diagram
$$
\CD
\bGamma(S'\underset{\CY_2}\times \CY_1,\CC)   @<<<   \bGamma(S\underset{\CY_2}\times \CY_1,\CC) \\
@VVV   @VVV   \\
\bGamma(S',\CC)   @<<<   \bGamma(S,\CC),
\endCD
$$
commutes. 

\medskip

However, this follows from Step 1 using base change for $\QCoh$ for maps between quasi-compact and quasi-separated
algebraic spaces (see, e.g., \cite[Corollary 1.4.5]{DrGa} for the latter assertion). 

\qed

\ssec{Global sections via a \v{C}ech cover} \label{ss:Cech}

Let $\CY$ be a quasi-compact algebraic stack. 
In this subsection we will describe a more economical way to compute the functor $\bGamma(\CY,-)$. 

\sssec{}

Let $U\to \CY$ be an fppf cover, where $U$ is a quasi-compact and 
quasi-separated algebraic space. Note that all the terms of the \v{C}ech nerve $U^\bullet/\CY$ 
are also quasi-compact and quasi-separated algebraic spaces. 

\medskip

Consider the co-simplicial category $\bGamma(U^\bullet/\CY,\CC)^*$: for $[i]\in \bDelta$ the category of $i$-th simplices is 
$\bGamma(U^i/\CY,\CC)$, and for a map $\alpha:[j]\to [i]$ in $\bDelta$, and the corresponding map
$f^\alpha:U^i/\CY\to U^j/\CY$, the functor 
$$\bGamma(U^j/\CY,\CC)\to \bGamma(U^i/\CY,\CC)$$
is $(f^\alpha)^*_\CC$. 

\medskip

The following results from \corref{c:Cech descent}(a): 

\begin{lem}  \label{l:sect via Cech}
For $\CC\in \on{ShvCat}(\CY)$, the restriction map
$$\Gamma(\CY,\CC)\to \on{Tot}\left(\bGamma(U^\bullet/\CY,\CC)^*\right)$$
is an equivalence.
\end{lem}

\medskip

We shall now describe the category $\on{Tot}\left(\bGamma(U^\bullet/\CY,\CC)^*\right)$ as co-modules
over a co-monad acting on $\bGamma(U,\CC)$. 

\sssec{}

We claim that the co-simplicial category $\bGamma(U^\bullet/\CY,\CC)$ satisfies the \emph{co-monadic} Beck-Chevalley condition
(see \secref{ss:BC} for what this means). Indeed, this follows from \propref{p:direct image}. 

\medskip

Hence, from \lemref{l:comonadicity}, we obtain:

\begin{lem} \hfill \label{l:BC*}

\smallskip

\noindent{\em(a)} 
The functor of evaluation on $0$-simplices 
$$\on{ev}^0:\on{Tot}\left(\bGamma(U^\bullet/\CY,\CC)^*\right)\to \bGamma(U,\CC)$$
admits a (continuous) right adjoint; to be denoted $(\on{ev}^0)^R$.

\smallskip

\noindent{\em(b)} The co-monad $\on{Av}^{U/\CY}_*:=\on{ev}^0\circ (\on{ev}^0)^R$, acting on $\bGamma(U,\CC)$, is isomorphic,
as a plain endo-functor, to $(\on{pr}_2)_{\CC,*}\circ (\on{pr}_1)_\CC^*$, where 
$\on{pr}_1,\on{pr}_2$ denote the two projections
$$U\underset{\CY}\times U\rightrightarrows U.$$

\smallskip

\noindent{\em(c)} The functor
$$\on{ev}^0:\on{Tot}\left(\bGamma(U^\bullet/\CY,\CC)^*\right) \to \bGamma(U,\CC)$$
is co-monadic, i.e., the natural functor $$\on{Tot}\left(\bGamma(U^\bullet/\CY,\CC)\right)\to 
\on{Av}^{U/\CY}_*\comod(\bGamma(U,\CC))$$
is an equivalence. 

\end{lem}

Combining with \lemref{l:sect via Cech}, we obtain:

\begin{cor} \label{c:sect via comonad}
The functor 
$$f^*_\CC:\bGamma(\CY,\CC)\to  \bGamma(U,\CC)$$
is co-monadic.
\end{cor} 

\ssec{The \v{C}ech picture for sheaves of categories}

We retain the notations of \secref{ss:Cech}.

\sssec{}

Consider the following co-simplicial object of $\inftyCat$, denoted $\on{ShvCat}(U^\bullet/\CY)$. For $[i]\in \bDelta$,
its category of $i$-simplices is $\on{ShvCat}(U^i/\CY)$. For a map $\alpha:[j]\to [i]$ in $\bDelta$, the functor
$$\on{ShvCat}(U^i/\CY)\to \on{ShvCat}(U^i/\CY)$$
is $\coRres_{f^\alpha}$. 

\medskip

The following results from \corref{c:shv via Cech}(a):

\begin{lem}  \label{l:shv via Cech}
The pullback functor
$$\on{ShvCat}(\CY)\to \on{Tot}\left(\on{ShvCat}(U^\bullet/\CY)\right)$$
is an equivalence.
\end{lem}

\sssec{}

We now claim that $\on{ShvCat}(U^\bullet/\CY)\in \inftyCat^{\bDelta}$ satsfies the co-monadic Beck-Chevalley 
condition. This amounts to the fact that the diagram
$$
\CD
\on{ShvCat}(U^i/\CY)   @<{\coIind_{f^{\partial_i}}}<<  \on{ShvCat}(U^{i+1}/\CY)    \\
@A{\coRres_{f^\alpha}}AA   @A{\coRres_{f^{\alpha+1}}}AA \\
\on{ShvCat}(U^j/\CY)   @<{\coIind_{f^{\partial_i}}}<<  \on{ShvCat}(U^{j+1}/\CY) 
\endCD
$$
being commutative, which follows from the definitions.

\medskip

Hence, from \lemref{l:comonadicity}, we obtain:

\begin{lem} \hfill \label{l:BC-cat}

\smallskip

\noindent{\em(a)} 
The functor of evaluation on $0$-simplices 
$$\on{ev}^0:\on{Tot}\left(\on{ShvCat}(U^\bullet/\CY)\right)\to \on{ShvCat}(U)$$
admits a right adjoint; to be denoted $(\on{ev}^0)^R$.

\smallskip

\noindent{\em(b)} The co-monad $\on{ev}^0\circ (\on{ev}^0)^R$, acting on $\on{ShvCat}(U)$, is isomorphic,
as a plain endo-functor, to $\coIind_{\on{pr}_2}\circ \coRres_{\on{pr}_1}$.

\smallskip

\noindent{\em(c)} The functor
$$\on{ev}^0:\on{Tot}\left(\on{ShvCat}(U^\bullet/\CY)\right) \to \on{ShvCat}(U)$$
is co-monadic, i.e., the natural functor $$\on{Tot}\left(\on{ShvCat}(U^\bullet/\CY)\right)\to 
\left(\on{ev}^0\circ (\on{ev}^0)^R\right)\comod(\on{ShvCat}(U))$$
is an equivalence. 

\end{lem}

Combining with \lemref{l:shv via Cech}, we obtain:

\begin{cor} \label{c:shv via comonad}
The functor
$$\coRres_f:\on{ShvCat}(\CY)\to \on{ShvCat}(U)$$
is co-monadic.
\end{cor}

\section{Algebraic stacks: criteria for 1-affineness}  \label{s:crit aff}

In this section we let $\CY$ be a quasi-compact algebraic stack, which  
is \emph{passable} as a prestack (see \secref{ss:passable} for what this means).
Note that in view of \corref{c:Loc and pass}, the assumption that $\CY$ be
passable is not very restrictive.

\medskip

We will give a series of equivalent conditions for $\CY$ to be 1-affine. 

\ssec{The 1st criterion for 1-affineness}

\sssec{}

Note that under our assumptions on $\CY$, the functor $\bGamma_\CY^{\on{enh}}$ is fully faithful, by
\propref{p:alg}. Hence, the question of 1-affineness for $\CY$ is equivalent to that of fully faithfulness of the
functor $\bLoc_\CY$. 

\medskip

We are going to prove:

\begin{prop} \label{p:cond for Gamma}
The following conditions are equivalent:

\smallskip

\noindent{\em(i)} $\CY$ is $1$-affine. 

\smallskip

\noindent{\em(ii)} The functor $\bLoc_\CY$ is fully faithful;

\smallskip

\noindent{\em(iii)} The functor $\bLoc_\CY$ is conservative; 

\smallskip

\noindent{\em(iv)} The functor $\bGamma^{\on{enh}}_{\CY}$ commutes with tensor products by objects
of $\StinftyCat_{\on{cont}}$. 

\smallskip

\noindent{\em(v)} The unit morphism $\on{Id}\to \bGamma^{\on{enh}}(\CY,\bLoc_\CY(-))$
is an equivalence on objects of the form 
$$\QCoh(\CY)\otimes \bD\in \QCoh(\CY)\mmod, \quad  \bD\in \StinftyCat_{\on{cont}}.$$

\end{prop}

\begin{proof}

Since $\bGamma_\CY$ is fully faithful, the equivalence of (i), (ii) and (iii) is evident. It is also clear that (i) implies (iv)
and that (iv) implies (v). 

\medskip

Let assume (v) and deduce (ii). We need to show that for $\bC\in \QCoh(\CY)\mmod$, the unit of the adjunction
$$\bC\to \bGamma^{\on{enh}}(\CY,\bLoc_\CY(\bC))$$
is an equivalence. Since $\QCoh(\CY)$ is rigid, an arbitrary object of $\QCoh(\CY)\mmod$
can be written as a \emph{limit} of objects of the form $\QCoh(\CY)\otimes \bD$ for $\bD\in \StinftyCat_{\on{cont}}$,
see \corref{c:any object as limit}. 

\medskip

Now, the assertion follows from the fact that both $\bGamma^{\on{enh}}_{\CY}$ 
and $\bLoc_\CY$ commute with limits: the former because $\bGamma_\CY^{\on{enh}}$ admits a left adjoint, and the latter
by \corref{c:Loc commutes with limits rigid}(b).

\end{proof}

\ssec{The 2nd criterion for 1-affineness}

We shall now give another set of criteria for $\CY$ to be $1$-affine. The idea of this approach goes back to Jacob 
Lurie. 

\medskip

We shall consider another functor 
$$\on{co-}\bGamma^{\on{enh}}(\CY,-):\on{ShvCat}(\CY)\to \QCoh(\CY)\mmod.$$

\sssec{} 

Choose a fppf cover $U\to \CY$, where $U$ is a quasi-compact and quasi-separated algebraic space.

\medskip

For $\CC\in \on{ShvCat}(\CY)$ we consider a \emph{simplicial} object 
$$\bGamma(U^\bullet/\CY,\CC)_*\in \QCoh(\CY)\mmod.$$ 

Namely, for $[i]\in \bDelta$, the category of $i$-simplices in 
$\bGamma(U^i/\CY,\CC)$. For a map $[j]\to [i]$ in $\bDelta$, the
corresponding functor
$$\bGamma(U^i/\CY,\CC)\to \bGamma(U^j/\CY,\CC)$$
is $(f^\alpha)_{\CC,*}$, see \secref{ss:dir image}.

\sssec{} 

We set
$$\on{co-}\bGamma_U(\CY,\CC):=|\bGamma(U^\bullet/\CY,\CC)_*|.$$

\medskip

Note that the simplicial object $\bGamma(U^\bullet/\CY,\CC)_*$ of $\StinftyCat_{\on{cont}}$ naturally upgrades to one in $\QCoh(\CY)\mmod$;
we will denote it by $\bGamma^{\on{enh}}(U^\bullet/\CY,\CC)_*$ 

\medskip

We denote the resulting functor $\on{ShvCat}(\CY)\to \QCoh(\CY)\mmod$ by $\on{co-}\bGamma^{\on{enh}}_U(\CY,\CC)$.

\medskip


\medskip

Note that the functors  $\on{co-}\bGamma_U(\CY,\CC)$ and $\on{co-}\bGamma^{\on{enh}}_U(\CY,\CC)$ commute with 
colimits and tensor products by objects of $\StinftyCat_{\on{cont}}$, by construction. 

\begin{rem}
The functor
$$\CC\mapsto \on{co-}\bGamma^{\on{enh}}_U(\CY,\CC)$$
a priori depends on the choice of the cover. However, it follows from \propref{p:co condition}
below (specifically, from the fact that \propref{p:co condition}(b) holds for quasi-compact quasi-separated algebraic spaces) that this definition 
can be rewritten invariantly as 
$$\underset{X\in (\on{AlgSpc}_{\on{qc,qs}})_{/\CY}}{\underset{\longrightarrow}{colim}}\, \bGamma(X,\CC)\simeq 
\underset{S\in \affdgSch_{/\CY}}{\underset{\longrightarrow}{colim}}\, \bGamma(S,\CC).$$
So, in fact, we have a well-defined functor
$$\on{co-}\bGamma^{\on{enh}}_U(\CY,-):\on{ShvCat}(\CY)\to \QCoh(\CY)\mmod.$$
\end{rem}

\sssec{} \label{sss:cosections}

We claim now that there exists a canonically defined natural transformation
\begin{equation} \label{e:co and sections}
\on{co-}\bGamma_U(\CY,-)\to \bGamma(\CY,-).
\end{equation}

Indeed, if $f^i$ denotes the morphism $U^i\to \CY$, the corresponding compatible family of functors
$$\bGamma(U^i,\CC)\to \bGamma(\CY,\CC)$$
is given by $(f^i)_{\CC,*}$. 

\medskip

The natural transformation \eqref{e:co and sections} can be interpreted in the framework of the following general
paradigm (specifically, it is a particular case of the map \eqref{e:from coinv to inv} below):

\medskip

Let $\bC^\bullet$ be a co-simplicial category, in which all functors admit right adjoints. 
Let $\bC^{\bullet,R}$ be the simplicial category obtained by passing to the right adjoint
functors. Then each of the evaluation functors
$$\on{ev}^i:\on{Tot}(\bC^\bullet)\to \bC^i$$
admits a right adjoint, and these right adjoints together define a functor
\begin{equation} \label{e:from coinv to inv}
|\bC^{\bullet,R}|\to \on{Tot}(\bC^\bullet).
\end{equation} 

Note also that by replacing the word ``right'' by ``left" we obtain a functor
$$|\bC^{\bullet,L}|\to \on{Tot}(\bC^\bullet),$$
which is an equivalence by \cite[Lemma 1.3.3]{DGCat}.

\medskip

In the above constructions, the category of indices $\bDelta$ can be replaced by any
other index category.

\sssec{}

The natural transformation \eqref{e:co and sections} upgrades to 
\begin{equation} \label{e:co and sections enh}
\on{co-}\bGamma^{\on{enh}}_U(\CY,-)\to \bGamma^{\on{enh}}(\CY,-).
\end{equation}

\medskip

In particular, by evaluating on $\CC:=\QCoh_{/\CY}$,  we obtain a functor
\begin{equation} \label{e:QCoh as colimit}
|\QCoh(U^\bullet/\CY)_*|\to \QCoh(\CY).
\end{equation}

\sssec{}

We claim:

\begin{prop}  \label{p:co condition}
The following conditions are equivalent:  

\smallskip

\noindent{\em(a)} $\CY$ is $1$-affine. 

\smallskip

\noindent{\em(b)} The natural transformation in \eqref{e:co and sections} is an isomorphism.

\smallskip

\noindent{\em(c)} The functor $\on{co-}\bGamma^{\on{enh}}_U(\CY,-)$ is \emph{a} left inverse of $\bLoc_\CY$.

\smallskip

\noindent{\em(d)} The functor in \eqref{e:QCoh as colimit} is an equivalence. 

\smallskip

\noindent{\em(e)} There exists \emph{an} isomorphism $|\QCoh(U^\bullet/\CY)_*|\simeq \QCoh(\CY)$
as $\QCoh(\CY)$-module categories. 

\smallskip

\noindent{\em(f)} The category $|\QCoh(U^\bullet/\CY)_*|$ is dualizable.

\end{prop}

\begin{proof}

First note that following tautological implications: (b) $\Rightarrow$ (d), (d) $\Rightarrow$ (e), and
(e) $\Rightarrow$ (f).

\medskip

The implication (b) $\Rightarrow$ (a) follows from the implication (iv) $\Rightarrow$ (i) in \propref{p:cond for Gamma}.
The implication (c) $\Rightarrow$ (a) follows from the implication (iii) $\Rightarrow$ (i) in \propref{p:cond for Gamma}.
Clearly, (a) and (b) together imply (c); hence (b) implies (c). 

\medskip

Let us show that (e) implies (c). Given $\bC\in \QCoh(\CY)\mmod$,
we have:
$$\on{co-}\bGamma^{\on{enh}}_U(\CY,\bLoc_\CY(\bC))\simeq \on{co-}\bGamma^{\on{enh}}_U(\CY,\QCoh_{/\CY})\underset{\QCoh(\CY)}\otimes \bC,$$
which identifies with $\bC$ by assumption.

\medskip

Let us now show that (a) implies (d). The assumption in (a) implies that $\bLoc_\CY$ is conservative, i.e., it suffices to show
that the map
$$\bLoc_\CY\left(\on{co-}\bGamma^{\on{enh}}_U(\CY,\QCoh_{/\CY})\right)\to \QCoh_{/\CY}$$
is an isomorphism. By \lemref{l:shv via Cech},
it suffices to show that the map
$$\bLoc_\CY\left(\on{co-}\bGamma^{\on{enh}}_U(\CY,\QCoh_{/\CY})\right)|_{U}\to \QCoh_{/U}$$
is an isomorphism. Since $U$ is 1-affine, it suffices to show that
$$\bGamma\left(U,\bLoc_\CY\left(\on{co-}\bGamma^{\on{enh}}_U(\CY,\QCoh_{/\CY})\right)|_{U}\right)\to \QCoh(U)$$
is an isomorphism. Using \propref{p:almost commute 1}(a,ii), we have:
$$\bGamma\left(U,\bLoc_\CY\left(\on{co-}\bGamma^{\on{enh}}_U(\CY,\QCoh_{/\CY})\right)|_{U}\right)\simeq
\QCoh(U)\underset{\QCoh(\CY)}\otimes \on{co-}\bGamma^{\on{enh}}_U(\CY,\QCoh_{/\CY}),$$
and the latter identifies with
$$|\QCoh(U)\underset{\QCoh(\CY)}\otimes \QCoh(U^\bullet/\CY)_*|.$$

However, the simplicial category $\QCoh(U)\underset{\QCoh(\CY)}\otimes \QCoh(U^\bullet/\CY)_*$ 
identifies with $$\QCoh(U^{\bullet+1}/\CY)_*,$$ (e.g., because $\CY$ is passable), and hence is split by $\QCoh(U)$. This implies
the required assertion.

\medskip

Thus, we obtain that (a) $\Leftrightarrow$ (c) $\Leftrightarrow$  (d) $\Leftrightarrow$ (e). Let us show that these
conditions imply (b). It is sufficient to evaluate both sides of \eqref{e:co and sections} on objects of the form $\bLoc_\CY(\bC)$
for $\bC\in \QCoh(\CY)\mmod$.  By (a), both sides in \eqref{e:co and sections} commute with colimits and 
tensor products by objects of $\StinftyCat_{\on{cont}}$; hence, we can take $\bC=\QCoh(\CY)$. In this case,
the required isomorphism is supplied by (d). 

\medskip

Finally, let us show that (f) implies (e). This follows from the fact that
$$\on{Funct}_{\on{cont}}(|\QCoh(U^\bullet/\CY)_*|,\Vect)\simeq \on{Tot}(\QCoh(U^\bullet/\CY)^*)\simeq
\QCoh(\CY),$$
while $\QCoh(\CY)$ is its own dual. 

\end{proof}

\ssec{The 3rd criterion for 1-affineness}  \label{ss:3rd crit}

We will now give an explicit criterion for condition (d) of \propref{p:co condition} to hold. This will
provide a 3rd criterion for 1-affiness of an algebraic stack. 

\sssec{}

We define the \emph{co-simplicial} object of $\StinftyCat$, denoted $\QCoh(U^\bullet/\CY)^?$ as follows.
For $[i]\in \bDelta$, the category of $i$-simplices is $\QCoh(U^i/\CY)$. For a map $[j]\to [i]$, the
corresponding functor
$$\QCoh(U^j/\CY)\to \QCoh(U^i/\CY)$$
is $(f^\alpha)^?$, \emph{right} adjoint to $f^\alpha_*:\QCoh(U^i/\CY)\to \QCoh(U^j/\CY)$. 

\medskip

We note that the functors $(f^\alpha)^?$ are typically \emph{non-continuous}, so $\QCoh(U^\bullet/\CY)^?$
is a co-simplicial object of $\StinftyCat$, but not in $\StinftyCat_{\on{cont}}$.

\medskip

Note that by \cite[Lemma 1.3.3]{DGCat}, we have a canonical equivalence 
\begin{equation} \label{e:lim and colim}
|\QCoh(U^\bullet/\CY)_*|\to \on{Tot}\left(\QCoh(U^\bullet/\CY)^?\right).
\end{equation}

\sssec{}

We now claim that the co-simplicial category $\QCoh(U^\bullet/\CY)^?$ satsifies the \emph{monadic}
Beck-Chevalley condition (see \secref{ss:BC} for what this means). Indeed, this follows by
applying 
\lemref{l:double right adjoint new} to the simplicial category $\QCoh(U^\bullet/\CY)_*$.

\medskip

Hence, from  from \lemref{l:monadicity}, we obtain:

\begin{lem} \hfill  \label{l:BC?}

\smallskip

\noindent{\em(a)} 
The functor of evaluation on $0$-simplices 
$$\on{ev}^0:\on{Tot}\left(\QCoh(U^\bullet/\CY)^?\right)\to \QCoh(U)$$
admits a left adjoint; to be denoted $(\on{ev}^0)^L$.

\smallskip

\noindent{\em(b)} The monad $\on{Av}^{U/\CY}_?:=\on{ev}^0\circ (\on{ev}^0)^L$, acting on $\QCoh(U)$, is isomorphic,
as a plain endo-functor, to $(\on{pr}_2)_*\circ (\on{pr}_1)^?$.

\smallskip

\noindent{\em(c)} The functor 
$$\on{ev}^0:\on{Tot}\left(\QCoh(U^\bullet/\CY)^?\right)\to \QCoh(U)$$
is monadic, i.e., the natural functor 
$$\on{Tot}\left(\QCoh(U^\bullet/\CY)^?\right)\to 
\on{Av}^{U/\CY}_?\mod(\QCoh(U))$$
is an equivalence. 

\end{lem}

\sssec{}

Consider now the functor $f_*:\QCoh(U)\to \QCoh(\CY)$. Let $f^?$ denote its \emph{right adjoint}. Consider the monad
$f^?\circ f_*$ acting on $\QCoh(U)$, and the corresponding functor
\begin{equation} \label{e:? via monad}
(f^?)^{\on{enh}}:\QCoh(\CY)\to (f^?\circ f_*)\mod(\QCoh(U)).
\end{equation}

\medskip

Note now that ?-pullback defines a functor
$$\QCoh(\CY)\to  \on{Tot}\left(\QCoh(U^\bullet/\CY)^?\right).$$
which under the equivalence of \eqref{e:lim and colim}, identifies with the right adjoint 
to the functor
$$|\QCoh(U^\bullet/\CY)_*|\to \QCoh(U)$$
of \eqref{e:QCoh as colimit}.

\medskip

Note also that the following diagram tautologically commutes:
\begin{equation} \label{e:factoring through monad 1}
\xy
(-30,0)*+{\QCoh(\CY)}="A";
(30,0)*+{\on{Tot}(\QCoh(U^\bullet/\CY)^?)}="B";
(0,-30)*+{\QCoh(U).}="C";
{\ar@{->} "A";"B"};
{\ar@{->}^{\on{ev}^0} "B";"C"};
{\ar@{->}_{f^?} "A";"C"};
\endxy
\end{equation}

Hence, we obtain a homomorphism of monads acting on $\QCoh(U)$
\begin{equation} \label{e:map of monads}
\on{Av}^{U/\CY}_?:=\left((\on{ev}^0)^L\circ \on{ev}^0\right)\to (f^?\circ f_*).
\end{equation}

\begin{lem}
The map \eqref{e:map of monads} is an isomorphism.
\end{lem}

\begin{proof}

It is enough to show that the map in question is an isomorphism as plain endo-functors. 
Using \lemref{l:BC?}(b), the map \eqref{e:map of monads} identifies with
$$(\on{pr}_2)_*\circ (\on{pr}_1)^?\to f^?\circ f_*,$$
which is obtained by passing to right adjoints from the base change map
$$f^*\circ f_*\to (\on{pr}_1)_*\circ (\on{pr}_2)^*,$$
while the latter is an isomorphism.

\end{proof}

\sssec{}

Hence, in view of \lemref{l:BC?}(c), we can identify the diagram \eqref{e:factoring through monad 1} with 
\begin{equation} \label{e:factoring through monad 2}
\xy
(-30,0)*+{\QCoh(\CY)}="A";
(30,0)*+{(f^?\circ f_*)\mod(\QCoh(U))}="B";
(0,-30)*+{\QCoh(U).}="C";
{\ar@{->}^{(f^?)^{\on{enh}}} "A";"B"};
{\ar@{->}"B";"C"};
{\ar@{->}_{f^?} "A";"C"};
\endxy
\end{equation}

From here we obtain:

\begin{prop}  \label{p:? criterion} 
The following conditions are equivalent:  

\smallskip

\noindent{\em(1)} $\CY$ is $1$-affine. 

\smallskip

\noindent{\em(2)} The functor $|\QCoh(U^\bullet/\CY)_*|\to \QCoh(\CY)$ is an equivalence.

\smallskip

\noindent{\em(3)} The functor $\QCoh(\CY)\to  \on{Tot}\left(\QCoh(U^\bullet/\CY)^?\right)$ is an equivalence.

\smallskip

\noindent{\em(4)} The functor $f^?:\QCoh(\CY)\to \QCoh(U)$ 
is monadic.

\end{prop}

\section{Classifying stacks of algebraic groups}  \label{s:BG}

The goal of this section is to prove Theorems \ref{t:BG}, \ref{t:quotient} and \ref{t:thick}.

\ssec{Reduction steps}   

Let $G$ be a classical affine algebraic group of finite type.

\sssec{}  \label{sss:global quotient}

First, let us note that \thmref{t:BG} implies \thmref{t:quotient}. Indeed, apply \corref{c:1-affine base and fiber} to
the morphism
$$Z/G\to \on{pt}/G.$$

\sssec{}

Let $G_1\hookrightarrow G_2$ be a closed embedding. Note that the corresponding map
$$\on{pt}/G_1\to \on{pt}/G_2$$
is schematic, quasi-compact and quasi-separated. 

\medskip

Hence, by \corref{c:1-affine base and fiber}, if $G_2$ is such that $\on{pt}/G_2$ is 1-affine, the same
will be true for $\on{pt}/G_1$.

\sssec{}

Choose a closed embedding $G\hookrightarrow GL_n$. We obtain that in order to prove \thmref{t:BG},
it is enough to consider the case of $G=GL_n$. I.e., we can assume that $G$ is reductive. 

\ssec{Proof of \thmref{t:BG} in the reductive case}  \label{ss:reductive}

The idea of the proof belongs to Jacob Lurie. 

\sssec{}  

Note that for $\CY=\on{pt}/G$, the category $\QCoh(\CY)$ identifies with $\Rep(G)$,
i.e., the category
of $G$-representations. Under this identification, the functor
$$f^*:\QCoh(\on{pt}/G)\to \QCoh(\on{pt})\simeq \Vect,$$
corresponding to $f:\on{pt}\to \on{pt}/G$, is the forgetful functor $\oblv_G:\Rep(G)\to \Vect$. 

\medskip

The right
adjoint $f_*$ of $f^*$ is the (usual) functor of co-induction
$$\on{coind}_G:\Vect\to \Rep(G),$$ 
\emph{right} adjoint to the forgetful functor $\oblv_G$. 

\sssec{}  \label{sss:monadicity of coinduction}

By \propref{p:? criterion}, it is enough to show that the functor
$$(\coind_G)^R:\Rep(G)\to \Vect$$
is monadic, where $(\coind_G)^R$ is the (discontinuous) \emph{right} adjoint of $\coind_G$.

\begin{rem} \label{r:usual coinduction}
Note that, according to \propref{p:? criterion}, the assertion of \thmref{t:BG} is equivalent to the fact that the functor $(\coind_G)^R$
is monadic for any classical affine algebraic group of finite type (i.e., not necessarily reductive).
\end{rem}

\sssec{}

Let $A$ be the set of irreducible representations of $G$. Since $G$ is reductive (and we are over a field
of characteristic $0$), choosing representatives, we obtain an equivalence
$$\Rep(G)\simeq \Vect^A,$$
where $\Vect^A$ is the product of copies of $\Vect$, indexed by the set $A$.

\medskip

The functor
$$\coind_G:\Vect\to \Rep(G)$$ 
is the functor
$$\sS:\Vect\to \Vect^A,\quad V\rightsquigarrow V^A\in \Vect^A,$$
(i.e., $V$ in each component).

\medskip

The right adjoint functor $(\coind_G)^R$ is
$$\sT:\Vect^A\to \Vect,\quad \{W_a,a\in A\}\in \Vect^A \rightsquigarrow \underset{a\in A}\Pi\, W_a\in \Vect.$$ 

\sssec{}

\medskip

We recall that a DG category $\bC$ equipped with a t-structure is said to be \emph{right-complete} with respect to
this t-structure if the functor
$$\bC\to \underset{n\in \BZ^+}{\underset{\longleftarrow}{lim}}\, \bC^{\leq n},  \quad \bc\mapsto \{\tau^{\leq n}(\bc)\}$$
is an equivalence. 

\medskip

If this happens, the inverse equivalence is given by
$$\{\bc_n\in \bC^{\leq n}\}\mapsto  \underset{n\in \BZ^+}{\underset{\longrightarrow}{lim}}\, \bc_n.$$

\medskip

Recall also $\bC$ is said to be \emph{left-complete} with respect to its t-structure if the functor
$$\bC\to \underset{n\in \BZ^+}{\underset{\longleftarrow}{lim}}\, \bC^{\geq -n},  \quad \bc\mapsto \{\tau^{\geq -n}(\bc)\}$$
is an equivalence. 

\medskip

If this happens, the inverse equivalence is given by
$$\{\bc_n\in \bC^{\geq -n}\}\mapsto  \underset{n\in \BZ^+}{\underset{\longleftarrow}{lim}}\, \bc_n.$$

\sssec{}

Note that both categories $\Vect$ and $\Vect^A$ carry t-structures, in which they are both
right-complete and left-complete. 

\medskip

The functors $\sS$ and $\sT$ are t-exact. In particular, they define a pair of adjoint functors
$$\sS^{\leq n}:\Vect^{\leq n}\rightleftarrows (\Vect^A)^{\leq n}:\sT^{\leq n}$$
for every $n$. 

\sssec{}

Consider the following general paradigm. Let $I$ be an index category, and let 
$$i\rightsquigarrow \bC_i \text{ and } i\rightsquigarrow \bD_i$$
be two family of categories. Let
$$\bC:=\underset{i\in I}{lim}\, \bC_i  \text{ and } \bD:=\underset{i\in I}{lim}\, \bD_i$$
be the limits.

\medskip

Let us be given a compatible system of adjoint functors
$$\sS_i:\bC_i\rightleftarrows \bD_i:\sT_i.$$

Denote by
$$\sS:\bC\rightleftarrows \bD:\sT$$
the resulting adjoint pair.

\medskip

We have the following general lemma:

\begin{lem}  \label{l:monads in limit}
Suppose that for every $i$, the pair $\sS_i:\bC_i\rightleftarrows \bD_i:\sT_i$ is monadic. 
Then the pair $\sS:\bC\rightleftarrows \bD:\sT$ is also monadic.
\end{lem}

\sssec{}

Applying \lemref{l:monads in limit}, we obtain that it suffices to show that the pair of adjoint functors
$$\sS^{\leq n}:\Vect^{\leq n}\rightleftarrows (\Vect^A)^{\leq n}:\sT^{\leq n}$$
is monadic. With no restriction of generality, we can assume that $n=0$. 

\medskip

We will prove that the pair 
$$\sS^{\leq 0}:\Vect^{\leq 0}\rightleftarrows (\Vect^A)^{\leq 0}:\sT^{\leq 0}$$
is monadic by verifying that the functor $\sT^{\leq 0}$ 
satsifies the conditions of the Barr-Beck-Lurie theorem (\cite[Theorem 6.2.2.5]{Lu2}). 

\medskip

The fact that $\sT$ is conservative is manifest from the explicit
description of the functor in question. 

\medskip

The fact that the functor $\sT$, restricted to $(\Vect^A)^{\leq 0}$, commutes
with $\sG$-split geometric realizations follows from the next general assertion:

\begin{lem} \label{l:bounded}
Let $\sT:\bC_1\to \bC_2$ be a functor between DG categories. Assume that both
categories are equipped with a t-structure and that $\sT$ sends $\bC^{\leq 0}_1$
to $\bC^{\leq k}_2$ for some $k$. Assume that $\bC_2$ is \emph{left}-complete in its
t-structure. Then $\sT$ commutes with all geometric realizations
of objects in $\bC_1^{\leq 0}$.
\end{lem}

\ssec{Classifying stacks of group-schemes of infinite type}  \label{ss:pro infty}

In this subsection we will prove \thmref{t:thick}.

\sssec{}

Let $G$ be the affine group-scheme $\underset{n}{\underset{\longleftarrow}{lim}}\, (\BG_a)^{\times n}$. 
Note that the map
\begin{equation} \label{e:BG and pt/G}
BG\to \on{pt}/G
\end{equation}
is an isomorphism (indeed, on an affine DG scheme $S$ there are no non-trivial $G$-torsors). 

\sssec{}

Note that since $G$ is of infinite type, the map $\on{pt}\to \on{pt}/G$ is not an fppf cover
(it is an fpqc cover). However, since \eqref{e:BG and pt/G} is an isomorphism, the map
$$\on{ShvCat}(\on{pt}/G)\to \on{ShvCat}(BG)\simeq 
\on{Tot}(\on{ShvCat}(B^\bullet G)=\on{Tot}\left(\QCoh(B^\bullet G)\mmod\right)$$
is an equivalence.

\medskip

Denote 
$$\Rep(G):=\QCoh(\on{pt}/G).$$

Let $f$ denote the tautological morphism $\on{pt}\to \on{pt}/G$. Set
$$\on{oblv}_G:=f^*:\Rep(G)\to \Vect, \quad \on{coind}_G:=f_*:\Rep(G)\to \Vect.$$

\medskip

Suppose, for the sake of contradiction that $\on{pt}/G$ was 1-affine. Then by the same logic as
in \secref{sss:monadicity of coinduction}, we would obtain that the functor $(\on{coind}_G)^R$,
right adjoint to  $\on{coind}_G$ would be monadic. 

\medskip

However, we claim that $(\on{coind}_G)^R$ fails to be conservative:

\sssec{}

Note that the functor $\on{coind}_G$
sends $k\in \Vect$ to the regular representation $\CO_G\in \Rep(G)$. We claim
that 
$$\CMaps_{\Rep(G)}(\CO_G,k)=0,$$
where $k\in \Rep(G)$ is the trivial representation. 

\medskip

Indeed, if $G=\Spec(\on{Sym}(W))$, where $W$ is a countable-dimensional vector space, then the object
$k\in \Rep(G)$ admits a resolution whose $n$-term is $\on{coind}_G(\Lambda^n(W))$. 

\medskip

Hence,
$\CMaps_{\Rep(G)}(\CO_G,k)$ is computed by the complex whose $n$-th term is
$$\Hom_{\Vect^\heartsuit}(\on{Sym}(W),\Lambda^n(W)),$$
which is easily seen to be acyclic. 

\qed

\section{Algebraic stacks: proof of \thmref{t:alg}}  \label{s:stacks}

Let $\CY$ be as in \thmref{t:alg}.  I.e., $\CY$ is a quasi-compact algebraic stack, locally almost of 
finite type, which is eventually coconnective and has an affine diagonal. 

\medskip

We know that $\CY$ is passable by \thmref{t:ev coconn pass}.
We will prove that $\CY$ is 1-affine by verifying condition (4) of \propref{p:? criterion}.

\ssec{Strategy}

Let $f:U\to \CY$ be an smooth cover, where $U$ is an affine DG scheme. We consider the functor
$$f^?:\QCoh(\CY)\to \QCoh(U),$$
and the resulting monad $f^?\circ f_*$ acting on $\QCoh(U)$. We denote 
the resulting pair of adjoint functors by
$$(f_*)^{\on{enh}}:(f^?\circ f_*)\mod(\QCoh(U))\rightleftarrows \QCoh(\CY):(f^?)^{\on{enh}}.$$

\medskip

We will deduce \thmref{t:alg} from the combination of the following two statements:

\begin{prop}  \label{p:generation} \hfill

\smallskip

\noindent{\em(a)} The functor $f^?$ is conservative.

\smallskip

\noindent{\em(b)} The functor $(f_*)^{\on{enh}}$ is conservative. 

\end{prop}

\begin{prop}  \label{p:bdd dim}
There exists a constant $n$ that depends only on $\CY$, such that for any flat map $f:U\to \CY$
with $U\in \affdgSch$, the functor 
$$f^?:\QCoh(\CY)\to \QCoh(U),$$
right adjoint to $f_*$, has a cohomological amplitude bounded on the right by $n$.
\end{prop}

\ssec{Proof of \thmref{t:alg}}

Let us assume both Propositions \ref{p:generation} and \ref{p:bdd dim} and deduce \thmref{t:alg}. 

\sssec{}

Let $f:U\to \CY$ be an fppf cover, where $U\in \affdgSch$.

\medskip

By \propref{p:generation}(b), we only have to show that co-unit map
$$(f_*)^{\on{enh}}\circ (f^?)^{\on{enh}}\to \on{Id}_{\QCoh(\CY)}$$
is an isomorphism.

\medskip 

For an object $\CF\in \QCoh(\CY)$, the object $(f_*)^{\on{enh}}\circ (f^?)^{\on{enh}}(\CF)$ is 
the gemetric realization of the simplicial object given by
$$[i]\mapsto (f^i)_*\circ (f^i)^?(\CF),$$
where $f^i:U^i\to \CY$, and where $U^i$ is the $i$-th term of the \v{C}ech nerve of $f:U\to \CY$. 

\medskip

The map
$$|(f^\bullet)_*\circ (f^\bullet)^?(\CF)|\to \CF$$
is the natural augmentation map.

\sssec{Step 1}

We first consider the case when $\CF$ is bounded above with respect to the standard t-structure on $\QCoh(\CY)$. 
With no restriction of generality, let us assume that $\CF\in \QCoh(\CY)^{\leq 0}$. 

\medskip

Since the functor $f^?$ is conservative (by \propref{p:generation}(a)), it suffices to show that the map
$$f^?\left(|(f^\bullet)_*\circ (f^\bullet)^?(\CF)|\right)\to 
f^?(\CF)$$
is an isomorphism. 

\medskip

Consider the composition 
$$|f^?((f^\bullet)_*\circ (f^\bullet)^?(\CF))|\to f^?\left(|(f^\bullet)_*\circ (f^\bullet)^?(\CF)|\right)\to 
f^?(\CF).$$

The composed map is an isomorphism, since the simplicial object $f^?((f^\bullet)_*\circ (f^\bullet)^?(\CF))$
of $\QCoh(U)$ is split by $f^?(\CF)$. Hence, it suffices to show that the map
$$|f^?((f^\bullet)_*\circ (f^\bullet)^?(\CF))|\to f^?\left(|(f^\bullet)_*\circ (f^\bullet)^?(\CF)|\right)$$
is an isomorphism.

\medskip

Since the t-structure on $\QCoh(U)$ is left-complete, it suffices to show that the map
$$\tau^{\geq -k}\left(|f^?((f^\bullet)_*\circ (f^\bullet)^?(\CF))|\right)\to
\tau^{\geq -k}\left(f^?\left(|(f^\bullet)_*\circ (f^\bullet)^?(\CF)|\right)\right)$$
is an isomorphism for every $k\in \BZ^{\geq 0}$. 

\medskip

Let $n$ be the integer from \propref{p:bdd dim}. Consider the commutative diagram
\begin{equation} \label{e:compare trunc}
\CD
\tau^{\geq -k}\left(|f^?((f^\bullet)_*\circ (f^\bullet)^?(\CF))|\right)   @>>>  
\tau^{\geq -k}\left(f^?\left(|(f^\bullet)_*\circ (f^\bullet)^?(\CF)|\right)\right)  \\
@AAA   @AAA \\
\tau^{\geq -k}\left(|f^?((f^\bullet)_*\circ (f^\bullet)^?(\CF))|_{k+2n}\right)   @>>>  
\tau^{\geq -k}\left(f^?\left(|(f^\bullet)_*\circ (f^\bullet)^?(\CF)|_{k+2n}\right)\right),
\endCD
\end{equation}
where for $m\in \BZ^{\geq m}$, we denote by 
$|-|_{\leq m}$ the gemetric realization of the $m$-skeleton
(i.e., the colimit over the subcategory of $\bDelta^{\on{op}}$ corresponding to $[i]$
with $i\leq m$). 

\medskip

We claim that both vertical arrows and the bottom horizontal arrow in the diagram \eqref{e:compare trunc}
are isomorphisms. 

\medskip

The assertion regarding the bottom horizontal arrow follows from the fact that $|-|_{\leq k+2n}$
is a \emph{finite} colimit, and hence commutes with $f^?$.

\medskip

Note that since the diagonal of $\CY$ is affine, all of the maps $f^i$ are affine. In particular, each of
the functors $(f^i)_*$ is t-exact. Furthermore, by the assumption on $n$, each of the functors $(f^i)^?$ 
has a cohomological amplitude bounded on the right by $n$.

\medskip

In particular, each of the terms $(f^i)_*\circ (f^i)^?(\CF)$ lies in $\QCoh(\CY)^{\leq n}$. Hence, we obtain that
for any $k'$, the map 
$$\tau^{\geq -k'}\left(|(f^\bullet)_*\circ (f^\bullet)^?(\CF)|\right)\to
\tau^{\geq -k'}\left(|(f^\bullet)_*\circ (f^\bullet)^?(\CF)|_{\leq k'+n}\right)$$
is an isomorphism. 

\medskip

Taking $k'=k+n$, and using the fact that the cohomological amplitude of $f^?$ 
is bounded on the right by $n$, we obtain that the right vertical arrow in \eqref{e:compare trunc}
is an isomorphism.

\medskip

Similarly, the terms of the simplicial object $f^?((f^\bullet)_*\circ (f^\bullet)^?(\CF))$ lie in 
$\QCoh(\CY)^{\leq 2n}$. Hence, the left vertical arrow in \eqref{e:compare trunc}
is an isomorphism, as required.

\sssec{Step 2}

Let now $\CF$ be arbitrary. Consider the commutative diagram
\begin{equation} \label{e:trunc diag}
\CD
|(f^\bullet)_*\circ (f^\bullet)^?(\CF)|   @>>>   \CF  \\
@AAA    @AAA   \\
\underset{n}{\underset{\longrightarrow}{colim}}\, |(f^\bullet)_*\circ (f^\bullet)^?(\tau^{\leq n}(\CF))|   
@>>>  \underset{n}{\underset{\longrightarrow}{colim}}\, \tau^{\leq n}(\CF).
\endCD
\end{equation}

\medskip

We need to show that the top horizontal arrow in \eqref{e:trunc diag} is an isomorphism. We will do so by
showing that the bottom horizontal arrow, as well as the vertical arrows, are isomorphisms.

\medskip

The assertion regarding the bottom horizontal arrow follows from Step 1. The assertion regarding 
the right vertical arrow expresses the fact that the t-structure on $\QCoh(\CY)$ is right-complete. 

\medskip

To show that the left vertical arrow in \eqref{e:trunc diag} is an isomorphism, 
it suffices to show that for every $i$, the map
$$\underset{n}{\underset{\longrightarrow}{colim}}\, (f^i)_*\circ (f^i)^?(\tau^{\leq n}(\CF))\to (f^i)_*\circ (f^i)^?(\CF)$$
is an isomorphism. 

\medskip

Since the functor $(f^i)_*$ commutes with colimits, it suffices to show that the map
$$\underset{n}{\underset{\longrightarrow}{colim}}\,  (f^i)^?(\tau^{\leq n}(\CF))\to (f^i)^?(\CF)$$
is an isomorphism in $\QCoh(U^i)$. Since the t-structure on $\QCoh(U^i)$ is right-complete
and compatible with filtered colimits, it suffices to show that for every $k\in \BZ^{\geq 0}$, the map
$$\underset{n}{\underset{\longrightarrow}{colim}}\,  \tau^{\leq k}\left((f^i)^?(\tau^{\leq n}(\CF))\right)\to \tau^{\leq k}\left((f^i)^?(\CF)\right)$$
is an isomorphism. However, the map
$$\tau^{\leq k}\left((f^i)^?(\tau^{\leq n}(\CF))\right)\to \tau^{\leq k}\left((f^i)^?(\CF)\right)$$
is already an isomorphism for any $n\geq k$, since the functor $(f^i)^?$ is left t-exact (the latter is because
its left adjoint, namely $(f^i)_*$, is t-exact, and in particular, right t-exact.)

\qed

\ssec{Proof of \propref{p:bdd dim}}

\sssec{}

Let $A$ be a connective $k$-algebra. Let us recall that an object $M\in A\mod$ is said to be \emph{flat} if

\begin{itemize}

\item $M\in A\mod^{\leq 0}$;

\item $H^0(M)$ is flat as an $H^0(A)$-module;

\item the natural maps $H^{-i}(A)\underset{H^0(A)}\otimes H^0(M)\to H^{-i}(M)$ are isomorphisms. 

\end{itemize}

For a prestack $\CY$ and $\CF\in \QCoh(\CY)$, we shall say that $\CF$ is flat if its pullback to
any affine scheme is flat. 

\medskip

Note that the assumption of the proposition implies that the object $f_*(\CO_U)\in \QCoh(\CY)$ is flat.
Hence, the assertion of \propref{p:bdd dim} follows from the next general result:

\begin{prop} \label{p:flat bdd}
Let $\CY$ be a QCA\footnote{See \cite[Definition 1.1.8]{DrGa} for what this means.}
stack, locally almost of finite type. There exists an integer $n$, such
that for any flat $\CE\in \QCoh(\CY)$, the functor 
$$\CMaps_{\QCoh(\CY)}(\CE,-):\QCoh(\CY)\to \Vect$$
has a cohomological amplitude bounded on the right by $n$.
\end{prop}

The proof of this proposition will occupy the rest of this subsection.

\sssec{Step 1}

We claim that it is sufficient to show that there exists an integer $n$ such that 
$$\Hom_{\QCoh(\CY)}(\CE,\CF[i])=0 \text{ for all } i>n \text{ and } \CF\in \QCoh(\CY)^\heartsuit.$$

The proof is the same as that of \cite[Lemma 2.1.3]{DrGa}. 

\medskip

As in \cite[Sect. 2.2.1]{DrGa}, this allows to assume that $\CY$ is classical:
if $n$ is the integer that works for $^{cl}\CY$, then it also works for $\CY$. 

\medskip

Since $\CY$ is Noetherian, we can further replace $^{cl}\CY$ by $({}^{cl}\CY)_{red}$: if an integer
$n$ works for $({}^{cl}\CY)_{red}$, then it works also for $^{cl}\CY$. 

\sssec{Step 2}

Recall the setting and notations of \secref{sss:with supports}. 

\medskip

We consider the (discontinuous) functor 
$$\wh\imath^{\QCoh}_?:\QCoh(\CY)_{\CY'}\to \QCoh(\CY),$$
\emph{right} adjoint to $\wh\imath^{\QCoh}_!$.

\medskip

Let 
$$\jmath^?:\QCoh(\CY)\to \QCoh(\CY_0)$$ denote the (discontinuous) right adjoint of $\jmath_*$. For 
$\CF\in \QCoh(\CY)$ we have a distinguished triangle
\begin{equation} \label{e:strange triangle}
\jmath_*\circ \jmath^?(\CF)\to \CF\to \wh\imath^{\QCoh}_?\circ \wh\imath^{\QCoh,!}(\CF).
\end{equation}

\medskip

We will prove:

\begin{lem} \label{l:on completion} Let $\CY$ satisfy the assumption of \propref{p:flat bdd}, and assume that
$\CY$ is classical. 

\smallskip

\noindent{\em(a)}
The functor $\wh\imath^{\QCoh}_?\circ \wh\imath^{\QCoh,!}$ is right t-exact. 

\smallskip

\noindent{\em(b)}
Assume that $\CY'$ satisfies the conclusion of \propref{p:flat bdd} with an integer $n'$. Then
for $\CF\in \QCoh(\CY)^{\leq 0}$, we have
$$\CMaps_{\QCoh(\CY)}\left(\CE,\wh\imath^{\QCoh}_?\circ \wh\imath^{\QCoh,!}(\CF)\right)\in \Vect^{\leq n'+1}.$$

\end{lem}

\sssec{Step 3}

Let us assume \lemref{l:on completion} and finish the proof of \propref{p:flat bdd}. 

\medskip

By Step 1, we can assume that $\CY$ is classical and reduced. 
We can choose a closed substack $\CY'\subset \CY$ such that the complementary open $\CY_0$ is smooth.  
In Step 4 we will show that a smooth QCA stack satisfies the conclusion of \propref{p:flat bdd}. 
Let $n_0$ denote the corresponding integer. 

\medskip

We claim that $n:=\on{max}(n',n_0)+1$ will work for $\CY$. Indeed, let
$\CF$ be an object $\QCoh(\CY)^{\leq 0}$, and let us show that
$$\CMaps_{\QCoh(\CY)}(\CE,\CF)\in \Vect^{\leq n+1}.$$

\medskip

By \eqref{e:strange triangle} and \lemref{l:on completion}(b), it suffices to show that
$$\CMaps_{\QCoh(\CY)}\left(\CE,\jmath_*\circ \jmath^?(\CF)\right)\simeq
\CMaps_{\QCoh(\CU)}\left(\jmath^*(\CE),\jmath^?(\CF)\right)\in \Vect^{\leq n_0+1}.$$

\medskip

By the construction of $n_0$, it suffices to show that
$\jmath^?(\CF)\in \QCoh(\CU)^{\leq 1}$. Since $\jmath^*\circ \jmath_*\simeq \on{Id}_{\QCoh(\CU)}$
and $\jmath^*$ is t-exact, it suffices to show that 
$$\jmath_*\circ \jmath^?(\CF)\in \QCoh(\CY)^{\leq 1}.$$

However, the latter follows from \lemref{l:on completion}(a).

\sssec{Step 4}

Assume that $\CY$ is a smooth QCA stack. Let us prove \propref{p:flat bdd} directly in this case. We claim that
in this case the category $\QCoh(\CY)$ is of bounded cohomological dimension. 

\medskip

Indeed, as in Step 1, it is sufficient to show that there exists an integer $n$ such that
$$\Hom_{\QCoh(\CY)}(\CF',\CF[i])=0 \text{ for all } i>n \text{ and } \CF\in \QCoh(\CY)^\heartsuit, \CF'\in \Coh(\CY)^\heartsuit.$$

\medskip

Since $\CY$ is smooth and quasi-compact, there exists an integer $n'$, such that 
any $\CF'\in \QCoh(\CY)^\heartsuit$ admits a resolution of length $n'$ consisting of locally free $\CO_\CY$-modules
of finite rank. Hence,
it suffices to show that there exists an integer $n''$ such that
$$\Hom_{\QCoh(\CY)}(\CE',\CF[i])=0 \text{ for all } i>n'' \text{ and } \CF\in \QCoh(\CY)^\heartsuit$$
for $\CE'$ locally free of finite rank (we will then set $n=n'+n''$). 

\medskip

However,
$$\Hom_{\QCoh(\CY)}(\CE',\CF[i])\simeq \Gamma(\CY,\CF'\otimes (\CE')^\vee),$$
and the existence of $n''$ follows from \cite[Theorem 1.4.2(ii)]{DrGa}. 

\qed

\ssec{Proof of \lemref{l:on completion}}

\sssec{}

Note that from \propref{p:compl and supp} we obtain: 

\begin{cor} \label{c:compl and supp}
The endo-functor $\wh\imath^{\QCoh}_?\circ \wh\imath^{\QCoh,!}$
is canonically isomorphic to
$$\wh\imath_*\circ \wh\imath^*,$$
where 
$$\wh\imath_*:\QCoh(\CY^\wedge_{\CY'})\to \QCoh(\CY)$$
is the right adjoint to $\wh\imath^*$.
\end{cor}

\sssec{}

Let $\CY'_k\overset{\imath_k}\hookrightarrow \CY$ denote the $k$-th \emph{classical} 
infinitesimal neighborhood of $\CY'$ inside $\CY$. 
The following results
from \cite[Proposition 6.8.2]{IndSch}:

\begin{lem}
The map $\underset{k}{colim}\, \CY'_k\to \CY^\wedge_{\CY'}$
becomes an isomorphism after fppf sheafification.
\end{lem}

\begin{cor} \hfill  \label{c:desc compl}

\smallskip

\noindent{\em(a)} We have a canonical equivalence
$$\QCoh(\CY^\wedge_{\CY'})\simeq \underset{k\in \BZ^+}{\underset{\longleftarrow}{lim}}\, \QCoh(\CY'_k).$$

\smallskip

\noindent{\em(b)} The endo-functor $\wh\imath_*\circ \wh\imath^*$ of $\QCoh(\CY)$ is canonically isomorphic to
$$\CF\mapsto \underset{k\in \BZ^+}{\underset{\longleftarrow}{lim}}\, (\imath_k)_*\circ (\imath_k)^*(\CF).$$

\end{cor}

\sssec{}

We are now ready to prove \lemref{l:on completion}. Point (a) of \lemref{l:on completion}.
follows immediately from Corollaries \ref{c:compl and supp}
and \ref{c:desc compl}(b).

\medskip

To prove point (b) of \lemref{l:on completion}, we note by \propref{p:compl and supp} and \corref{c:desc compl}(a), 
that for $\CE,\CF\in \QCoh(\CY)$, we have:
$$\CMaps_{\QCoh(\CY)}\left(\CE,\wh\imath^{\QCoh}_?\circ \wh\imath^{\QCoh,!}(\CF)\right)\simeq
\underset{k\in \BZ^+}{\underset{\longleftarrow}{lim}}\, \CMaps_{\CY'_k}(\imath_k^*(\CE),\imath_k^*(\CF)).$$

\medskip

Since the functor of projective limit over $\BZ^+$ in $\Vect$ has cohomological amplitude $1$, it is sufficient to
show that for $\CE$ flat and $\CF\in \QCoh(\CY)^{\leq 0}$, each term 
$$\CMaps_{\CY'_k}(\imath_k^*(\CE),\imath_k^*(\CF))$$
belongs to $\Vect^{\leq n'}$. 

\medskip

However, $\imath_k^*(\CE)$ is flat on $\CY'_k$, and $\imath_k^*(\CF)\in \QCoh(\CY'_k)^{\leq 0}$,
and the assertion follows from the definition of $n'$, combined with Step 1 of the proof of
\propref{p:flat bdd}.

\qed

\ssec{Proof of \propref{p:generation}(a)}

\sssec{Step 1}

Let $\CY'\overset{\imath}\hookrightarrow \CY$ be a closed substack, and let $\CY_0\overset{\jmath}\hookrightarrow \CY$
be a complementary open. Denote $U':=U\underset{\CY}\times \CY'$ and $U_0:=U\underset{\CY}\times \CY_0$,
and let $\imath_U$ and $\jmath_U$ denote the corresponding maps, and
$$f':U'\to \CY' \text{ and } f_0:U_0\to \CY_0.$$

\medskip

Consider the commutative diagrams
$$
\CD
\QCoh(U) @>{(\jmath_U)^*}>> \QCoh(U_0)  \\
@V{f_*}VV   @VV{(f_0)_*}V   \\
\QCoh(\CY)  @>{\jmath^*}>> \QCoh(\CY_0)
\endCD
$$
and
$$
\CD
\QCoh(U) @<{(\jmath_U)_*}<< \QCoh(U_0)  \\
@V{f_*}VV   @VV{(f_0)_*}V   \\
\QCoh(\CY)  @<{\jmath_*}<< \QCoh(\CY_0).
\endCD
$$
By passing to right adjoints we obtain the commutative diagrams
$$
\CD
\QCoh(U) @<{(\jmath_U)_*}<< \QCoh(U_0)  \\
@A{f^?}AA   @AA{f_0^?}A  \\
\QCoh(\CY)  @<{\jmath_*}<< \QCoh(\CY_0)
\endCD
$$
and
$$
\CD
\QCoh(U) @>{(\jmath_U)^?}>> \QCoh(U_0) \\
@A{f^?}AA   @AA{f_0^?}A  \\
\QCoh(\CY)  @>{\jmath^?}>> \QCoh(\CY_0).
\endCD
$$

Hence, from the exact triangle \eqref{e:strange triangle}, we obtain that $f^?$ defines a functor
$$(\wh{f}')^?:\QCoh(\CY)_{\CY'}\to \QCoh(U)_{U'}$$
that makes the diagrams
$$
\CD
\QCoh(U)_{U'}  @<{(\wh\imath_U)^{\QCoh,!}}<<  \QCoh(U) \\
@A{(\wh{f}')^?}AA    @AA{f^?}A   \\
\QCoh(\CY)_{\CY'}  @<{\wh\imath^{\QCoh,!}}<<  \QCoh(\CY)
\endCD
$$
and
$$
\CD
\QCoh(U)_{U'}  @>{(\wh\imath_U)^{\QCoh}_?}>>  \QCoh(U) \\
@A{(\wh{f}')^?}AA    @AA{f^?}A   \\
\QCoh(\CY)_{\CY'}  @>{\wh\imath^{\QCoh}_?}>>  \QCoh(\CY)
\endCD
$$
commute. 

\medskip

From the exact triangle \eqref{e:strange triangle}, we obtain that if the functors $f_0^?$ and 
$(\wh{f}')^?$ are both conservative, then so is $f^?$.

\sssec{Step 2}

We will now show that if the functor $(f')^?:\QCoh(\CY')\to \QCoh(U')$ is conservative, then so is 
$$(\wh{f}')^?:\QCoh(\CY)_{\CY'}\to \QCoh(U)_{U'}.$$

\medskip

We note that the functor $\imath_*:\QCoh(\CY')\to \QCoh(\CY)$ factors canonically as
$$\QCoh(\CY')\to \QCoh(\CY)_{\CY'}\overset{\wh\imath^{\QCoh}_!}\hookrightarrow \QCoh(\CY),$$
and the diagram
$$
\CD
\QCoh(U')  @>>>  \QCoh(U)_{U'}  @>{(\wh\imath_U)^{\QCoh}_!}>> \QCoh(U) \\
@V{f'_*}VV    @V{f_*}VV    @VV{f_*}V  \\
\QCoh(\CY')  @>>>  \QCoh(\CY)_{\CY'}  @>{\wh\imath^{\QCoh}_!}>> \QCoh(\CY)
\endCD
$$
commutes. Hence, the diagram 
$$
\CD
\QCoh(U')  @<<<  \QCoh(U)_{U'}  @<{(\wh\imath_U)^{\QCoh,!}}<< \QCoh(U) \\
@A{(f')^?}AA    @A{(\wh{f}')^?}AA    @VV{f^?}V  \\
\QCoh(\CY')  @<<<  \QCoh(\CY)_{\CY'}  @<{\wh\imath^{\QCoh,!}}<< \QCoh(\CY)
\endCD
$$
commutes, as well, where the left horizontal arrows are obtained by restricting the (discontinuous)
functor
$\imath^{\QCoh,!}$ (resp., $(\imath_U)^{\QCoh,!}$), right adjoint to $\imath_*$ (resp., $(\imath_U)_*$),
to  $\QCoh(\CY)_{\CY'}\subset \QCoh(\CY)$ (resp., $\QCoh(U)_{U'}\subset \QCoh(U)$). 

\medskip

Hence, in order to show that $(\wh{f}')^?$ is conservative, it is sufficient to show that the restriction
of the functor 
$$\imath^{\QCoh,!}: \QCoh(\CY)\to \QCoh(\CY')$$
to $\QCoh(\CY)_{\CY'}\subset \QCoh(\CY)$ is conservative. 

\medskip

I.e., we have to show that the essential image of the functor 
$$\imath_*:\QCoh(\CY')\to \QCoh(\CY)$$
generates $\QCoh(\CY)_{\CY'}$. 

\medskip

However, that latter is established in \cite[Sect. 2.6.8]{DrGa}. \footnote{Here the assumption that
$\CY$ be eventually coconnective is crucial.} 

\sssec{Step 3}

By Step 2, we can assume that $\CY$ is classical and reduced (take $\CY':=({}^{cl}\CY)_{red}$). 
By Step 1 and Noetherian induction,
it suffices to show that $\CY$ contains a non-empty open substack $\CY_0$, for which 
\propref{p:generation}(a) holds.

\medskip

Hence, by \cite[Proposition 2.3.4]{DrGa}, we can assume that $\CY$ admits a finite 
\'etale cover $\pi:\wt\CY \to \CY$, such that $\wt\CY$ is isomorphic to the quotient 
of a quasi-compact quasi-separated classical reduced scheme by an action of a classical
affine algebraic group of finite type.

\medskip

For $\CY$ of this form we will establish the assertion of \propref{p:generation}(a) directly. 
First, since the right adjoint of $\pi_*$ is isomorphic to $\pi^*$, we can replace $\CY$ by $\wt\CY$.
So, we can assume that $\CY$ itself is of the form $Y/G$, where $Y$ is a  
quasi-compact quasi-separated classical scheme and $G$ is an algebraic group of finite 
type.

\medskip

In this case, the assertion of \propref{p:generation}(a) follows from \thmref{t:quotient} and
\propref{p:? criterion}. Here is an argument independent of \thmref{t:quotient}: 

\medskip

We need to show that the essential image of the functor $f_*:\QCoh(U)\to \QCoh(\CY)$
generates $\QCoh(\CY)$, where $\CY=Y/G$. 

\medskip

With no restriction of generality, we can assume that $G$ is reductive. Let $p$ denote the
projection $Y\to Y/G$. Since $G$ is reductive, its regular representation contains the trivial
representation as a direct summand. Hence, any $\CF\in \QCoh(Y/G)$ is a direct summand
of $p_*\circ p^*(\CF)$. This reduces the assertion to the case when $Y/G$ is replaced by $Y$,
i.e., we can assume that $\CY$ is a quasi-compact quasi-separated classical reduced scheme $Y$.  
One further easily reduces to the case when $Y$ is affine. 

\medskip

Now, for a map $f:U\to Y$, where $Y$ is an affine classical reduced scheme, the assertion of \propref{p:generation}(a)
is easy: the full subcategory generated by the essential image of $f_*$ is a tensor
ideal, and since $f$ is faithfully flat, it contains the structure sheaf of the generic point of any irreducible 
subscheme of $Y$, and thus equals all of $\QCoh(Y)$. 

\medskip

Alternatively, the assertion of \propref{p:generation}(a) for any affine DG scheme follows from \propref{p:? criterion}.

\ssec{Passing from $\QCoh$ to $\IndCoh$}

In order to prove \propref{p:generation}(b), we will need to replace the categories $\QCoh(\CY)$ and $\QCoh(U)$ by
$\IndCoh(\CY)$ and $\IndCoh(U)$, respectively. The reason for this will be explained in Step 2 of the proof of
\propref{p:IndCoh generation}, which is an $\IndCoh$ version of \propref{p:generation}(b).

\sssec{}

Let us recall that for an algebraic stack $\CY$ locally almost of finite type, there exists a canonically defined
natural transformation
$$\Psi_\CY:\IndCoh(\CY)\to \QCoh(\CY),$$
see \cite[Sect. 11.7]{IndCoh}.

\medskip

Moreover, when $\CY$ is eventually coconnective, the functor $\Psi_\CY$ admits a left adjoint $\Xi_\CY$,
see \cite[Sect. 11.7.3]{IndCoh}. The interactions of the pair $(\Xi_\CY,\Psi_\CY)$ with the functors arising
from schematic maps between stacks are the same as those for maps between DG schemes, see 
\cite[Sect. 3]{IndCoh}.

\sssec{}

Consider the commutative diagram:
\begin{equation} \label{e:Psi dir image}
\CD
\QCoh(U)  @<{\Psi_U}<<  \IndCoh(U) \\
@V{f_*}VV    @VV{f^{\IndCoh}_*}V   \\
\QCoh(\CY)  @<{\Psi_\CY}<<  \IndCoh(\CY).
\endCD
\end{equation}
 
According to \cite[Proposition 3.6.7]{IndCoh}, since $f$ is fppf, the diagram
\begin{equation} \label{e:Xi dir image}
\CD
\QCoh(U)  @>{\Xi_U}>>  \IndCoh(U) \\
@V{f_*}VV    @VV{f^{\IndCoh}_*}V   \\
\QCoh(\CY)  @>{\Xi_\CY}>>  \IndCoh(\CY),
\endCD
\end{equation}
obtained from \eqref{e:Psi dir image} by passing to left adjoints along the horizontal arrows is also
commutative.

\sssec{}

Let $f^{\IndCoh,?}$ denote the (discontinuous) right adjoint to $f^{\IndCoh}_*$. By passing to right adjoints
along all arrows in \eqref{e:Xi dir image}, we obtain a commutative diagram
\begin{equation} \label{e:Psi ?}
\CD
\QCoh(U)  @<{\Psi_U}<<  \IndCoh(U) \\
@A{f^?}AA    @AA{f^{\IndCoh,?}}A   \\
\QCoh(\CY)  @<{\Psi_\CY}<<  \IndCoh(\CY).
\endCD
\end{equation}

Consider the monad $\on{Av}^{\IndCoh,U/\CY}_?:=f^{\IndCoh,?}\circ f^{\IndCoh}_*$ acting on $\IndCoh(U)$. We obtain that the functor
$\Psi_U$ intertwines the actions of the monads $\on{Av}^{\IndCoh}_?$ on $\IndCoh(U)$ and 
$\on{Av}^{U/\CY}_?:=f^?\circ f_*$ on $\IndCoh(\CY)$, respectively. 

\medskip

In particular, we obtain commutative diagrams

$$\xymatrix{
\on{Av}_?^{U/\CY}\mod(\QCoh(U))  \ar[d]^{\oblv}
& \on{Av}_?^{\IndCoh,U/\CY}\mod(\IndCoh(U)) \ar[l]_{\Psi^{\on{Av}}_U}   \ar[d]^{\oblv} \\
\QCoh(U)   \ar@<.7ex>[u]^{\ind} & \IndCoh(U)   \ar[l]_{\Psi_U}  \ar@<.7ex>[u]^{\ind}}$$
and 

$$
\CD
\on{Av}^{U/\CY}_?\mod(\QCoh(U))  @<{\Psi^{\on{Av}}_U}<<  \on{Av}^{\IndCoh,U/\CY}_?\mod(\IndCoh(U)) \\
@A{(f^?)^{\on{enh}}}AA    @AA{(f^{\IndCoh,?})^{\on{enh}}}A   \\
\QCoh(\CY)  @<{\Psi_\CY}<<  \IndCoh(\CY).
\endCD
$$

\sssec{}

In \propref{p:generation}(b) we need to show that the left adjoint $(f_*)^{\on{enh}}$ of $(f^?)^{\on{enh}}$
is conservative. 

\medskip

In the remainder of this subsection, we will deduce this assertion from the next one:

\begin{prop}  \label{p:IndCoh generation}
The functor 
$$(f^{\IndCoh}_*)^{\on{enh}}: \on{Av}^{\IndCoh,U/\CY}_?\mod(\IndCoh(U))\to \IndCoh(\CY),$$
left adjoint to $(f^{\IndCoh,?})^{\on{enh}}$, is conservative.
\end{prop}

\sssec{Proof of \propref{p:generation}(b)}

We need to show that the essential image of the functor $(f^?)^{\on{enh}}$ 
\emph{co-generates} $\on{Av}^{U/\CY}_?\mod(\QCoh(U))$, i.e., generates under the operation of 
taking \emph{limits}.

\medskip

The assertion of \propref{p:IndCoh generation} is equivalent to the fact that the essential image of 
the functor $(f^{\IndCoh,?})^{\on{enh}}$ co-generates $\on{Av}^{\IndCoh,U/\CY}_?\mod(\IndCoh(U))$.

\medskip

Hence, it is sufficient to show that the functor $\Psi^{\on{Av}}_U$ is essentially surjective. We will show that the 
functor $\Psi^{\on{Av}}_U$ admits a fully faithful \emph{right} adjoint. 

\medskip

First, since the functor $\Psi_U$ is a co-localization (see \cite[Proposition 1.5.3]{IndCoh}) and is continuous, it admits a right adjoint,
denoted $\Phi_U$, which is also fully faithful. Hence, it suffices to show that the functor $\Phi_U$
intertwines the actions of the monads $\on{Av}^{U/\CY}_?$ on $\QCoh(U)$ and
$\on{Av}^{\IndCoh,U/\CY}_?$ on $\IndCoh(U)$, respectively. For that, it is sufficient to show that
the diagrams 
$$
\CD
\QCoh(U)  @>{\Phi_U}>> \IndCoh(U) \\
@V{f_*}VV   @VV{f^{\IndCoh}_*}V   \\
\QCoh(\CY)  @>{\Phi_\CY}>> \IndCoh(\CY) 
\endCD
$$
and 
$$
\CD
\QCoh(U)  @>{\Phi_U}>> \IndCoh(U) \\
@A{f^?}AA   @AA{f^{\IndCoh,?}}A   \\
\QCoh(\CY)  @>{\Phi_\CY}>> \IndCoh(\CY) 
\endCD
$$
commute. 

\medskip

The commutation of these diagrams is obtained by passing to right adjoints
in the diagrams 
$$
\CD
\QCoh(U)  @<{\Psi_U}<< \IndCoh(U) \\
@A{f^*}AA   @AA{f^{\IndCoh,*}}A   \\
\QCoh(\CY)  @<{\Psi_\CY}<< \IndCoh(\CY) 
\endCD
$$
(see \cite[Proposition 3.5.4]{IndCoh}) and \eqref{e:Psi dir image}, respectively. 

\qed

\ssec{Proof of the $\IndCoh$-version of \propref{p:generation}(b)}

In this subsection we will prove \propref{p:IndCoh generation}. We need to show that the essential image of the functor
$(f^{\IndCoh,?})^{\on{enh}}$ co-generates $\on{Av}^{\IndCoh,U/\CY}_?\mod(\IndCoh(U))$. 

\sssec{Step 1}

Let $\CY'\overset{\imath}\hookrightarrow \CY$ be a closed substack, and let $\CY_0\overset{\jmath}\hookrightarrow \CY$
be the complementary open. Consider the corresponding adjoint pair of functors
$$\jmath^{\IndCoh,*}:\IndCoh(\CY) \rightleftarrows \IndCoh(\CY_0):\jmath^{\IndCoh}_*.$$

Recall the notation
$$\IndCoh(\CY)_{\CY'}:=\on{ker}(\jmath^{\IndCoh,*}),$$
see \cite[Sect. 4.1.2]{IndCoh}. Let
$$\wh\imath_!:\IndCoh(\CY)_{\CY'}\rightleftarrows \IndCoh(\CY):\wh\imath^!$$
be the resulting adjoint pair of functors. Let $\wh\imath_?$ denote the (discontinuous)
right adjoint of $\wh\imath^!$. We will use similar notations for the corresponding objects on $U$.

\medskip

We have a commutative diagram

$$
\xymatrix{
\IndCoh(U_0) \ar[r]_{(\jmath_U)^{\IndCoh}_*}  \ar[d]_{(f_0)^{\IndCoh}_*} & \IndCoh(U) 
 \ar@<-.7ex>[l]_{(\jmath_U)^{\IndCoh,*}}   \ar[d]^{f^{\IndCoh}_*}  \\
\IndCoh(\CY_0)  \ar[r]_{\jmath^{\IndCoh}_*}  & \IndCoh(\CY)  \ar@<-.7ex>[l]_{\jmath^{\IndCoh,*}} }
$$
and the diagram
$$
\CD
\IndCoh(U_0) @>{(\jmath_U)_*}>> \IndCoh(U) \\
@A{f_0^{\IndCoh,*}}AA  @AA{f^{\IndCoh,*}}A \\
\IndCoh(\CY_0) @>{\jmath_*}>> \IndCoh(\CY)
\endCD
$$
is also commutative, see \cite[Lemma 3.6.9]{IndCoh}.

\medskip

Hence, by passing to the right adjoint functors, we obtain that both functors $(\jmath_U)^{\IndCoh}_*$ and $(\jmath_U)^{\IndCoh,?}$
intertwine the monads $\on{Av}_?^{\IndCoh,U/\CY}$ and $\on{Av}_?^{\IndCoh,U_0/\CY_0}$
acting on $\IndCoh(U)$ and $\IndCoh(U_0)$, respectively. 

\medskip

In addition, we obtain a monad $\wh{\on{Av}}^{U'/\CY'}_?$, acting on  $\IndCoh(U)_{U'}$,
such that both functors $(\wh\imath_U)^!$ and $(\wh\imath_U)_?$ intertwine the monads
$\on{Av}_?^{\IndCoh,U/\CY}$ and $\wh{\on{Av}}_?^{\IndCoh,U'/\CY'}$
acting on $\IndCoh(U)$ and $\IndCoh(U)_{U'}$, respectively. 

\medskip

Thus, we obtain a localization sequence of categories
$$
\xymatrix{
\on{Av}_?^{\IndCoh,U_0/\CY_0}\mod  \ar[r]
& \on{Av}_?^{\IndCoh,U/\CY}\mod  \ar@<.7ex>[l]
\ar[r]
& \wh{\on{Av}}_?^{\IndCoh,U'/\CY'}\mod, \ar@<.7ex>[l]
}$$
which makes the diagram 
$$
\xymatrix{
\IndCoh(\CY_0)   \ar[r]^{\jmath^{\IndCoh}_*}  \ar[d]_{(f_0^{\IndCoh,?})^{\on{enh}}}  & \IndCoh(\CY)
\ar[r]^{\wh\imath^!} \ar@<.7ex>[l]^{\jmath^{\IndCoh,?}}  \ar[d]_{(f^{\IndCoh,?})^{\on{enh}}}    &  \IndCoh(\CY)_{\CY'} \ar@<.7ex>[l]^{\wh\imath_?}
\ar[d]^{((\wh{f'})^{\IndCoh,?})^{\on{enh}}} \\
\on{Av}_?^{\IndCoh,U_0/\CY_0}\mod  \ar[r] \ar[d]_{\oblv}
& \on{Av}_?^{\IndCoh,U/\CY}\mod  \ar@<.7ex>[l]
\ar[r] \ar[d]_{\oblv}
& \wh{\on{Av}}_?^{\IndCoh,U'/\CY'}\mod \ar@<.7ex>[l]  \ar[d]^{\oblv} \\
\IndCoh(U_0) \ar[r]^{(\jmath_U)^{\IndCoh}_*}  & \IndCoh(U) 
\ar[r]^{(\wh\imath_U)^!} \ar@<.7ex>[l]^{(\jmath_U)^{\IndCoh,?}}   
&  \IndCoh(U)_{U'} \ar@<.7ex>[l]^{(\wh\imath_U)_?}}
$$
commute. 

\medskip

Therefore we obtain that if the essential image of $(f_0^{\IndCoh,?})^{\on{enh}}$ co-generates 
the category $\on{Av}^{\IndCoh,U_0/\CY_0}_?\mod(\IndCoh(U_0))$, and the essential image of
$((\wh{f'})^{\IndCoh,?})^{\on{enh}}$ co-generates the category $\wh{\on{Av}}_?^{\IndCoh,U'/\CY'}\mod(\IndCoh(U)_{U'})$,
then the essential image of $(f^{\IndCoh,?})^{\on{enh}}$ co-generates 
$\on{Av}^{\IndCoh,U/\CY}_?\mod(\IndCoh(U))$. 

\sssec{Step 2}

We will now show that if the essential image of 
$$(f')^{\IndCoh,?})^{\on{enh}}:\IndCoh(\CY')\to \on{Av}^{\IndCoh,U'/\CY'}_?\mod(\IndCoh(U'))$$
co-generates  $\on{Av}^{\IndCoh,U'/\CY'}_?\mod(\IndCoh(U'))$, then the essential image of
$$((\wh{f'})^{\IndCoh,?})^{\on{enh}}: \IndCoh(\CY)_{\CY'}\to \wh{\on{Av}}_?^{\IndCoh,U'/\CY'}\mod(\IndCoh(U)_{U'})$$
co-generates $\wh{\on{Av}}_?^{\IndCoh,U'/\CY'}\mod(\IndCoh(U)_{U'})$. 

\medskip

We have a commutative diagram
\begin{equation} \label{e:IndCoh res1}
\CD
\IndCoh(\CY)_{\CY'}  @>{\imath^!}>>   \IndCoh(\CY') \\ 
@A{(\wh{f'})^{\IndCoh}_*}AA    @AA{(f')^{\IndCoh}_*}A   \\ 
\IndCoh(U)_{U'}  @>{(\imath_U)^!}>>   \IndCoh(U') 
\endCD
\end{equation}
(which expresses the base change isomorphism)
and the commutative diagram
\begin{equation} \label{e:IndCoh res2}
\CD
\IndCoh(\CY)_{\CY'}  @>{\imath^!}>>   \IndCoh(\CY') \\ 
@V{(\wh{f'})^{\IndCoh,*}}VV    @VV{(f')^{\IndCoh,*}}V   \\ 
\IndCoh(U)_{U'}  @>{(\imath_U)^!}>>   \IndCoh(U'),
\endCD
\end{equation}
see \cite[Proposition 7.1.6]{IndCoh}.

\medskip

Note that since the functor $\imath^!$ is continuous, it admits a (discontinuous)
right adjoint \footnote{This manipulation is the reason for replacing $\QCoh$ 
by $\IndCoh$ in the proof of \propref{p:generation}(b).}, denoted 
$\imath_?:\IndCoh(\CY)_{\CY'}\to \IndCoh(\CY')$.
By passing to right adjoints in the diagrams \eqref{e:IndCoh res1} and \eqref{e:IndCoh res2}, we obtain commutative diagrams
$$
\CD
\IndCoh(\CY)_{\CY'}  @<{\imath_?}<<   \IndCoh(\CY') \\ 
@V{(\wh{f'})^{\IndCoh,?}}VV    @VV{(f')^{\IndCoh,?}}V   \\ 
\IndCoh(U)_{U'}  @<{(\imath_U)_?}<<   \IndCoh(U') 
\endCD
$$
and
$$
\CD
\IndCoh(\CY)_{\CY'}  @<{\imath_?}<<   \IndCoh(\CY') \\ 
@A{(\wh{f'})^{\IndCoh}_*}AA    @AA{(f')^{\IndCoh}_*}A   \\ 
\IndCoh(U)_{U'}  @<{(\imath_U)_?}<<   \IndCoh(U').
\endCD
$$

\medskip

In particular, we obtain that the functor $(\imath_U)_?$ intertwines monads 
$\on{Av}^{\IndCoh,U'/\CY'}_?$ acting on $\IndCoh(U')$ and $\wh{\on{Av}}_?^{\IndCoh,U'/\CY'}$ acting on $\IndCoh(U)_{U'}$, and 
we have a commutative diagram
$$
\CD
 \IndCoh(\CY)_{\CY'}    @<{\imath_?}<<   \IndCoh(\CY') \\
@V{((\wh{f'})^{\IndCoh,?})^{\on{enh}}}VV   @VV{((f')^{\IndCoh,?})^{\on{enh}}}V   \\
\wh{\on{Av}}_?^{\IndCoh,U'/\CY'}\mod(\IndCoh(U)_{U'})  @<{(\imath_U^{\on{Av}})_?}<<  \on{Av}^{\IndCoh,U'/\CY'}_?\mod(\IndCoh(U')).
\endCD
$$

\medskip

Hence, to carry out Step 2, it remains to show that the essential image of $(\imath_U^{\on{Av}})_?$ co-generates 
$\wh{\on{Av}}_?^{\IndCoh,U'/\CY'}\mod(\IndCoh(U)_{U'})$. For this, it suffices to show that $(\imath_U^{\on{Av}})_?$
admits a left adjoint, to be denoted $(\imath_U^{\on{Av}})^!$, which is conservative.

\medskip

We claim that $(\imath_U^{\on{Av}})^!$ exists and makes the diagram
$$
\CD
\wh{\on{Av}}_?^{\IndCoh,U'/\CY'}\mod(\IndCoh(U)_{U'})  @>{(\imath_U^{\on{Av}})^!}>>  \on{Av}^{\IndCoh,U'/\CY'}_?\mod(\IndCoh(U')) \\
@V{\oblv}VV  @VV{\oblv}V  \\
\IndCoh(U)_{U'} @>{(\imath_U)^!}>> \IndCoh(U')
\endCD
$$
commutative. This would also imply that $(\imath_U^{\on{Av}})^!$ is conservative, since $(\imath_U)^!$ is conservative,
by \cite[Proposition 4.1.7(a)]{IndCoh}.

\medskip

To prove the existence of $(\imath_U^{\on{Av}})^!$ with the required property, it suffices to show that the functor $(\imath_U)^!$
intertwines the monads $\wh{\on{Av}}_?^{\IndCoh,U'/\CY'}$ acting on $\IndCoh(U)_{U'}$ and 
$\on{Av}^{\IndCoh,U'/\CY'}_?$ acting on $\IndCoh(U')$.

\medskip

The latter follows from the commutativity of the diagram \eqref{e:IndCoh res2} and the diagram
$$
\CD
\IndCoh(\CY)_{\CY'} @>{\imath^!}>> \IndCoh(\CY')  \\
@V{(\wh{f}')^{\IndCoh,?}}VV    @VV{(f')^{\IndCoh,?}}V   \\
\IndCoh(U)_{U'} @>{(\imath_U)^!}>> \IndCoh(U'), \\
\endCD
$$
which is obtained by passing to right adjoints in the commutative diagram
$$
\CD
\IndCoh(\CY)_{\CY'} @<{\imath_*}<< \IndCoh(\CY')  \\
@A{(\wh{f}')^{\IndCoh}_*}AA    @AA{(f')^{\IndCoh}_*}A   \\
\IndCoh(U)_{U'} @<{(\imath_U)_*}<< \IndCoh(U'). \\
\endCD
$$

\sssec{Step 3}

By Step 2, we can assume that $\CY$ is classical and reduced (take $\CY':=({}^{cl}\CY)_{red}$). 
By Step 1 and Noetherian induction,
it suffices to show that $\CY$ contains a non-empty open substack $\CY_0$, for which 
\propref{p:IndCoh generation} holds. 

\medskip

In particular, we can replace $\CY$ by its open substack, which is smooth. By assumption,
the morphism $f:U\to \CY$ is smooth, so $U$ is smooth as well. 
Now, for smooth schemes, there is no difference between $\IndCoh$ and $\QCoh$,
and the conclusion of \propref{p:IndCoh generation} is identical to that of \propref{p:generation}(b).

\medskip

Thus, it suffices to show that any classical reduced $\CY$ contains a non-empty open that satisfies the 
conclusion of \propref{p:generation}(b). 

\medskip

The rest of the proof is the same as that of Step 3 in the proof of \propref{p:generation}(a):  

\medskip

We reduce the assertion to the case when $\CY$ is of the form $Y/G$, where $Y$ is a quasi-compact quasi-separated
scheme, and $G$ is a classical affine algebraic group of finite type. However, $\CY$ if this form 
satisfies the conclusion of \propref{p:generation}(b) because of \thmref{t:quotient}
and \propref{p:? criterion}.

\bigskip

\centerline{\bf Part III: (DG) Indschemes, Classifying Prestacks and De Rham Prestacks} 

\section{DG indschemes}  \label{s:indsch}

\ssec{A key proposition}

\sssec{}

Recall that $\on{PreStk}_{\on{laft}}$ denotes
the full subcategory of $\on{PreStk}$ formed by prestacks \emph{locally almost of finite type}, and that on this category we have
a well-defined functor 
$$\IndCoh_{\on{PreStk}_{\on{laft}}}:(\on{PreStk}_{\on{laft}})^{\on{op}}\to \StinftyCat_{\on{cont}},$$
see \cite[Sect. 10.1]{IndCoh}.

\medskip

In addition, we recall that the functor $\IndCoh_{\on{PreStk}_{\on{laft}}}$, comes equipped with a natural transformation, denoted 
$$\Upsilon_{\on{PreStk}_{\on{laft}}}:\QCoh_{\on{PreStk}_{\on{laft}}}\to \IndCoh_{\on{PreStk}_{\on{laft}}},$$
see \cite[Sect. 10.3]{IndCoh}.

\medskip

For an individual object $\CY\in \on{PreStk}_{\on{laft}}$, the corresponding functor
$$\Upsilon_\CY:\QCoh(\CY)\to \IndCoh(\CY)$$ 
is given by 
$$\CF\in \QCoh(\CY) \mapsto \CF\otimes\omega_\CY\in  \IndCoh(\CY),$$
where $\omega_\CY\in \IndCoh(\CY)$ is the dualizing object, and where $\otimes$ is the canonical
action on $\QCoh(\CY)$ on $\IndCoh(\CY)$. 

\sssec{}

A key technical tool that will allow us to establish the results pertaining to 1-affineness
of DG indschemes (as well as formal classifying spaces and de Rham prestacks) is
the following assetion:

\begin{prop} \label{p:main prop} 
Let $\CY\in \on{PreStk}_{\on{laft}}$ be such that the unctor $\Upsilon_\CY:\QCoh(\CY)\to \IndCoh(\CY)$
is an equivalence. Suppose also that $\CY$ can be written as a colimit of $\underset{i\in I}{\underset{\longrightarrow}{colim}}\, \CZ_i$
with $\CZ_i\in \on{PreStk}_{\on{laft}}$, where $I$ is some index category, such that:

\begin{enumerate}

\item For every $i$, the functor $\bLoc_{\CZ_i}$ is fully faithful;

\item For every $i$, the functor $\Upsilon_{\CZ_i}:\QCoh(\CZ_i)\to \IndCoh(\CZ_i)$ is fully faithful and admits a continuous right adjoint,
compatible with the action of $\QCoh(\CZ_i)$; 

\item For every arrow $(i\to j)\in I$ and the corresponding map $f_{i,j}:\CZ_i\to \CZ_j$, the
functor $f_{i,j}^!:\IndCoh(\CZ_j)\to \IndCoh(\CZ_i)$ admits a left adjoint, compatible with the action of $\QCoh(\CZ_j)$; 

\end{enumerate}

Then the functor $\bLoc_\CY$ is fully faithful.
\end{prop}

The rest of this subsection is devoted to the proof of \propref{p:main prop}.
 
\sssec{Step 1} 

Let $\CY\simeq \underset{i}{\underset{\longrightarrow}{colim}}\, \CZ_i$ be a presentation of $\CY$
as in \propref{p:main prop}. 

\medskip

For $\bC\in \QCoh(\CY)\mmod$, we have:
$$\bGamma(\CY,\bLoc_\CY(\bC))\simeq \underset{i}{\underset{\longleftarrow}{lim}}\, \bGamma(\CZ_i,\bLoc_\CY(\bC)).$$

Note, however, that
$$\bLoc_\CY(\bC)|_{\CZ_i}\simeq \bLoc_{\CZ_i}(\QCoh(\CZ_i)\underset{\QCoh(\CY)}\otimes \bC).$$

Hence, the assumption that the functor $\bLoc_{\CZ_i}$ is fully faithful implies that
$$ \bGamma(\CZ_i,\bLoc_\CY(\bC))\simeq \QCoh(\CZ_i)\underset{\QCoh(\CY)}\otimes \bC.$$

\medskip

Hence, we conclude that
$$\bGamma(\CY,\bLoc_\CY(\bC))\simeq \underset{i}{\underset{\longleftarrow}{lim}}\,
\left(\QCoh(\CZ_i)\underset{\QCoh(\CY)}\otimes \bC\right).$$

\sssec{Step 2}

Consider also the category 
\begin{equation} \label{e:IndCoh values}
\bGamma^!(\CY,\bLoc_\CY(\bC)):=\underset{i}{\underset{\longleftarrow}{lim}}\, \left(\IndCoh(\CZ_i)\underset{\QCoh(\CY)}\otimes \bC\right),
\end{equation} 
where the functors
$$\IndCoh(\CZ_j)\underset{\QCoh(\CY)}\otimes \bC\to \IndCoh(\CZ_i)\underset{\QCoh(\CY)}\otimes \bC$$
are $f_{i,j}^!\otimes \on{id}_\bC$, where for an arrow $(i\to j)\in I$, we denote by $f_{i,j}$ the corresponding map $\CZ_i\to \CZ_j$.

\medskip

Since in the formation of \eqref{e:IndCoh values}, the transition functors admit
left adjoints, by \cite[Lemma 1.3.3]{DGCat}, we can rewrite $\bGamma^!(\CY,\bLoc_\CY(\bC))$ also as 
$$\underset{i}{\underset{\longrightarrow}{colim}}\, \left(\IndCoh(\CZ_i)\underset{\QCoh(\CY)}\otimes \bC\right),$$
where the transition functors $$\IndCoh(\CZ_i)\underset{\QCoh(\CY)}\otimes \bC\to \IndCoh(\CZ_j)\underset{\QCoh(\CY)}\otimes \bC$$
are now $(f_{i,j}^!)^L\otimes \on{Id}_\bC$. 

\medskip

Commuting the colimit with the tensor product we obtain:
\begin{multline*}
\bGamma^!(\CY,\bLoc_\CY(\bC))\simeq \underset{i}{\underset{\longrightarrow}{colim}}\, \left(\IndCoh(\CZ_i)\underset{\QCoh(\CY)}\otimes \bC\right)
\simeq \\
\simeq \left(\underset{i}{\underset{\longrightarrow}{colim}}\, \IndCoh(\CZ_i)\right) \underset{\QCoh(\CY)}\otimes \bC 
\simeq \left(\underset{i}{\underset{\longleftarrow}{lim}}\, \IndCoh(\CZ_i)\right) \underset{\QCoh(\CY)}\otimes \bC \simeq
\IndCoh(\CY)\underset{\QCoh(\CY)}\otimes \bC.
\end{multline*}

\sssec{Step 3}

Recall now the natural transformation $\Upsilon_{\on{PreStk}_{\on{laft}}}$. 

\medskip

For any $\CY\in \on{PreStk}_{\on{laft}}$, this gives rise to a functor, denoted $\Upsilon_{\CY,\bC}$, 
\begin{multline*}
\bGamma(\CY,\bLoc_\CY(\bC))\simeq \underset{i}{\underset{\longleftarrow}{lim}}\,
\left(\QCoh(\CZ_i)\underset{\QCoh(\CY)}\otimes \bC\right)\to \\
\to \underset{i}{\underset{\longleftarrow}{lim}}\,
\left(\IndCoh(\CZ_i)\underset{\QCoh(\CY)}\otimes \bC\right)\simeq\bGamma^!(\CY,\bLoc_\CY(\bC)),
\end{multline*}
and we have a commutative diagram
\begin{equation} \label{e:IndCoh QCoh diagram}
\CD
\bC @>{\sim}>> \QCoh(\CY)\underset{\QCoh(\CY)}\otimes \bC @>>>  \bGamma(\CY,\bLoc_\CY(\bC)) \\
& & @V{\Upsilon_{\CY}\otimes \on{Id}_\bC}VV     @VV{\Upsilon_{\CY,\bC}}V    \\
& & \IndCoh(\CY)\underset{\QCoh(\CY)}\otimes \bC  @>{\sim}>>  \bGamma^!(\CY,\bLoc_\CY(\bC)).
\endCD
\end{equation}

By assumption, $\Upsilon_{\CY}$ is an equivalence, and hence so is the left vertical arrow in 
diagram \eqref{e:IndCoh QCoh diagram}.

\medskip

Thus, we obtain that $\bC$ is a retract of $\bGamma(\CY,\bLoc_\CY(\bC))$. To prove the proposition,
it remains to show that the functor $\Upsilon_{\CY,\bC}$ is fully faithful.

\sssec{Step 4}

To prove that $\Upsilon_{\CY,\bC}$ is fully faithful, it is sufficient to show that
each of the functors 
$$\Upsilon_{\CZ_i}\otimes \on{Id}_{\bC}:\QCoh(\CZ_i)\underset{\QCoh(\CY)}\otimes \bC
\to \IndCoh(\CZ_i)\underset{\QCoh(\CY)}\otimes \bC$$
is fully faithful.

\medskip

However, this follows from the fact that $\Upsilon_{\CZ_i}$ is fully faithful and admits a continuous
right adjoint, compatible with the action of $\QCoh(\CZ_i)$. 
Indeed, in this case, the functor $\Upsilon_{\CZ_i}\otimes \on{Id}_{\bC}$ admits a right
adjoint, given by $(\Upsilon_{\CZ_i})^R\otimes \on{Id}_{\bC}$, and the unit of the adjunction
$$\on{Id}_{\QCoh(\CZ_i)\underset{\QCoh(\CY)}\otimes \bC}\to 
\left((\Upsilon_{\CZ_i})^R\otimes \on{Id}_{\bC}\right)\circ \left(\Upsilon_{\CZ_i}\otimes \on{Id}_{\bC}\right)
\simeq \left((\Upsilon_{\CZ_i})^R\circ \Upsilon_{\CZ_i}\right) \otimes \on{Id}_{\bC}$$
is an isomorphism.

\qed

\ssec{Fully faithfulness of $\bLoc$}  \label{ss:indsch}

In this subsection we will prove Theorems \ref{t:indsch} and \ref{t:loop group}.

\sssec{Proof of \thmref{t:indsch}}

Let $\CY$ be a DG indscheme, which is weakly $\aleph_0$, locally almost of finite type and formally smooth.
First, by \cite[Theorem 10.1.1]{IndSch}, the functor $\Psi_\CY$ is an equivalence. 

\medskip

Next, by 
\cite[Theorem 9.1.6]{IndSch}, we obtain that $\CY$ is \emph{classical}, i.e., it can be written (up to
fppf sheafification) as
$$\underset{\alpha}{\underset{\longrightarrow}{lim}}\, Z_\alpha,$$
where $Z_\alpha$ are classical schemes of finite type, and the maps $Z_{\alpha_1}\to Z_{\alpha_2}$
are closed embeddings. 

\medskip

Let us show that this presentation satisfies that conditions of \propref{p:main prop}. 
Indeed, Condition (1) is satisfied by \thmref{t:alg space}. 
Condition (2) is satisfied by \cite[Corollary 9.6.3]{IndCoh}. Condition (3)
is satisfied because the left adjoints are given by $(f_{i,j})^{\IndCoh}_*$.

\sssec{Proof of \thmref{t:loop group}} \label{sss:loop group}

Consider the affine Grassmannian $\Gr_G$ and the canonical projection $\pi:G\ppart\to \Gr_G$. 

\medskip

By \thmref{t:indsch}, the functor $\bLoc_{\Gr_G}$ is fully faithful. Hence, in order to show that
$\bLoc_{G\ppart}$ is also fully faithful, it suffices to show that $\pi$ satisfies the assumptions
of \propref{p:almost commute 1}(a). 

\medskip

However, since $\pi$ is affine this follows from \cite[Proposition 3.2.1]{QCoh} 
(reproduced for the reader's convenience in the Appendix as \propref{p:quasi-aff}).  

\ssec{Non 1-affineness of $\BA^\infty$}  \label{ss:A-infty}

In this subsection we will prove \thmref{t:A-infty}. I.e., we will show that the functor
$\bGamma_{\BA^\infty}$ fails to be fully faithful.

\sssec{}

Let $\iota$ denote the map $\on{pt}\to \BA^\infty$ corresponding to $0\in \BA^\infty$. Consider
$$\CC:=\coIind_\iota(\Vect).$$

We will show that the counit map 
$$\bGamma\left(\on{pt},\bLoc_{\BA^\infty}\circ \bGamma^{\on{enh}}_{\BA^\infty}(\CC)\right)\to \bGamma(\on{pt},\CC)$$
is not an equivalence. 

\medskip

By definition
$$\bGamma(\on{pt},\CC)\simeq \QCoh(\on{pt}\underset{\BA^\infty}\times \on{pt}),$$
and
$$\bGamma\left(\on{pt},\bLoc_{\BA^\infty}\circ \bGamma^{\on{enh}}_{\BA^\infty}(\CC)\right)
\simeq \Vect\underset{\QCoh(\BA^\infty)}\otimes \Vect.$$

Thus, we want to show that the naturally defined functor
\begin{equation} \label{e:functor to fail}
\Vect\underset{\QCoh(\BA^\infty)}\otimes \Vect\to 
\QCoh(\on{pt}\underset{\BA^\infty}\times \on{pt})
\end{equation}
is not an equivalence.

\sssec{}

Let $\sS$ denote the tautological functor
$$\Vect\to  \Vect\underset{\QCoh(\BA^\infty)}\otimes \Vect,$$
and let $\sT$ be its (discontinuous) right adjoint. 

\medskip 

The composed functor
$$\Vect\overset{\sS}\to \Vect\underset{\QCoh(\BA^\infty)}\otimes \Vect \to \QCoh(\on{pt}\underset{\BA^\infty}\times \on{pt})$$
is the pull-back functor $p^*$, where $p$ denotes the map $\on{pt}\underset{\BA^\infty}\times \on{pt}\to \on{pt}$. The 
right adjoint of $p^*$ is the usual direct image functor $p_*$.

\medskip

We will show:

\begin{lem} \label{l:inf monadic}
The functor $\sT$ is monadic.
\end{lem}

\begin{lem} \label{l:inf non-monadic}
The functor $p_*$ is \emph{not} monadic.
\end{lem}

Together, these two lemmas imply that \eqref{e:functor to fail} is not an equivalence.

\sssec{Proof of \lemref{l:inf monadic}}

To analyze $(\sS,\sT)$ we identify $\QCoh(\BA^\infty)$ with $\IndCoh(\BA^\infty)$ via the functor
$\Psi_{\BA^\infty}$. The usual monoidal structure on $\QCoh(\BA^\infty)$ goes over to the 
$\overset{!}\otimes$ monoidal structure on $\IndCoh(\BA^\infty)$. Similarly, 
$\iota^*:\QCoh(\BA^\infty)\to \Vect$ goes over to the functor $\iota^!$.

\medskip

Note in this case the tensor product functor 
$$\IndCoh(\BA^\infty)\otimes \IndCoh(\BA^\infty)\to \IndCoh(\BA^\infty)$$
admits a \emph{left} adjoint, given by $(\Delta_{\BA^\infty})^\IndCoh_*$.
Similarly, the functor $\iota^!$ admits a \emph{left} adjoint given by $\iota^\IndCoh_*$,
which is a map of $\IndCoh(\BA^\infty)$-module categories. 

\medskip

This implies that $\sT$ is monadic by \corref{c:tensor prod Beck-Chevalley}.

\qed

\ssec{Proof of \lemref{l:inf non-monadic}}

\sssec{}

By definition:
$$\QCoh(\on{pt}\underset{\BA^\infty}\times \on{pt})\simeq \underset{n}{\underset{\longleftarrow}{lim}}\, 
\QCoh(\on{pt}\underset{V_n}\times \on{pt}),$$
where $V_n=\BA^n$. Set $V:=\underset{n}{\underset{\longrightarrow}{lim}}\, V_n$.

\medskip

For each $n$, we have:
$$\QCoh(\on{pt}\underset{V_n}\times \on{pt})\simeq \on{Sym}(V_n^*[1])\mod.$$

Hence, the monad $\sH:=p_*\circ p^*$ is given by 
$$M\mapsto \underset{n}{\underset{\longleftarrow}{lim}}\, \left( \on{Sym}(V_n^*[1])\otimes M\right).$$

Consider the corresponding functor 
$$(p_*)^{\on{enh}}:\QCoh(\on{pt}\underset{\BA^\infty}\times \on{pt})\to \sH\mod(\Vect).$$

We need to show that $(p_*)^{\on{enh}}$ is \emph{not} an equivalence. We will do so by showing
that it does not send a certain direct sum to the direct sum.

\sssec{}

Consider the object $N$ of 
$$\QCoh(\on{pt}\underset{\BA^\infty}\times \on{pt})\simeq \underset{n}{\underset{\longleftarrow}{lim}}\, 
\on{Sym}(V_n^*[1])\mod,$$
whose $n$-th term is 
$$\underset{m\geq n}{\underset{\longleftarrow}{lim}}\, \on{Sym}\left(\on{ker}(V^*_m\to V^*_n)[2]\right),$$
viewed as an object of $\on{Sym}(V_n^*[1])\mod$ via the trivial action. 

\medskip

We have $p_*(N)=k\in \Vect$. It is also easy to see that $(p_*)^{\on{enh}}=k\in \sH\mod$, 
where the action of $\sH$ on $k$ is trivial. 

\medskip

Consider now the object $\underset{i}\oplus\, N[-2i]\in \QCoh(\on{pt}\underset{\BA^\infty}\times \on{pt})$.
We will show that the map
$$\underset{i}\bigoplus\, (p_*)^{\on{enh}}(N)[-2i] \to (p_*)^{\on{enh}}\left(\underset{i}\oplus\, N[-2i]\right)$$
is \emph{not} an isomorphism.

\sssec{}

On the one hand, it is easy to see that the object of $\Vect$, underlying the object $$\underset{i}\oplus\, k[-2i]\in \sH\mod$$
is isomorphic to $\underset{i}\oplus\, k[-2i]$ (although the forgetful functor $\sH\mod\to \Vect$ does not commute 
with arbitrarty direct sums).  

\sssec{}

On the other hand, we will show that the object of $\Vect$ underlying 
$$(p_*)^{\on{enh}}\left(\underset{i}\oplus\, N[-2i]\right)$$
has a non-trivial cohomology in degree $1$.  

\medskip

Namely, the 1st cohomology in question is equal to $R^1(limproj)$ of the following inverse family of vector spaces
$$n\mapsto \underset{i}\bigoplus\, \left(\underset{m\geq n}{\underset{\longleftarrow}{lim}}\, 
\left(\on{Sym}^i(\on{ker}(V_m^*\to V_n^*))\right)\right).$$

\medskip

We compute $R(limproj)$ of the above family by embedding it into the constant family with value
$$\underset{i}\bigoplus\, \left(\underset{m}{\underset{\longleftarrow}{lim}}\, \on{Sym}^i(V^*_m)\right).$$

To see that $R^1(limproj)\neq 0$, we need to show that the map
\begin{multline} \label{e:map of limproj}
 \underset{i}\bigoplus\,  \underset{m}{\underset{\longleftarrow}{lim}}\, \on{Sym}^i(V^*_m)\to \\
\to \underset{n}{limproj}\, \left( \underset{i}\bigoplus\, \on{coker}\left( \underset{m\geq n}{\underset{\longleftarrow}{lim}}\, 
\on{Sym}^i(\on{ker}(V_m^*\to V_n^*))\to \underset{m}{\underset{\longleftarrow}{lim}}\, \on{Sym}^i(V^*_m)
 \right)\right)
 \end{multline}
is \emph{not} surjective. 

\sssec{}

Choose a basis $v_1,v_2,\ldots$ of $V$ so that $\{v_1,...,v_n\}$ is a basis of $V_n$. For every $m$,
let $\{v_{m,1}^*,...,v_{m,m}^*\}$ be the corresponding dual basis of $V_m^*$. 

\medskip

The following
element in the right-hand side of \eqref{e:map of limproj} does not lie in the image of the left-hand side:

\medskip

Its $n$-th component, i.e., the corresponding element of 
\begin{equation} \label{e:nth comp}
\underset{i}\bigoplus\, \on{coker}\left( \underset{m\geq n}{\underset{\longleftarrow}{lim}}\, 
\on{Sym}^i(\on{ker}(V_m^*\to V_n^*))\to \underset{m}{\underset{\longleftarrow}{lim}}\, \on{Sym}^i(V^*_m)\right)
\end{equation}
equals the sum over $i=1,2,\ldots$ of the elements $w_{n,i}$, where each $w_{n,i}$ is the image under
$$\underset{m}{\underset{\longleftarrow}{lim}}\, \on{Sym}^i(V^*_m)\to
\on{coker}\left( \underset{m\geq n}{\underset{\longleftarrow}{lim}}\, 
\on{Sym}^i(\on{ker}(V_m^*\to V_n^*))\to \underset{m}{\underset{\longleftarrow}{lim}}\, \on{Sym}^i(V^*_m)\right)$$
of the family of elements
$$m\mapsto (v_{m,i}^*)^{\otimes i}\in \on{Sym}^i(V_m^*),\quad m\geq i.$$

Note that $w_{n,i}=0$ if $i\geq n$, because in this case
$$(v_{m,i}^*)^{\otimes i} \in \on{Sym}^i(\on{ker}(V_m^*\to V_n^*)).$$

\medskip 

Hence, the sum $\underset{i}\Sigma\, w_{n,i}$ is finite, i.e., gives rise to a well-defined element in 
\eqref{e:nth comp}.

\qed

\section{Classifying prestacks}   \label{s:classifying}

\ssec{Sheaves of categories over classifying prestacks} \label{ss:groups}

\sssec{}

In this section we let $\CG$ be a group-object of $\on{PreStk}$ such that
the functor $\bLoc_\CG$ is fully faithful.

\medskip

Note that by \corref{c:Loc and prod}, the functor $\bLoc_{\CG^n}$ is fully faithful
for any $n$. 

\medskip

We will give a more explicit description of the category $\on{ShvCat}(B\CG)$ as well
as the functors $\bGamma_{B\CG}$ and $\bLoc_{B\CG}$.

\sssec{}

First, by definition, we have:
$$\on{ShvCat}(B\CG)\simeq \on{Tot}(\on{ShvCat}(B^\bullet\CG)).$$

We now claim:

\begin{prop} \label{p:two Tots}
The term-wise $\bLoc$ functor 
$$\on{Tot}(\QCoh(B^\bullet\CG)\mmod)\to \on{Tot}(\on{ShvCat}(B^\bullet\CG))$$
is an equivalence. 
\end{prop}

\begin{proof}
The functor in question is fully faithful since each $\bLoc_{\CG^{\times n}}$ is fully faithful.
To prove that it is essentially surjective, we need to show that for $\bC^\bullet\in \on{Tot}(\on{ShvCat}(B^\bullet\CG))$,
each term $\bC^n$ lies in the essential image of the functor $\bLoc_{\CG^{\times n}}$. 

\medskip

Choosing any map $[0]\to [n]$ in $\bDelta$, from the commutative diagram
$$
\CD
\QCoh(\CG^{\times n})\mmod @>{\bLoc_{\CG^{\times n}}}>>  \on{ShvCat}(\CG^{\times n})  \\
@AAA   @AAA \\
\QCoh(\CG^{\times 0})\mmod @>{\bLoc_{\CG^{\times 0}}}>>  \on{ShvCat}(\CG^{\times 0}),
\endCD
$$
we obtain that it is sufficient to consider the case of $n=0$. However, since $\CG^{\times 0}=\on{pt}$,
the latter case is evident.
\end{proof}

\sssec{}  \label{sss:as G-mod}

Note that, by \propref{p:Loc and prod}, the assumption that $\bLoc_\CG$ be fully faithful implies that 
for any $n$, the functor $$\QCoh(\CG)^{\otimes n}\to \QCoh(\CG^{\times n})$$
is an equivalence. 

\medskip

In particular, the structure on $\CG$ is group-object of $\on{PreStk}$ defines on $\QCoh(\CG)$ a structure
of augmented co-monoidal DG category, such that the corresponding co-simplicial category 
$\on{co-Bar}^\bullet(\QCoh(\CG))$ identifies with $\QCoh(B^\bullet\CG)$. 

\medskip

We note, however, that the above co-monoidal structure on $\QCoh(\CG)$ naturally extends
to a commutative Hopf algebra structure (see \secref{s:Hopf} for what this means) as an object 
of $\StinftyCat_{\on{cont}}$, via \emph{pointwise} (symmetric) monoidal structure.
In particular, $\on{co-Bar}^\bullet(\QCoh(\CG))$ is naturally
a co-simplicial (symmetric) monoidal DG category. 

\medskip

From \propref{p:Hopf} we obtain:

\begin{cor}  \label{c:over group}
The category $\on{Tot}(\QCoh(B^\bullet\CG)\mmod)$ is canonically equivalent to 
the category $\QCoh(\CG)\commod$. Under this identification, for $\bD\in \QCoh(\CG)\commod$,
the corresponding object of $\on{Tot}(\QCoh(B^\bullet\CG)\mmod)$ identifies with
$$\on{co-Bar}^\bullet(\QCoh(\CG),\bD)\in \on{co-Bar}^\bullet(\QCoh(\CG))\mmod
\simeq \QCoh(B^\bullet\CG)\mmod.$$
\end{cor}

\ssec{Categories acted on by $\CG$}

\sssec{}  \label{sss:G-mod}

Let $\CG\mmod$ denote the category $\QCoh(\CG)\commod$. 

\medskip

Combining \corref{c:over group}
and \propref{p:two Tots}, we obtain an equivalence
\begin{equation} \label{e:shvs as comod}
\on{ShvCat}(B\CG)\simeq \CG\mmod.
\end{equation}

In what follows we shall refer to objects of $\CG\mmod$ as \emph{categories endowed with an action of the 
group-prestack $\CG$}.

\sssec{}

Under the equivalence \eqref{e:shvs as comod} the category $\QCoh(B\CG)$ identifies with
$$\Rep(\CG):=\uHom_{\CG}(\Vect,\Vect).$$

\medskip

The functor 
$$\bGamma_{B\CG}:\on{ShvCat}(B\CG)\to \StinftyCat_{\on{cont}}$$
identifies with the functor
$$\Iinv_\CG:\CG\mmod\to \StinftyCat_{\on{cont}},\quad \bD\mapsto \uHom_{\CG}(\Vect,\bD)
\simeq\on{Tot}\left(\on{co-Bar}^\bullet(\QCoh(\CG),\bD)\right).$$

This functor naturally upgrades to the functor
$$\Iinv^{\on{enh}}_\CG:\CG\mmod\to \Rep(\CG)\mmod,$$
and the latter identifies with $\bGamma_{B\CG}^{\on{enh}}$.

\begin{rem}
We regard $\Rep(\CG)$ as being equipped with a monoidal structure resulting from
its definition as $\uHom_{\CG}(\Vect,\Vect)$. However, this structure naturally
extends to a symmetric (i.e., $E_\infty$) monoidal structure: 

\medskip

The right-lax symmetric monoidal structure on the functor $\Iinv_\CG$ defines on
$$\Rep(\CG)\simeq \Iinv_\CG(\Vect)$$ a structure of unital symmetric monoidal DG category,
which is compatible (=distributive) with the monoidal structure given by composition. 
Hence, by Eckmann-Hilton, the latter structure is induced by the former.
\end{rem}

\sssec{}

The functor 
$$\bLoc_{B\CG}:\Rep(\CG)\mmod\to  \on{ShvCat}(B\CG)$$
identifies with the functor
$$\Rrec^{\on{enh}}_{\CG}:\Rep(\CG)\mmod\to \CG\mmod, \quad \bC\mapsto \Vect\underset{\Rep(\CG)}\otimes \bC,$$
where \footnote{The symbol  $\Rrec^{\on{enh}}$ is for ''reconstruction."}
the right hand side naturally acquires a structure of an object of $\CG\mmod$ via the commuting $\CG$- and
 $\Rep(\CG)$-actions on $\Vect$. 
 
\medskip
 
Here the action of $\Rep(\CG)$ on $\Vect$ is given by the augmentation 
(forgetful) functor 
$$\on{oblv}_\CG:\Rep(\CG)\to \Vect, \quad \uHom_{\CG}(\Vect,\Vect)\to \uHom(\Vect,\Vect).$$

\medskip

In what follows we shall denote by $\Rrec_{\CG}$ the composition of $\Rrec^{\on{enh}}_{\CG}$ and the forgetful
functor
$$\Oblv_\CG:\CG\mmod\to \StinftyCat_{\on{cont}}.$$

I.e., 
$$\Rrec_{\CG}(\bC)= \Vect\underset{\Rep(\CG)}\otimes \bC\in \StinftyCat_{\on{cont}}.$$

\sssec{}

It is easy to see that $\Iinv_\CG(\QCoh(\CG))\simeq \Vect$, and moreover, this equivalence extends
to an isomorphism
$$\Iinv^{\on{enh}}_\CG(\QCoh(\CG))\simeq \Vect$$
in $\Rep(\CG)\mmod$, where $\Rep(\CG)$ acts on $\Vect$ ``trivially", i.e., via the functor 
$\on{oblv}_\CG$.

\medskip

In particlar, by adjunction, we obtain a map in $\CG\mmod$
\begin{equation} \label{e:recovery map enh}
\Rrec^{\on{enh}}_\CG(\Vect)\to \QCoh(\CG).
\end{equation}

At the level of plain DG categories, the map \eqref{e:recovery map enh} identifies with the functor
\begin{equation} \label{e:recovery map}
p^*:\Vect\underset{\Rep(\CG)}\otimes \Vect\to \QCoh(\CG).
\end{equation}

\sssec{}   \label{sss:on G dualizable}

Assume for a moment that $\QCoh(\CG)$ is dualizable as a plain DG category. Consider its dual, $\QCoh(\CG)^\vee$,
as a monoidal DG category (the monoidal structure on $\QCoh(\CG)^\vee$ is the dual of the co-monoidal structure
on $\QCoh(\CG)$). 

\medskip

In this case, by taking the dual of the co-action of $\QCoh(\CG)^\vee$, 
we can identify
$$\QCoh(\CG)\commod\simeq \QCoh(\CG)^\vee\mmod,$$
and hence
\begin{equation} \label{e:G cat as modules}
\CG\mmod\simeq \QCoh(\CG)^\vee\mmod.
\end{equation}

Let us note that in this case, in addition to the functors $\Iinv_\CG$ and  $\Iinv^{\on{enh}}_\CG$ we also have the functor 
$$
\ccoinv_\CG:\CG\mmod\to \StinftyCat_{\on{cont}} \text{ and } \ccoinv^{\on{enh}}_\CG:\CG\mmod\to \Rep(\CG)\mmod,
$$
defined by 
$$\bD\mapsto \Vect\underset{\QCoh(\CG)^\vee}\otimes \bD\simeq |\on{Bar}^\bullet(\QCoh(\CG)^\vee,\bD)|.$$

\ssec{Affine group DG schemes}

In this subsection we let  $\CG$ be a group-object of $\affdgSch$. 

\sssec{}

For $\bD\in \CG\mmod$ we consider the co-simplicial category $\on{co-Bar}^\bullet(\QCoh(\CG),\bD)$. As in
\lemref{l:BC*}, we note that $\on{co-Bar}^\bullet(\QCoh(\CG),\bD)$ satisfies the co-monadic Beck-Chevalley
condition.

\medskip

In particular, the forgetful functor
$$\on{oblv}_\CG:\Iinv_\CG(\bD)\to \bD$$
admits a right adjoint, denoted $\on{coind}_\CG$, and the co-monad
$$\on{oblv}_\CG\circ \on{coind}_\CG$$
identifies, when viewed as a plain endo-functor of $\bD$, with the composition
$$\bD\overset{\on{co-action}}\longrightarrow 
\bD\otimes \QCoh(\CG)\overset{\on{Id}_\bD\otimes p_*}\to \bD,$$
where $p:\CG\to \on{pt}$.

\sssec{}

Note that in the affine case we have a canonical
identification
\begin{equation} \label{e:self-duality affine}
\QCoh(\CG)^\vee\simeq \QCoh(\CG).
\end{equation}

Moreover, we note that the \emph{monoidal} structure on $\QCoh(\CG)$, induced by
\eqref{e:self-duality affine} and the \emph{co-monoidal} structure on $\QCoh(\CG)$, 
is canonically equivalent to that given by the structure on $\CG$ of an algebra-object
in $\affdgSch$ (with respect to the monoidal structure on $\affdgSch$ given by the Cartesian
product) and the monoidal functor
$$\QCoh_*:\affdgSch\to \StinftyCat_{\on{cont}},\quad S\rightsquigarrow \QCoh(S), \quad (f:S_1\to S_2)\rightsquigarrow f_*.$$

\medskip

We shall denote the resulting monoidal DG category by $\QCoh(\CG)_{\on{conv}}$. 

\medskip

When we consider $\QCoh(\CG)$
with the (symmetric) monoidal structure given by the pointwise tensor product, we shall denote it by $\QCoh(\CG)_{\on{ptw}}$.

\ssec{A criterion for 1-affineness in the affine case case}

In this subsection we continue to assume that $\CG$ is an affine group DG scheme.

\sssec{}

Consider the simplicial category $\on{Bar}^\bullet(\QCoh(\CG)_{\on{conv}})$. It is obtained by applying the functor
$\QCoh_*$ to the simplicial object $B^\bullet\CG\in \affdgSch$. In particular, we have a canonical identification
$$|\QCoh(B^\bullet\CG)_*|\simeq \Vect\underset{\QCoh(\CG)_{\on{conv}}}\otimes \Vect.$$

\medskip

As in \lemref{l:BC?}, from \lemref{l:double right adjoint new} and \cite[Lemma 1.3.3]{DGCat}, we obtain:

\begin{lem} \label{l:coinduction}  The right adjoint of the tautological functor
$$\Vect\to \Vect\underset{\QCoh(\CG)_{\on{conv}}}\otimes \Vect$$
is monadic, and the corresponding monad on $\Vect$, viewed as a plain endo-functor,
identifies canonically with $p_*\circ (p_*)^R$, where
$$p:\CG\to \on{pt}.$$
\end{lem} 

\sssec{}

Finally, we note that as in \secref{sss:cosections}, from the natural transformation \eqref{e:from coinv to inv} we obtain that
there exists a canonically defined functor
\begin{equation} \label{e:from coinv to inv groups}
\on{coind}_\CG:\Vect\underset{\QCoh(\CG)_{\on{conv}}}\otimes \Vect\to \Rep(\CG),
\end{equation}
which lifts to a map
$$\ccoinv^{\on{enh}}_\CG(\Vect)\to \Iinv^{\on{enh}}_\CG(\Vect)=\Rep(\CG)$$
in $\Rep(\CG)\mmod$. 

\medskip

We claim:

\begin{prop} \label{p:1-aff and rigid} 
The following conditions are equivalent:

\smallskip

\noindent{\em(a)} $B\CG$ is 1-affine;

\medskip

\noindent{\em(b)} The following two conditions hold:

\begin{enumerate}

\item The functor \eqref{e:from coinv to inv groups} is an equivalence; 

\item The functor \eqref{e:recovery map} is an equivalence.

\end{enumerate}

\noindent{\em(b')} The following two conditions hold:

\begin{enumerate}

\item
There exists \emph{some} isomorphism of objects of $\Rep(\CG)\mmod$
$$\ccoinv^{\on{enh}}_\CG(\Vect)\simeq \Rep(\CG).$$

\item
There exists \emph{some} isomorphism of objects of objects of $\CG\mmod$
$$\Rrec^{\on{enh}}_\CG(\Vect)\simeq \QCoh(\CG)_{\on{conv}}.$$

\end{enumerate}

\end{prop}

\begin{proof} 

Let assume (a) and deduce (b). The fact that the map  \eqref{e:recovery map} is an equivalence
holds is the expression of the fact that the functor $\bGamma^{\on{enh}}_{B\CG}$ is fully faithful.

\medskip

Let us view $\Vect$ as object of $\CG\mmod$,
endowed with a commuting action of $\Rep(\CG)$. Tensoring up the map  \eqref{e:from coinv to inv groups}
on the right by $\Vect$ over $\Rep(\CG)$, we obtain a map
\begin{multline} \label{e:triple ten prod}
\Vect\underset{\QCoh(\CG)_{\on{conv}}}\otimes \left(\Vect\underset{\Rep(\CG)}\otimes \Vect\right)\simeq \\
\simeq \left(\Vect\underset{\QCoh(\CG)_{\on{conv}}}\otimes \Vect\right)\underset{\Rep(\CG)}\otimes \Vect\to 
\Rep(\CG)\underset{\Rep(\CG)}\otimes \Vect,
\end{multline} 
so that the diagram
$$
\CD
\Vect\underset{\QCoh(\CG)_{\on{conv}}}\otimes \left(\Vect\underset{\Rep(\CG)}\otimes \Vect\right)  @>>>
\Rep(\CG)\underset{\Rep(\CG)}\otimes \Vect  \\
@V{\on{Id}_{\Vect}\otimes p^*}VV  @VV{\sim}V \\
\Vect\underset{\QCoh(\CG)_{\on{conv}}}\otimes \QCoh(\CG)_{\on{conv}}  @>{\sim}>>  \Vect
\endCD
$$
commutes. Since the left vertical arrow is an isomorphism (by the above), we
obtain that so is the top horizontal map. 

\medskip

Now, since $\Rrec_\CG$ is conservative, the fact that \eqref{e:triple ten prod} is an equivalence, implies
that so is the map \eqref{e:from coinv to inv groups}.

\medskip

The fact that (b) implies (b') is tautological. 

\medskip

Let us show that (b') implies (a). However, this is obvious: (b') implies that 
the functors $\ccoinv^{\on{enh}}_\CG$ and $\Rrec^{\on{enh}}_\CG$ are mutually inverse
on the nose.

\end{proof}

\section{Groups with a rigid convolution category}   \label{s:formal groups}

\ssec{The rigidity condition}   \label{ss:rigid two}

We return to the conext of \secref{sss:on G dualizable}. In this subsection we will assume that
$\QCoh(\CG)$ is dualizable as a plain category, and, moreover, that the \emph{monoidal} category 
$\QCoh(\CG)^\vee$ is rigid (see \secref{ss:rigidity} for what this means).

\sssec{}

The self-duality of $\QCoh(\CG)^\vee$ induced by its rigid monoidal structure (see \secref{sss:rigid duality}) 
defines, in particular, an identification
$$\QCoh(\CG)^\vee\simeq \QCoh(\CG),$$
as plain categories.

\medskip

Thus, we can again think of $\QCoh(\CG)$ is a monoidal DG category; when considered as such,
it will be denoted by $\QCoh(\CG)_{\on{conv}^L}$. This monoidal structure should not be confused
with the pointwise (symmetric) monoidal structure; the latter is denoted by $\QCoh(\CG)_{\on{ptw}}$. 

\sssec{}

By \secref{ss:dual and adj rigid modules}, we obtain that the monoidal structure on $\QCoh(\CG)_{\on{conv}^L}$
is obtained from the co-monoidal structure on $\QCoh(\CG)$ by passage to the \emph{left}
adjoint functors. 

\medskip

By construction, the unit in $\QCoh(\CG)_{\on{conv}^L}$ is given by the functor 
$$(e^*)^L:\Vect\to \QCoh(\CG),$$
which is both the \emph{left} adjoint and the dual of $e^*:\QCoh(\CG)\to \Vect$, where $e:\on{pt}\to\CG$
is the unit point. 

\medskip

By \propref{p:right adj and dual homomorphisms}, 
the augmentation on $\QCoh(\CG)_{\on{conv}^L}$ is given by the functor
$$(p^*)^L:\QCoh(\CG)\to \Vect,$$
which is both the \emph{left} adjoint and the dual of $p^*:\Vect\to \QCoh(\CG)$.

\sssec{}

Let now $\bD$ be an object of $\CG\mmod$, thought of as an object of $\QCoh(\CG)_{\on{conv}^L}\mmod$. 
From \secref{ss:dual and adj rigid modules} we obtain:

\begin{lem}  \label{l:co-action as right adjoint}   \hfill

\smallskip

\noindent{\em(a)} The \emph{right adjoint} of the action data
$$\QCoh(\CG)_{\on{conv}^L}\otimes \bD\to \bD$$
identifies with the data of co-action 
$$\bD\to \QCoh(\CG)\otimes \bD.$$

\smallskip

\noindent{\em(b)} The simplicial category $\on{Bar}^\bullet(\QCoh(\CG)_{\on{conv}^L},\bD)$ 
is obtained from $\on{co-Bar}^\bullet(\QCoh(\CG),\bD)$ by passage to the left adjoint functors. 
\end{lem}

From \corref{c:dual and rigid module bis}, we obtain:

\begin{cor} \label{c:induction} \hfill

\smallskip

\noindent{\em(a)} The co-simplicial category $\on{co-Bar}^\bullet(\QCoh(\CG),\bD)$ satisfies
the monadic Beck-Chevalley condition. 

\smallskip

\noindent{\em(b)} The forgetful functor 
$$\on{oblv}_{\CG}:\Iinv_\CG(\bD)\to \bD$$
admits a left adjoint (denoted $\on{ind}_\CG$) and 
is monadic. The monad $$\on{oblv}_{\CG}\circ \on{ind}_\CG,$$ viewed
as a plain endo-functor of $\bD$, identifies with the composition
$$\bD\overset{\on{co-action}} \longrightarrow \QCoh(\CG)\otimes \bD \overset{(p^*)^L\otimes \on{Id}_\bD}\longrightarrow \bD,$$
where $(p^*)^L$ is the left adjoint of the functor $p^*:\Vect\to \QCoh(\CG)$.
\end{cor}

\sssec{}

Combining \lemref{l:co-action as right adjoint}(b) with \cite[Lemma 1.3.3]{DGCat}, we obtain:

\begin{cor} \label{c:inv and coinv}
There exists a canonical isomorphism of functors
$$\ccoinv_\CG \simeq \Iinv_\CG:\CG\mmod\to \StinftyCat_{\on{cont}}.$$
\end{cor}

It follows from the construction, the isomorphism of \corref{c:inv and coinv} lifts
to an isomorphism
\begin{equation}  \label{e:inv and coinv}
\ccoinv_\CG^{\on{enh}} \simeq \Iinv_\CG^{\on{enh}}
\end{equation}
as functors $\CG\mmod\to \Rep(\CG)\mmod$. 

\ssec{A criterion for 1-affineness in the rigid case}

The goal of this subsection is to prove the following assertion: 

\begin{prop} \label{p:1-aff and rigid'} 
Let $\CG$ be as in \secref{ss:rigid two}. Then the following conditions are equivalent:

\smallskip

\noindent{\em(a)} $B\CG$ is 1-affine;

\medskip

\noindent{\em(b)} The functor \eqref{e:recovery map}
is an equivalence.

\medskip

\noindent{\em(b')} There exists \emph{an} equivalence in $\CG\mmod$:
$$\Rrec^{\on{enh}}_\CG(\Vect)\simeq  \QCoh(\CG).$$

\medskip

\noindent{\em(c)} The functor $p_*:\QCoh(\CG)\to \Vect$, right adjoint
to $p^*$, is monadic.

\end{prop} 

\sssec{Step 1}

Let us assume (a). Then (b) expresses the fact that the functor $\bGamma^{\on{enh}}_{B\CG}$ is 
fully faithful.

\medskip

The implication (b) $\Rightarrow$ (b') is tautological. The implication (b') $\Rightarrow$ (a) is easy: we obtain that 
the functors $\ccoinv_\CG^{\on{enh}}$ and $\Rrec^{\on{enh}}_\CG$ are mutually inverse on the nose.

\sssec{Step 2}

It remains to establish the equivalence of (b) and (c). We claim that the functor
$$\Vect\to \Vect\underset{\Rep(\CG)}\otimes \Vect$$
is monadic, and the corresponding monad on $\Vect$ maps isomorphically, 
as a plain endo-functor, to $p_*\circ p^*$. 

\medskip

We will prove this by applying \corref{c:tensor prod Beck-Chevalley}. In Step 3 we will show
that the monoidal operation 
$$\Rep(\CG)\otimes \Rep(\CG)\to \Rep(\CG)$$
admits a left adjoint. Assuming this, the required assertion follows from 
\corref{c:tensor prod Beck-Chevalley}, combined with \lemref{l:verify tensor prod Beck-Chevalley}:

\medskip

Indeed, it remains to show that the map
$$\on{oblv}_\CG \circ (\on{oblv}_\CG)^R\to p_*\circ p^*$$
is an isomorphism, which follows from the fact that the corresponding map of
left adjoints
$$(p^*)^L\circ p^*\to \on{oblv}_\CG\circ \on{ind}_\CG$$
is an isomorphism, by \corref{c:induction}(b).

\sssec{Step 3} 

Using \corref{c:inv and coinv}, we interpret $\Rep(\CG)$ as 
$$\Vect\underset{\QCoh(\CG)_{\on{conv}^L}}\otimes \Vect.$$

Hence, we can identify
$$\Rep(\CG)\otimes \Rep(\CG)\simeq \Rep(\CG\times \CG),$$
so that the monoidal operation on $\Rep(\CG)$ identifies with restriction
under the diagonal map.

\medskip

We now claim that if $\phi:\CG_1\to \CG_2$ is any homomorphism between group-objects
of $\on{PreStk}$, satisfying the assumption of \secref{ss:rigid two}, then the 
restriction functor $\Rep(\CG_2)\to \Rep(\CG_1)$ admits a left adjoint.

\medskip

Indeed, interpreting $\Rep(\CG_i)$ as $\Vect\underset{\QCoh(\CG_i)_{\on{conv}^L}}\otimes \Vect$,
the left adjoint in question is given by the homomorphism
$$\QCoh(\CG_1)_{\on{conv}^L}\to \QCoh(\CG_2)_{\on{conv}^L},$$
which is the left adjoint (and simultaneously dual, see \propref{p:right adj and dual homomorphisms}) of the restriction map 
$\phi^*:\QCoh(\CG_2)\to \QCoh(\CG_1)$ of the corresponding co-monoidal categories.

\qed

\ssec{Classifying prestacks of formal groups}  \label{ss:formal groups}

In this subsection we will prove \thmref{t:main formal groups}. 
Let $\CG$ be a weakly $\aleph_0$ formally smooth formal group locally almost of finite type. 

\sssec{Proof of point (a)}

We will deduce the required assertion by applying \propref{p:main prop}.

\medskip

By definition,
$$B\CG:=|B^\bullet \CG|,$$
and we claim that this presentation satisfies the conditions of \propref{p:main prop}.
Note that each term of $B^\bullet\CG$ is of the form $\CG^{\times n}$. 

\medskip

Condition (1) is satisfied by \thmref{t:indsch}. 

\medskip

Condition (2) is satisfied by \cite[Theorem 10.1.1]{IndSch}: indeed, each of the functors
$$\Upsilon_{\CG^{\times n}}:\QCoh(\CG^{\times n})\to \IndCoh(\CG^{\times n})$$
is an equivalence. This also shows that
$$\QCoh(B\CG)\simeq \on{Tot}(\QCoh(B^\bullet\CG))\to \on{Tot}(\IndCoh(B^\bullet\CG))\simeq
\IndCoh(B\CG)$$
is an equivalence.

\medskip

Finally, condition (3) is satisfied because all the maps in $B^\bullet \CG$ are \emph{ind-proper}
(see \cite[Sect. 2.7.4 and Corollary 2.8.3]{IndSch}). 

\qed[\thmref{t:main formal groups}(a)]

\sssec{Proof of point (b)}

In \secref{sss:formal IndCoh}, we will show that $\CG$ satisfies the assumption of
\secref{ss:rigid two}. 

\medskip

Assuming this, in order to prove point (b) of the theorem,
by \propref{p:1-aff and rigid'}, it remains to show that the functor $p_*$, right
adjoint to $p^*$ is monadic if and only if the the tangent space of $\CG$ at the origin is 
finite-dimensional. We will do this in \secref{ss:monad on formal}. 

\sssec{} \label{sss:formal IndCoh}

Using \cite[Theorem 10.1.1]{IndSch}, we identify $\QCoh(\CG)\simeq \IndCoh(\CG)$
via the functor $\Upsilon_\CG$
as co-monoidal categories, where the co-monoidal structure on $\IndCoh(\CG)$
is induced by the structure on $\CG$ of group-object in $\on{PreStk}_{\on{laft}}$
via the operation of !-pullback. 

\medskip

Recall now the self-duality 
$$\IndCoh(\CG)\simeq \IndCoh(\CG)^\vee$$
(see \cite[Corollary 2.6.2]{IndSch}). 

\medskip

The above co-monoidal structure on $\IndCoh(\CG)$
defines via duality a monoidal structure on $\IndCoh(\CG)$; we shall denote the
resulting monoidal DG category by $\IndCoh(\CG)_{\on{conv}}$. By construction, 
the monoidal operation on $\IndCoh(\CG)_{\on{conv}}$ is given by the operation of
$(\IndCoh,*)$-direct image, i.e., it  is obtained by applying the (symmetric) monoidal functor 
$$\IndCoh_{\dgindSch_{\on{laft}}}:\dgindSch_{\on{laft}}\to \StinftyCat_{\on{cont}}$$
of \cite[Sect. 2.7]{IndSch} to the algebra object $\CG\in \dgindSch_{\on{laft}}$. 

\medskip

We claim:

\begin{lem} 
The monoidal DG category $\IndCoh(\CG)_{\on{conv}}$ is rigid.
\end{lem}

\begin{proof}
Follows from the fact that $\CG$ is ind-proper and the base change isomorphism
of \cite[Proposition 2.9.2]{IndSch}.
\end{proof}

This implies that $\QCoh(\CG)_{\on{conv}^L}$ is rigid as a monoidal DG category, as 
$$\IndCoh(\CG)_{\on{conv}} \simeq \QCoh(\CG)_{\on{conv}^L},$$
as monoidal DG categories, by construction.

\ssec{Computation of the monad}   \label{ss:monad on formal}

Let $\CG$ be as in \thmref{t:main formal groups}. 

\sssec{}

We identity $\QCoh(\CG)\simeq \IndCoh(\CG)$ by means of the functor $\Upsilon_\CG$, so that the functor $p^*$ corresponds to
$$p^!:\Vect\to \IndCoh(\CG),$$
and $p_*$ corresponds to the right adjoint of $p^!$, denoted $p_?$. 

\medskip

We will show that  the functor $p_?:\IndCoh(\CG)\to \Vect$ is monadic
if and only if the tangent space of $\CG$ at the origin is finite-dimensional.

\sssec{}

Note that this assertion does not involve the group structure on $\CG$. By \cite[Propositions 7.12.22 and 7.12.23]{BD},
the assumption
on $\CG$ implies that, as a DG indscheme, it can be described as follows:

\begin{itemize}

\item If the tangent space of $\CG$ at the origin is finite-dimensional, then $\CG$ is isomorphic 
to the completion of a the vector group $V$ at the origin, where $V\in \Vect^\heartsuit$ is finite-dimensional.

\item If the tangent space of $\CG$ at the origin is infinite-dimensional, then $\CG$ is isomorphic
to 
$$\underset{n}{\underset{\longrightarrow}{colim}}\, (\BA^n)^\wedge_0.$$

\end{itemize}

\sssec{}

Let $\CG=V^\wedge_0$. In this case the category $\IndCoh(\CG)$ identifies 
with $\Rep(V^*)$, where $V^*$ is the dual vector space. 

\medskip

Under this identification
$p^!$ identifies with the functor $\on{coind}_{V^*}$, the \emph{right} adjoint
to the forgetful functor $\oblv_G:\Rep(V^*)\to \Vect$. The functor $p_?$
is the (discontinuous) \emph{right} adjoint $(\on{coind}_{V^*})^R$ of
$\on{coind}_{V^*}$. 

\medskip

Hence, the functor $p_?$ is monadic by Remark 
\secref{r:usual coinduction}.

\begin{rem}
Let us note that for $\CG=V^\wedge_0$,  the assertion of \thmref{t:main formal groups}(b)
is equivalent to that of \thmref{t:BG} for the group $V^*$. Indeed, in this case
$$\IndCoh(\CG)\simeq \Rep(V^*),$$
as monoidal categories, so that $\CG\mmod\simeq \Rep(V^*)\mmod$ and the functor
$\Iinv_\CG^{\on{enh}}\simeq \ccoinv_\CG^{\on{enh}}$ identifies with $\Rrec^{\on{enh}}_{V^*}$. 
In addition, 
$$\Rep(\CG)\simeq \QCoh(V^*)_{\on{conv}},$$
also as monoidal categories, so
$$\Rep(\CG)\mmod\simeq V^*\mmod,$$
and the functor $\Rrec^{\on{enh}}_\CG$ idenitifies with $\ccoinv^{\on{enh}}_{V^*}$. 
\end{rem} 

\sssec{}

Let $$\CG:=\underset{n}{\underset{\longrightarrow}{colim}}\, (\BA^n)^\wedge_0.$$

We claim that in this case the functor $p_?$ fails to be conservative. Indeed,
let $\iota_n$ denote the embedding 
$$(\BA^n)^\wedge_0\hookrightarrow \CG.$$

We claim that the functor $p_?$ annihilates $(\iota_0)^{\IndCoh}_*(k)$. To prove this
we have to show that
$$\CMaps_{\IndCoh(\CG)}(\omega_{\CG},(\iota_0)^{\IndCoh}_*(k))=0.$$

We note that
$$\omega_{\CG}\simeq \underset{n}{\underset{\longrightarrow}{colim}}\, (\iota_n)^{\IndCoh}_*(\omega_{(\BA^n)^\wedge_0}),$$
and so
$$\CMaps_{\IndCoh(\CG)}(\omega_{\CG},(\iota_0)^{\IndCoh}_*(k))
\simeq 
\underset{n}{\underset{\longleftarrow}{lim}}\, \, 
\CMaps_{\IndCoh(\CG)}((\iota_n)^{\IndCoh}_*(\omega_{(\BA^n)^\wedge_0}),(\iota_0)^{\IndCoh}_*(k)).$$
Now, for every $n$, 
$$(\iota_n)^{\IndCoh}_*(\omega_{(\BA^n)^\wedge_0})\in \IndCoh(\BA^\infty)^{\leq -n},$$ so
$$\CMaps_{\IndCoh(\CG)}((\iota_n)^{\IndCoh}_*(\omega_{(\BA^n)^\wedge_0}),(\iota_0)^{\IndCoh}_*(k))\in \Vect^{\geq n},$$
and hence the above limit vanishes.

\section{De Rham prestacks}  \label{s:DR}

\ssec{De Rham prestacks of indschemes}  \hfill  \label{ss:DR}

\medskip

The goal of this subsection is to prove \thmref{t:main DR}. Recall that we fix a DG indscheme $\CZ$
locally almost of finite type, and we want to show that the prestack $\CZ_\dr$ is 1-affine.

\sssec{Step 1}

We will first prove that the prestack $Z_\dr$ is 1-affine, where $Z$ is an affine scheme of finite type. 
We can embed $Z$ into $\BA^n$. Since $Z_\dr$ identifies with its formal completion inside $(\BA^n)_\dr$,
by \thmref{t:main formal}, it is enough to consider the case of $Z=\BA^n$.

\medskip

Let $\CG$ be the formal completion of $\BA^n$ at the origin, considered as a formal group. Note that the
prestack quotient of $\BA^n$ by $\CG$ identifies with $(\BA^n)_\dr$. Hence, we have a canonical
map
$$(\BA^n)_\dr\to B\CG,$$
and for any $S\in (\affdgSch)_{B\CG}$, the fiber product 
$$S\underset{B\CG}\times (\BA^n)_\dr$$ 
identifies with $S\times \BA^n$. 

\medskip

Applying \thmref{t:main formal groups}(b) and \corref{c:1-affine base and fiber}, we deduce that $(\BA^n)_\dr$
is 1-affine. 

\sssec{Step 2}

We now claim that for an arbitrary scheme of finite type $Z$, the prestack $Z_\dr$ is 1-affine. Indeed, the reduction
to the affine case is routine and is left to the reader. 

\sssec{Step 3}

Let $\CZ$ be an indscheme written as
$$\underset{i\in I}{\underset{\longrightarrow}{colim}}\, Z_i,$$
where $Z_i$ are schemes of finite type, and the maps $f_{i,j}:Z_i\to Z_j$ are closed embeddings.

\medskip

The fact that the functor $\bLoc_{\CZ_\dr}$ is fully faithful follows from \propref{p:main prop}: indeed,
the functor
$$\Psi_{\CZ_\dr}:\QCoh(\CZ_\dr)\to \IndCoh(\CZ_\dr)$$
is an equivalence for \emph{any} $\CZ\in \on{PreStk}_{\on{laft}}$, see \cite[Proposition 2.4.4]{Crys}.

\sssec{Step 4}

It remains to show that for $\CC\in \on{ShvCat}(\CZ_\dr)$, the co-unit of the adjunction
$$\bLoc_{\CZ_\dr}\circ \bGamma_{\CZ_\dr}^{\on{enh}}(\CC)\to \CC$$
is an equivalence.

\medskip

Since the theorem has been established for schemes, it is sufficient to show that
for every index $i_0\in I$, the functor
\begin{equation} \label{e:need to prove dr indsch}
\QCoh((Z_{i_0})_\dr)\underset{\QCoh(\CZ_\dr)}\otimes \bGamma_{\CZ_\dr}(\CC)\simeq 
\bGamma\left((Z_{i_0})_\dr,\bLoc_{\CZ_\dr}\circ \bGamma^{\on{enh}}_{\CZ_\dr}(\CC)\right)\to \bGamma((Z_{i_0})_\dr,\CC)
\end{equation}
is an equivalence.

\medskip

As in the proof of \propref{p:main prop}, we can express $\bGamma_{\CZ_\dr}(\CC)$ as 
$$\underset{i\in I}{\underset{\longrightarrow}{colim}}\, \bGamma((Z_i)_\dr,\CC).$$

Since $I$ is filtered, the map $I_{i_0/}\to I$ is cofinal. Hence,
$$\bGamma_{\CZ_\dr}(\CC)\simeq \underset{i\in I_{i_0/}}{\underset{\longrightarrow}{colim}}\, \bGamma((Z_i)_\dr,\CC).$$

Therefore, the left-hand side in \eqref{e:need to prove dr indsch} identifies with
\begin{equation} \label{e:need to prove dr indsch next}
\underset{i\in I_{i_0/}}{\underset{\longrightarrow}{colim}}\,  \left(\QCoh((Z_{i_0})_\dr)\underset{\QCoh(\CZ_\dr)}\otimes \QCoh((Z_i)_\dr)\right)
\underset{\QCoh((Z_i)_\dr)}\otimes \bGamma((Z_i)_\dr,\CC).
\end{equation}

\medskip

Note, however, that for every $i\in I_{i_0/}$, the map 
$$\QCoh((Z_{i_0})_\dr)\underset{\QCoh(\CZ_\dr)}\otimes \QCoh((Z_i)_\dr)\to \QCoh((Z_{i_0})_\dr)$$
is an equivalence. Indeed, this follows by \lemref{l:tensor up local} from the fact that the restriction functor
$$\QCoh((Z_i)_\dr)\to \QCoh((Z_{i_0})_\dr)$$
admits a left adjoint that commutes with the $\QCoh((Z_i)_\dr)$-action. 

\medskip

Furthermore, the fact that \thmref{t:main DR} holds for schemes implies that
$$\QCoh((Z_{i_0})_\dr)\underset{\QCoh((Z_i)_\dr)}\otimes \bGamma((Z_i)_\dr,\CC)\to 
\bGamma((Z_{i_0})_\dr,\CC)$$
is an equivalence. 

\medskip

Hence, the expression in \eqref{e:need to prove dr indsch next} identifies with
$$\underset{i\in I_{i_0/}}{\underset{\longrightarrow}{colim}}\, \bGamma((Z_{i_0})_\dr,\CC).$$

However, since the category of indices is contractible, the resulting colimit is isomorphic to 
$\bGamma((Z_{i_0})_\dr,\CC)$, as required.

\qed

\sssec{} \label{sss:quotient by A-infty proof}

To conclude this subsection, consider the group prestack 
$$\CG=\underset{n}{\underset{\longrightarrow}{colim}}\, (\BG_a)^{\times n},$$
see \secref{sss:quotient by A-infty}. Let us show that $B\CG$ is not 1-affine.

\medskip

Let $\CG'$ be the formal completion of $\CG$ at the origin. I.e.,
$$\CG':=\underset{n}{\underset{\longrightarrow}{colim}}\, ((\BG_a)^{\times n})^\wedge_0.$$

\medskip

Consider the natural map
$$B\CG'\to B\CG.$$
Note that its base change by any $S\in \affdgSch_{/B\CG}$ yields the prestack 
$$S\times (\CG)_\dr,$$
which is 1-affine by \corref{c:product 1-affine} and \thmref{t:main DR}. 

\medskip

Assume for the sake of contradiction that $\bGamma^{\on{enh}}_{B\CG}$ was fully 
faithful. Then by  \propref{p:nice base change} and \propref{p:almost commute 1}(b),
we would obtain that $\bGamma^{\on{enh}}_{B\CG'}$
is also fully faithful. However, the latter is false by \thmref{t:main formal groups}(b). 

\ssec{De Rham prestacks of classifying stacks}  \label{ss:DR stack}

In this subsection we let $G$ be a classical affine algebraic group of finite type.

\sssec{}

Note that we tautologically have:
$$B(G_\dr)\simeq (BG)_\dr.$$

Next, we note that since the canonical map 
$$BG\to \on{pt}/G$$
becomes an isomorphism after the \'etale sheafification, the same is true for the map
$$(BG)_\dr\to (\on{pt}/G)_\dr.$$

Hence,
$$\on{ShvCat}((\on{pt}/G)_\dr)\simeq \on{ShvCat}((BG)_\dr)\simeq \on{ShvCat}(B(G_\dr)),$$
and by \thmref{t:main DR} and \secref{ss:groups}, we have
$$\on{ShvCat}(B(G_\dr))\simeq G_\dr\mmod.$$

\sssec{}

Let us now prove \propref{p:dr stack}. We will show that the functor $\bGamma_{B(G_\dr)}$ 
fails to be conservative for $G=\BG_a$.

\medskip

Consider the following two objects $\bD_1,\bD_2\in G_\dr\mmod$. Namely, we take $\bD_1=\Vect$,
with the trivial action, and $\bD_2:=\QCoh(G_\dr)$. There is a canonical map $\bD_2\to \bD_1$,
which is \emph{not} an equivalence. However, we claim that it becomes an equivalence after applying
the functor $\bGamma_{B(G_\dr)}$.

\medskip

Indeed, it is easy to see that $\Iinv_{G_\dr}(\QCoh(G_\dr))\simeq \Vect$. Note that
$$\Iinv_{G_\dr}(\Vect)\simeq \QCoh(B(G_\dr)).$$

\medskip

Thus, it remains to show that the natural functor
$$\Vect\to \QCoh(B(G_\dr))$$
is an equivalence for $G=\BG_a$. 

\sssec{}

We calculate $\QCoh(B(G_\dr))$ as
$$\on{Tot}(\QCoh((G^\bullet)_\dr)).$$

Note, however, that for $G=\BG_a$, for any $n$, the pullback functor 
$$\Vect\simeq \QCoh(\on{pt})\to \QCoh((G^n)_\dr)$$
is fully faithful.  Since it is an equivalence on $0$-simplices, we obtain that
$$\on{Tot}(\Vect^\bullet) \to \on{Tot}(\QCoh((G^\bullet)_\dr))$$
is an equivalence, where $\Vect^\bullet$ is the constant co-simplicial category
with value $\Vect$. 

\medskip

Since the category $\bDelta$ os contractible, we obtain that
$$\Vect\to \on{Tot}(\Vect^\bullet)$$
is also an equivalence, implying the desired assertion.

\ssec{Classifying prestack of a formal completion: Proof of \thmref{t:HCh}} \label{ss:HCh}

\sssec{}

Let $G$ be a classical affine algebraic group of finite type, and let $H\subset G$
be a closed subgroup. Let $\CG$ be denotes the formal completion of $H$
in $G$. 

\medskip

We need to show that the prestack $B\CG$ is 1-affine. 

\sssec{}

Consider the tautological homomorphism $\CG\to G$, and the resulting map
$$B\CG\to BG.$$

Since $BG$ is 1-affine (by \thmref{t:BG} and \corref{c:shv via Cech}(b)), by \corref{c:1-affine base and fiber},
in order to show that $B\CG$ is 1-affine, it suffices to show that for $S\in (\affdgSch)_{/BG}$, the prestack
$$S\underset{BG}\times B\CG$$
is 1-affine. 

\medskip

We note that any map $S\to BG$ factors as $S\to \on{pt}\to BG$, so 
$$S\underset{BG}\times B\CG\simeq S\times (\on{pt} \underset{BG}\times B\CG).$$

By \corref{c:product 1-affine}, we obtain that
it suffices to show that the prestack $\on{pt} \underset{BG}\times B\CG$ is 1-affine. 

\sssec{}

Note now that we have a canonical map
$$\on{pt} \underset{BG}\times B\CG\to (G/H)_{\dr},$$
which becomes an isomorphism after \'etale sheafification.

\medskip

This implies that $\on{pt} \underset{BG}\times B\CG$ is 1-affine by \thmref{t:main DR} and
\corref{c:shv via Cech}(b).

\section{Infinitesimal loop spaces} \label{s:inf loop}

\ssec{The setting}

\sssec{}

Consider the following situation. Let $Z$ be an affine DG scheme locally almost of finite type, and $\iota:\on{pt}\to Z$ a point
with image $z$.  

\medskip

Consider the adjoint pairs of functors:
$$\iota^*:\QCoh(Z)_{\{z\}}\rightleftarrows \Vect:\iota_*$$
and 
$$\iota_*:\Vect\rightleftarrows \QCoh(Z)_{\{z\}}:\iota^{\QCoh,!}.$$

\begin{conj} \label{conj:pre-loop}  
Assume that $Z$ is eventually coconnective. Then the functor $\iota^{\QCoh,!}$ is monadic.
\end{conj}

In \secref{ss:loop space} we will prove:

\begin{prop}  \hfill  \label{p:loop space}

\smallskip

\noindent{\em(1)} \conjref{conj:pre-loop} holds if $Z$ is smooth.

\smallskip

\noindent{\em(2)} \conjref{conj:pre-loop} holds if $Z$ is of the form $\on{pt}\underset{\BA^n}\times \on{pt}$.

\end{prop}

\begin{rem}
One can show that \propref{p:loop space} implies that \conjref{conj:pre-loop} holds for any $Z$,
which is \emph{quasi-smooth}.
\end{rem}

\ssec{Consequences of \conjref{conj:pre-loop}}

In this subsection we will assume that \conjref{conj:pre-loop} holds for a given $(Z,z)$, 
and deduce some consequences.

\sssec{}

Consider the group-object of $\affdgSch$
$$\Omega(Z,z):=\on{pt}\underset{Z}\times \on{pt},$$
i.e., the (derived!) inertia group of $Z$ at $z$.

\medskip

We will prove:

\begin{thm} \label{t:inf loops}
Assume that $(Z,z)$ satisfies \conjref{conj:pre-loop}. Then the prestack $B(\Omega(Z,z))$
is 1-affine and we have a canonical equivalence of symmetric monodical categories
$$\Rep(\Omega(Z,z))\simeq \QCoh(Z)_{\{z\}}.$$
\end{thm}

The rest of this subsection is devoted to the proof of this theorem.

\sssec{}

Note that the functor $\iota_*:\Vect\to \QCoh(Z)_{\{z\}}$ naturally upgrades to a functor:
\begin{equation} \label{e:loop space 1}
\Vect\underset{\QCoh(\Omega(Z,z))_{\on{conv}}}\otimes \Vect\to \QCoh(X)_{\{z\}}.
\end{equation}

Moreover, the functor \eqref{e:loop space 1} upgrades to a map in $\QCoh(Z)\mmod$, where $\QCoh(Z)$ 
acts on $\Vect$
via $\iota^*$.

\medskip

We claim (assuming that the pair $(Z,z)$ satisfies \conjref{conj:pre-loop}):

\begin{prop} \label{p:loop space cons} 
The functor \eqref{e:loop space 1} is an equivalence. 
\end{prop}

\begin{proof}

By \lemref{l:coinduction}, the right adjoint of the functor
$$\Vect\to \Vect\underset{\QCoh(\Omega(Z,z))_{\on{conv}}}\otimes \Vect$$
is monadic. Hence, to prove the assertion of the proposition, it remains to show that
the functor \eqref{e:loop space 1} induces an isomorphism of the resulting
monads on $\Vect$, regarded as plain endo-functors. 

\medskip

However, unwinding the definitions, we obtain that the resulting map of endo-functors
is 
$$p_*\circ (p_*)^R\simeq (p_*\circ p^*)^R\simeq (\iota^*\circ \iota_*)^R\simeq
(\iota_*)^R\circ \iota_*,$$
where the isomorphism $p_*\circ p^*\simeq \iota^*\circ \iota_*$ comes from base change
along the Cartesian diagram
$$
\CD
\on{pt}\underset{Z}\times \on{pt}   @>{p}>>  \on{pt} \\
@V{p}VV   @VV{\iota}V  \\
\on{pt}  @>{\iota}>> Z.
\endCD
$$

\end{proof}

\sssec{}

Consider now the category
$$\on{ShvCat}(Z^\wedge_{\{z\}}),$$
which according to \thmref{t:main formal}, identifies with 
$$\QCoh(Z^\wedge_{\{z\}})\mmod\simeq \QCoh(Z)_{\{z\}}\mmod.$$

\medskip

Consider the functor
\begin{equation} \label{e:taking the fiber}
\QCoh(Z)_{\{z\}}\mmod\to \StinftyCat_{\on{cont}}, \quad \bC\mapsto 
\Vect\underset{\QCoh(Z)_{\{z\}}}\otimes \bC\simeq \Vect\underset{\QCoh(Z)}\otimes \bC
\end{equation}
(the last equivalence is due to \lemref{l:tensor up local}).

\medskip

Note that since $\QCoh(Z)$ is rigid, by \corref{c:rigid Hochschild}, we can rewrite the functor \eqref{e:taking the fiber}
also as 
$$\uHom_{\QCoh(Z)}(\Vect,\bC)\simeq \uHom_{\QCoh(Z)_{\{z\}}}(\Vect,\bC).$$

We note that the functor \eqref{e:taking the fiber} naturally upgrades to a functor
\begin{equation} \label{e:taking the fiber enh}
\QCoh(Z^\wedge_{\{z\}})\mmod\to \Omega(Z,z)\mmod
\end{equation}
by regarding $\Vect$ as equipped with the trivial action of $\Omega(Z,z)$ that commutes with one
of $\QCoh(Z)$. 

\sssec{}

We claim (assuming that the pair $(Z,z)$ satisfies \conjref{conj:pre-loop}):

\begin{prop} \label{p:categories acted on by the loop group}
The functor \eqref{e:taking the fiber enh} is an equivalence.
\end{prop}

\begin{proof}
We construct a functor 
$$\Omega(Z,z)\mmod\to \QCoh(Z)\mmod$$
by 
\begin{equation} \label{e:going from fiber}
\bD\mapsto \Vect\underset{\QCoh(\Omega(Z,z))_{\on{conv}}}\otimes \bD;
\end{equation}

It is easy to see that the essential image of \eqref{e:going from fiber} lies in the full subcategory
$$\QCoh(Z^\wedge_{\{z\}})\mmod\subset \QCoh(Z)\mmod.$$

\medskip

We claim that the functors \eqref{e:taking the fiber} and \eqref{e:going from fiber}
are mutually inverse. Indeed, this follows from the (tautological) equivalence
\begin{equation} \label{e:QCoh on loop}
\Vect\underset{\QCoh(Z)_{\{z\}}}\otimes \Vect \simeq \Vect\underset{\QCoh(Z)}\otimes \Vect
\simeq \QCoh(\Omega(Z,z))
\end{equation}
combined with that of \eqref{e:loop space 1}.

\end{proof}

\sssec{}

Note that the functor $\iota^*$ naturally upgrades to a functor
\begin{equation} \label{e:loop space 2}
\QCoh(X)_{\{z\}}\to \uHom_{\QCoh(\Omega(Z,z))_{\on{conv}}}(\Vect,\Vect).
\end{equation} 

We claim (still assuming that the pair $(Z,z)$ satisfies \conjref{conj:pre-loop}):

\begin{prop}  \label{p:loop group reps}  
The functor \eqref{e:loop space 2} is an equivalence.
\end{prop}

\begin{proof}
Follows from \propref{p:categories acted on by the loop group}.
\end{proof}

\begin{cor} \hfill \label{c:B loop group prel}

\smallskip

\noindent{\em(a)} There exists a canonical equivalence
$$\Rep(\Omega(Z,z))\simeq \QCoh(Z)_{\{z\}}.$$

\smallskip

\noindent{\em(b)} The map \eqref{e:from coinv to inv groups} is an equivalence for 
$\Omega(Z,z)$.
\end{cor}

\begin{proof}
Point (a) is a reformulation of \propref{p:loop group reps}. Point (b) follows by combining point (a)
with \propref{p:loop space cons}. 
\end{proof}

Finally, we obtain (always assuming that the pair $(Z,z)$ satisfies \conjref{conj:pre-loop}):

\begin{cor}  \label{c:B loop group} 
The prestack $B(\Omega(Z,z))$ is 1-affine.
\end{cor}

\begin{proof}
Follows using \propref{p:1-aff and rigid}(b') from \corref{c:B loop group prel}(b) and 
the equivalence \eqref{e:QCoh on loop}.
\end{proof}

Note that Corollaries \ref{c:B loop group} and \ref{c:B loop group prel}(a) together amount to
the statement of \thmref{t:inf loops}. 

\ssec{Towards \conjref{conj:pre-loop}}   \label{ss:loop space}

\sssec{}

Consider the pair of adjoint functors
$$\iota_{\IndCoh,*}:\Vect\rightleftarrows \IndCoh(Z)_{\{z\}}:\iota^!.$$

We claim that the functor $\iota^!$ is monadic. Indeed, it is conservative by 
\cite[Proposition 4.1.7(a)]{IndCoh}, and is continuous. 

\medskip

Note that this implies the statement of \propref{p:loop space}(1), as in the smooth
case there is no difference between $\IndCoh$ and $\QCoh$.

\sssec{}

Let us show that the functor $\iota^{\QCoh,!}$ is conservative for $Z$ eventually coconnective. Recall
the functor
$$\Phi_Z:\QCoh(Z)\to \IndCoh(Z),$$
\emph{right} adjoint to the functor $\Psi_Z:\IndCoh(Z)\to \QCoh(Z)$. 

\medskip

Note that
$$\iota_*\simeq \Psi_Z\circ \iota_{\IndCoh,*},$$
and hence
$$\iota^{\QCoh,!}\simeq \iota^!\circ \Phi_Z.$$

We have just seen that the functor $\iota^!$ is conservative. Hence, it is enough to show that $Z$ eventually 
coconnective, the functor $\Phi_Z$ is conservative. 

\medskip

We claim that $\Phi_Z$ is in fact fully faithful. Indeed,
this follows from the fact that $\Psi_Z$ admits a fully faithful \emph{left} adjoint (see \cite[Proposition 1.5.3]{IndCoh}).

\medskip

Hence, we obtain that the monadicity of $\iota^{\QCoh,!}$ is equivalent to it satisfying the second condition in the
Barr-Beck-Lurie theorem.

\sssec{} \label{sss:loop space noncoconn}

Let us show that $\iota^{\QCoh,!}$ fails to be monadic for the non-eventually coconnective DG scheme
$Z=\on{pt}\underset{\on{pt}\underset{\BA^1}\times \on{pt}}\times \on{pt}$.
In fact, we claim that in this case, it fails to be conservative. 

\medskip

Indeed, $Z=\Spec(k[\xi])$, where $\deg(\xi)=-2$. The functor $\iota^{\QCoh,!}$ annihilates the module $k[\xi,\xi^{-1}]$. 

\ssec{Shift of grading and proof of \propref{p:loop space}(2)}

\sssec{Shift of grading}

In order to prove \propref{p:loop space}(2), we will use the ``shift of grading" trick (see, e.g., \cite[Sect. A.2]{AG}). 

\medskip

Consider
the symmetric monoidal DG category $\Rep(\BG_m)$, i.e., the category chain complexes of $\BZ$-graded vector spaces. 
It carries a canonical (symmetric monoidal self-equivalence), denoted 
$$M\mapsto M^{\on{shift}}.$$
Namely, the $n$-graded
piece of the $m$-th cohomology of $M^{\on{shift}}$ equals by definition the $n$-th graded piece of the $(m+2n)$-th
cohomology of $M$.

\medskip

If $A$ is an algebra object of $\Rep(\BG_m)$, we obtain an equivalence
\begin{equation} \label{e:shifted equivalence}
A\mod(\Rep(\BG_m))\simeq A^{\on{shift}}\mod(\Rep(\BG_m)).
\end{equation}

Note, however, that the equivalence \eqref{e:shifted equivalence} \emph{does not} commute with the
forgetful funcors
$$A\mod(\Rep(\BG_m))\to A\mod(\Vect) \text{ and } A^{\on{shift}}\mod(\Rep(\BG_m))\to A^{\on{shift}}\mod(\Vect).$$

\sssec{}

Let $\bO$ be an algebra object in the symmetric monoidal category $\on{ShvCat}(B\BG_m)\simeq \BG_m\mmod$,
and let $\bC_1$ and $\bC_2$ be right and left $\bO$-modules, respectively. Let
$$\bC_1\underset{\bO}\otimes \bC_2\to \bC$$
be a map in $\BG_m\mmod$.

\medskip

The following results from \thmref{t:BG} applied to $G=\BG_m$:

\begin{lem}  \label{l:grading trick}
The functor 
$\bC_1\underset{\bO}\otimes \bC_2\to \bC$ is an equivalence as plain DG categories if and only if
the functor
$$\Iinv_{\BG_m}(\bC_1)\underset{\Iinv_{\BG_m}(\bO)}\otimes \Iinv_{\BG_m}(\bC_2)\to \Iinv_{\BG_m}(\bC)$$
is an equivalence of plain DG categories.
\end{lem}

\sssec{Proof of \propref{p:loop space}(2)}  \label{sss:trick}

Let $V$ be a finite-dimensional vector space so that $\BA^n=\Spec(\on{Sym}(V))$.

\medskip

By \lemref{l:coinduction}, it is enough to show that the functor 
\begin{equation}  \label{e:for vector group shift}
\Vect\underset{\on{Sym}(V[2])\mod}\otimes \Vect\to \on{Sym}(V[1])\mod
\end{equation}
is an equivalence. 

\medskip 

We will deduce this from the fact that the functor
\begin{equation} \label{e:for vector group}
\Vect\underset{\on{Sym}(V)\mod}\otimes \Vect \to \on{Sym}(V[-1])\mod
\end{equation}
is an equivalence; the latter is due to \thmref{t:BG} and \propref{p:1-aff and rigid}, once we identify
$$\QCoh(V^*)_{\on{conv}}\simeq \on{Sym}(V)\mod \text{ and } \Rep(V^*)\simeq  \on{Sym}(V[-1])\mod.$$

\medskip

Indeed, by \lemref{l:grading trick}, to show that \eqref{e:for vector group shift} is an equivalence, it is enough
to show that 
$$\Rep(\BG_m)\underset{\on{Sym}(V[2])\mod(\Rep(\BG_m))}\otimes \Rep(\BG_m)\to \on{Sym}(V[1])\mod(\Rep(\BG_m))$$
is an equivalence. 

\medskip

By \eqref{e:shifted equivalence}, the latter is equivalent to 
$$\Rep(\BG_m)\underset{\on{Sym}(V)\mod(\Rep(\BG_m))}\otimes \Rep(\BG_m)\to \on{Sym}(V[-1])\mod(\Rep(\BG_m))$$
being an equivalence, which, again by \lemref{l:grading trick}, follows from the fact that \eqref{e:for vector group}
is an equivalence.

\section{Classifying prestacks of (co)-affine group-prestacks} \label{s:coaffine}

\ssec{The iterated classifying prestack}  

In this subsection we will prove \thmref{t:iterated B}. 

\sssec{Proof of \thmref{t:iterated B}(a)}

Let $V\in \Vect^\heartsuit$ be finite-dimensional. Set $\CG:=BV$; we already know that $\CG$ is 1-affine, so the category
$\on{ShvCat}(B\CG)$ can be described as in \secref{ss:groups}.  

\medskip

We note that there is a canonical equivalence
\begin{equation} \label{e:conv as ptw 1}
\QCoh(\CG)\simeq \QCoh(V^*)_{\{0\}},
\end{equation}
under which the co-monoidal structure on $\QCoh(\CG)$ (induced by the group structure on $\CG$) 
is obtained via the duality
$$\left(\QCoh(V^*)_{\{0\}}\right)^\vee\simeq \QCoh(V^*)_{\{0\}}$$
from the pointwise monoidal structure on $\QCoh(V^*)_{\{0\}}$. 

\medskip

Thus, the category $\CG\mmod$ can be identified with
$$\QCoh(V^*)_{\{0\}}\mmod,$$
with the functor $\Iinv_\CG$ being
$$\bD\mapsto \uHom_{\QCoh(V^*)_{\{0\}}}(\Vect,\bD),$$
where $\QCoh(V^*)_{\{0\}}$ acts on $\Vect$ via the functor $\iota^*$, 
where $\iota:\on{pt}\to V^*$ corresponds to $0\in V^*$. 

\medskip

Now, the assertion of  \thmref{t:iterated B}(a) follows from \propref{p:categories acted on by the loop group}
for $(Z,z)=(V^*,0)$.

\sssec{Proof of \thmref{t:iterated B}(b)}

Let $\CG_1=B^2(V)$. According to \thmref{t:iterated B}(a), proved above, $\CG_1$ is 1-affine. Hence,
the category $\on{ShvCat}(B\CG_1)$ can be described as in \secref{ss:groups}.

\medskip

Moreover, we have identified $\QCoh(\CG_1)$ as a plain category with $\QCoh(Z_1)$, where
$$Z_1:=\on{pt}\underset{V^*}\times \on{pt}.$$

\medskip

Under this identification, the co-monoidal structure on $\QCoh(\CG_1)$ (induced by the group structure on $\CG_1$) 
is obtained via the duality
$$\QCoh(Z_1)\simeq \QCoh(Z_1)^\vee,$$
from the pointwise monoidal structure on $\QCoh(Z_1)$. Since $\QCoh(Z_1)_{\on{ptw}}$ is rigid, the group $\CG_1$
satisfies the assumption of \secref{ss:rigid two}. 

\medskip

Thus, the category $\CG_1\mmod$ can be identified with
$$\QCoh(Z_1)\mmod,$$
with the functor $\Iinv_{\CG_1}$ being
$$\bD\mapsto \uHom_{\QCoh(Z_1)}(\Vect,\bD),$$
where $\QCoh(Z_1)$ acts on $\Vect$ via the functor $\iota_1^*$, 
where $\iota_1:\on{pt}\to Z_1$ is the unique $k$-point $z_1$ of $Z_1$. 

\medskip 

The category $\Rep(\CG_1)$ therefore identifies with $\QCoh(\Omega(Z_1,z_1))_{\on{conv}}$. 
We will show that $B\CG_1$ is 1-affine by applying \propref{p:1-aff and rigid'}(b'). Indeed, the condition that
$$\Vect\underset{\Rep(\CG_1)}\otimes \Vect\simeq \QCoh(\CG_1)$$
translates as
$$\Vect\underset{\QCoh(\Omega(Z_1,z_1))_{\on{conv}}}\otimes \Vect\simeq \QCoh(Z_1),$$
and follows from \propref{p:loop space cons} for $(Z,z)=(Z_1,z_1)$. 

\sssec{Proof of \thmref{t:iterated B}(c)}

Denote $\CG_2:=B^3(V)$. According to \thmref{t:iterated B}(a), proved above, $\CG_2$ is 1-affine. Hence,
the category $\on{ShvCat}(B\CG_2)$ can be described as in \secref{ss:groups}.

\medskip

Moreover, we have identified $\QCoh(\CG_2)$ as a plain DG category with $\QCoh(Z_2)$, where
$$Z_2:=\on{pt}\underset{Z_1}\times \on{pt}, \text{ where } Z_1:=\on{pt}\underset{V^*}\times \on{pt}.$$

\medskip

Under this identification, the co-monoidal structure on $\QCoh(\CG_2)$ (induced by the group structure on $\CG_2$)
is obtained via the duality
$$\QCoh(Z_2)\simeq \QCoh(Z_2)^\vee,$$
from the pointwise monoidal structure on $\QCoh(Z_2)$. Since $\QCoh(Z_2)_{\on{ptw}}$ is rigid, the group $\CG_2$
satisfies the assumption of \secref{ss:rigid two}. 

\medskip

Thus, the category $\CG_2\mmod$ can be identified with
$$\QCoh(Z_2)\mmod,$$
with the functor $\Iinv_{\CG_2}$ being
$$\bD\mapsto \uHom_{\QCoh(Z_2)}(\Vect,\bD),$$
where $\QCoh(Z_2)$ acts on $\Vect$ via the functor $\iota_2^*$, 
where $\iota_2:\on{pt}\to Z_2$ is the unique $k$-point $z_2$ of $Z_2$. 

\medskip

The category $\Rep(\CG_2)$ therefore identifies with $\QCoh(\Omega(Z_2,z_2))_{\on{conv}}$. 
We will show that $B\CG_2$ is \emph{not} 1-affine by applying \propref{p:1-aff and rigid'}(c). 

\medskip

The functor $p^*:\Vect\to \QCoh(\CG_2)$ is the right adjoint of the functor $\QCoh(\CG_2)_{\on{conv}^L}\to \Vect$
that defines the trivial action of $\CG_2$ on $\Vect$. Thus, under the identification
$$\QCoh(\CG_2)_{\on{conv}^L}\simeq \QCoh(Z_2)_{\on{ptw}},$$
the functor $p^*$ corresponds to $(\iota_2)_*$. Hence the functor
$(p_*)^R$ translates as $\iota_2^{\QCoh,!}$. However, we claim that $\iota_2^{\QCoh,!}$
is \emph{not} monadic. In fact, we claim that $\iota_2^{\QCoh,!}$ is \emph{not} conservative.
Indeed, this follows from \secref{sss:loop space noncoconn}. 

\sssec{Proof of \thmref{t:iterated B}(d)} 

Set $\CG_{\on{inf}}:=B(V^\wedge_{\{0\}})$. We know that $\CG_{\on{inf}}$ is 1-affine by \thmref{t:main formal groups}. 

\medskip

We note that $\CG_{\on{inf}}$ falls into the paradigm of \secref{ss:rigid two}, where
$$\QCoh(\CG_{\on{inf}})_{\on{conv}^L} \simeq \QCoh(V^*)_{\on{ptw}}.$$

\medskip

Under this identification, the trivial action of $\QCoh(\CG_{\on{inf}})_{\on{conv}^L}$ on $\Vect$ corresponds to 
$$\iota^*:\QCoh(V^*)\to \Vect,$$
where $\iota:\on{pt}\to V^*$ corresponds to $0\in V^*$. 

\medskip

We will show that $B\CG_{\on{inf}}$ is \emph{not} 1-affine by applying \propref{p:1-aff and rigid'}(c). 
We note that the functor 
$$p^*:\Vect\to \QCoh(\CG_{\on{inf}})$$
is the \emph{right} adjoint to one corresponding to the augmentation $\QCoh(\CG_{\on{inf}})\to \Vect$.
I.e., $p^*$ indentifies with the functor
$$\iota_*:\Vect\to \QCoh(V^*).$$

Its right adjoint is, therefore, the functor
$$\iota^{\QCoh,!}:\QCoh(V^*)\to \Vect.$$

However, it is clear that $\iota^{\QCoh,!}$ is not conservative: it annihilates any object that comes
as direct image from $V^*-\{0\}$. 

\ssec{Group DG schemes}  \label{ss:vector}

In this subsection we will prove \thmref{t:vector}.

\sssec{}

Denote $\CG_n=\Spec(\on{Sym}(V[n]))$. The case $n=0$ follows from \thmref{t:BG} for $G=V^*$. 
For $n>0$, we will consider the cases of $n$ even and odd separately.

\sssec{}

Let first $n=2$.  In this case, the assertion follows from \corref{c:B loop group} applied to
$Z:=\Spec(\on{Sym}(V[1]))$. 

\medskip

Let now $n$ be an arbitrary even integer. The required assertion follows \propref{p:1-aff and rigid}(b)
and the case of $n=2$, using the ``shift of grading" trick, as in \secref{sss:trick}. 

\sssec{}

Let us now take $n=1$. In this case the assertion follows from \corref{c:B loop group} applied to
$Z:=\Spec(\on{Sym}(V))$. 

\medskip

Let now $n$ be an arbitrary odd integer. The required assertion follows \propref{p:1-aff and rigid}(b)
and the case of $n=1$, using the ``shift of grading" trick, as in \secref{sss:trick}. 

\appendix

\section{Descent theorems}  \label{s:proof of descent}

\ssec{Descent for module categories}

In this subsection we will prove \thmref{t:descent 1}. Let $Y$ be an affine DG scheme. 

\sssec{Step 1}

Let 
$$\CP:(\affdgSch_{/Y})^{\on{op}}\to \CT$$
be a functor, where $\CT$ is some $\infty$-category. It is well-known that if $\CP$
satisfies descent with respect to finite flat maps and Nisnevich covers, then $\CP$
satisfies fppf descent. This follows from the fact that, Nisnevich-locally, any fppf
map admits a section after a finite flat base change. 

\medskip

We let $\CT=\StinftyCat_{\on{cont}}$ and $\CP$ be the functor
$$S\in \affdgSch_{/Y}\rightsquigarrow \QCoh(S)\underset{\QCoh(Y)}\otimes \bC.$$

\medskip

In Step 2 we will show that $\CT$ satisfies Nisnevich descent, and in Step 3 we will
show that $\CT$ satisfies finite flat descent. This will prove \thmref{t:descent 1}.

\sssec{Step 2}  \label{sss:Nisn}

For Nisnevich descent, it is enough to consider the case of \emph{basic Nisnevich covers}. I.e.,
let $\oS\overset{\jmath}\hookrightarrow S$ be an open embedding, and $\pi:S_1\to S$ an
\'etale map, such that $\pi$ is one-to-one over $S-\oS$. Set $\oS_1:=\oS\underset{S}\times S_1$.
Let $\opi:\oS_1\to \oS$ and $\jmath_1:\oS_1\to S_1$ denote the corresponding morphisms.

\medskip

We need to show that 
$$
\CD
\QCoh(S)\underset{\QCoh(Y)}\otimes \bC @>{\jmath^*\otimes \on{Id}_\bC}>>  \QCoh(\oS)\underset{\QCoh(Y)}\otimes \bC \\
@V{\pi^*\otimes \on{Id}_\bC}VV    @VV{\opi^*\otimes \on{Id}_\bC}V  \\
\QCoh(S_1)\underset{\QCoh(Y)}\otimes \bC @>{\jmath_1^*\otimes \on{Id}_\bC}>>  \QCoh(\oS_1)\underset{\QCoh(Y)}\otimes \bC
\endCD
$$
is a pull-back square. 

\medskip

By definition, the pull-back
$$\left(\QCoh(S_1)\underset{\QCoh(Y)}\otimes \bC\right) \underset{\QCoh(\oS_1)\underset{\QCoh(Y)}\otimes \bC}\times
\left(\QCoh(\oS)\underset{\QCoh(Y)}\otimes \bC\right)$$
is the category of quintuples
\begin{multline} \label{e:fiber product}
\{\bc_1\in \QCoh(S_1)\underset{\QCoh(Y)}\otimes \bC, \quad \obc_1\in \QCoh(\oS_1)\underset{\QCoh(Y)}\otimes \bC, \quad
\obc\in \QCoh(\oS)\underset{\QCoh(Y)}\otimes \bC, \\
\alpha: (\jmath_1^*\otimes \on{Id}_\bC)(\bc_1)\simeq \obc_1, \quad \quad 
\beta:(\opi^*\otimes \on{Id}_\bC)(\obc)\simeq \obc_1.\}
\end{multline}
 
We define the right adjoint to the natural functor 
$$ \QCoh(S)\underset{\QCoh(Y)}\otimes \bC \to \QCoh(S_1)\underset{\QCoh(Y)}\otimes \bC 
\underset{ \QCoh(\oS_1)\underset{\QCoh(Y)}\otimes \bC}\times \QCoh(\oS)\underset{\QCoh(Y)}\otimes \bC$$
by sending a quintuple as in \eqref{e:fiber product} to 
$$\on{Cone}\left((\pi_*\otimes \on{Id}_\bC)(\bc_1)\oplus (\jmath_*\otimes \on{Id}_\bC)(\obc)\to 
((\jmath\circ \opi)_*\otimes \on{Id}_\bC)(\obc_1)\right)[-1].$$

Thus, we obtain that the above right adjoint is obtained from the right adjoint for $\bC:=\QCoh(Y)$ by tensoring by
$-\underset{\QCoh(Y)}\otimes \bC$. Therefore, since the unit and the co-unit of the adjunction are
isomorphisms in the former case (by Nisnevich descent for $\QCoh$), they are isomorphisms for any $\bC$.

\begin{rem}  \label{r:Nisn}
It is very easy to prove directly that for $\QCoh(Y)$, the unit and co-unit map are isomorphisms, see \cite[Lemma 2.2.6]{QCoh}.
\end{rem}

\sssec{Step 3}

Let now $\pi:T\to S$ be a finite faithfully flat map, and let $T^\bullet/S$ be its \v{C}ech nerve. We need to show that the functor
$$\QCoh(S)\underset{\QCoh(Y)}\otimes \bC\to
\on{Tot}\left(\QCoh(T^\bullet/S)^*\underset{\QCoh(Y)}\otimes \bC\right)$$
is an equivalence.

\medskip

We note that for every map $\alpha:[j]\to [i]$ and the corresponding map $f^\alpha:T^i/S\to T^j/S$, the functor
$(f^\alpha)^*$ admits a \emph{left} adjoint, denoted $f^\alpha_\sharp$. Namely, for any finite flat map $f$,
the functor $f_\sharp$ is the ind-extension of the functor on $\QCoh(-)^{\on{perf}}$, defined as
$$\CE\mapsto (f_*(\CE^\vee))^\vee.$$

\medskip

In particular, we obtain that the functors $(f^\alpha)^*\otimes \on{Id}_\bC$ all admit left adjoints, given by
$f^\alpha_\sharp\otimes \on{Id}_\bC$. Hence, by \cite[Lemma 1.3.3]{DGCat}, we have
$$\on{Tot}\left(\QCoh(T^\bullet/S)^*\underset{\QCoh(Y)}\otimes \bC\right)\simeq
|\QCoh(T^\bullet/S)_\sharp\underset{\QCoh(Y)}\otimes \bC|.$$

From the commutative diagram
$$
\CD
\on{Tot}\left(\QCoh(T^\bullet/S)^*\right) \underset{\QCoh(Y)}\otimes \bC @>>> 
|\QCoh(T^\bullet/S)_\sharp|\underset{\QCoh(Y)}\otimes \bC  \\
@VVV   @VV{\sim}V \\
\on{Tot}\left(\QCoh(T^\bullet/S)^*\underset{\QCoh(Y)}\otimes \bC\right)  @>>> 
|\QCoh(T^\bullet/S)_\sharp\underset{\QCoh(Y)}\otimes \bC|
\endCD
$$
we obtain that the functor
$$\on{Tot}\left(\QCoh(T^\bullet/S)^*\right) \underset{\QCoh(Y)}\otimes \bC\to 
\on{Tot}\left(\QCoh(T^\bullet/S)^*\underset{\QCoh(Y)}\otimes \bC\right)$$
is an equivalence.

\medskip

Now, the required assertion follows from the commutative diagram
$$
\CD
\QCoh(S)\underset{\QCoh(Y)}\otimes \bC  @>>>  \on{Tot}\left(\QCoh(T^\bullet/S)^*\right) \underset{\QCoh(Y)}\otimes \bC  \\
@V{\on{Id}}VV   @VV{\sim}V  \\
\QCoh(S)\underset{\QCoh(Y)}\otimes \bC  @>>>  \on{Tot}\left(\QCoh(T^\bullet/S)^*\underset{\QCoh(Y)}\otimes \bC\right)
\endCD
$$
and the fact that 
$$\QCoh(S)\to \on{Tot}\left(\QCoh(T^\bullet/S)^*\right)$$
is an equivalence, by the fppf descent for $\QCoh$. 

\begin{rem}
One can show directly that $\QCoh(S)\to \on{Tot}\left(\QCoh(T^\bullet/S)^*\right)$ is an equivalence
(so that together with Remark \ref{r:Nisn} one obtains an alternative proof of fppf descent for $\QCoh$,
without apealing to the t-structures).  

\begin{proof}[Proof of finite flat descent for $\QCoh$] \hfill

\smallskip

\noindent By \lemref{l:monadicity}, the functor of evaluation on $0$-simplices
$$\on{Tot}\left(\QCoh(T^\bullet/S)^*\right)\to \QCoh(T)$$
is monadic, and the corresponding monad on $\QCoh(T)$, viewed as a plain
endo-functor, is canonically isomorphic to $(\on{pr}_2)_\sharp\circ \on{pr}_1^*$,
where $\on{pr}_1,\on{pr}_2:T\underset{S}\times T\rightrightarrows T$ are the two
projections. Note also that we have a canonical isomorphism of endo-functors:
$$(\on{pr}_2)_\sharp\circ \on{pr}_1^*\simeq ((\on{pr}_1)_*\circ \on{pr}_2^*)^L\simeq
(\pi^*\circ \pi_*)^L\simeq \pi^*\circ \pi_\sharp.$$

Hence, it remains to show that the functor $\pi^*:\QCoh(S)\to \QCoh(T)$ is monadic
(we have just seen that the corresponding monad $\pi^*\circ \pi_\sharp$ maps isomorphically
to one defining $\on{Tot}\left(\QCoh(T^\bullet/S)^*\right)$). 

\medskip

Since $\pi$ is faithfully flat, it is easy to see that $\pi^*$ is conservative. Now, since $\pi^*$
is continuous, it commutes with all geometric realizations. Hence, the monadicity of $\pi^*$
follows from the Barr-Beck-Lurie theorem. 

\end{proof} 

\end{rem}

\ssec{Descent for sheaves of categories}

In this subsection we will prove \thmref{t:descent 2}. 

\sssec{Step 0}

Let $\pi:T\to S$ be an fppf cover, and let $T^\bullet/S$ be its \v{C}ech nerve. We need to show that
the functor
\begin{equation} \label{e:descent functor}
\on{ShvCat}(S)\to \on{Tot}(\on{ShvCat}(T^\bullet/S))
\end{equation}
is an equivalence.

\medskip

Since all the DG schemes involved are affine, we have
$$\on{ShvCat}(S)=\QCoh(S)\mmod \text{ and } \on{ShvCat}(T^\bullet/S)=\QCoh(T^\bullet/S)\mmod.$$

\medskip

The right adjoint to the functor \eqref{e:descent functor} sends 
$$\bC^\bullet\in \on{Tot}(\QCoh\left(T^\bullet/S)\mmod\right)\rightsquigarrow \on{Tot}(\bC^\bullet),$$
where the totalization is taken in $\QCoh(S)\mmod$.

\medskip

We need to show that the functor \eqref{e:descent functor} and its right adjoint are mutually inverse.
We will check that the unit and the co-unit of the adjunction are isomorphisms.

\sssec{Step 1}

The fact that the unit of the adjunction is an isomorphism follows immediately from \thmref{t:descent 1}.

\sssec{Step 2}

To say that the co-unit of the adjunction is an isomorphism is equivalent to saying that 
for $\bC^\bullet\in \on{Tot}\left(\QCoh(T^\bullet/S)\mmod\right)$ the functor
$$\QCoh(T)\underset{\QCoh(S)}\otimes \on{Tot}(\bC^\bullet)\to \bC^0$$
is an equivalence.

\medskip

By \lemref{l:dualizable in rigid}, $\QCoh(T)$ is dualiazable as an object of $\QCoh(S)\mmod$. Hence,
the functor
$$\QCoh(T)\underset{\QCoh(S)}\otimes \on{Tot}(\bC^\bullet)\to
\on{Tot}(\QCoh(T)\underset{\QCoh(S)}\otimes \bC^\bullet)$$
is an equivalence. Hence, it remains to show that the functor
\begin{equation} \label{e:split abs}
\on{Tot}(\QCoh(T)\underset{\QCoh(S)}\otimes \bC^\bullet)\to \bC^0
\end{equation}
is an equivalence.

\medskip

Now, we note that the co-simplicial category $\QCoh(T)\underset{\QCoh(S)}\otimes \bC^\bullet$
identifies with $\bC^{\bullet+1}$, which is split, and its map to $\bC^0$ is the augmentation. Hence,
\eqref{e:split abs} is an equivalence, as desired. 

\section{Quasi-affine morphisms}   \label{s:quasi-affine}

\ssec{Fiber products of prestacks vs. tensor products of categories}  

\sssec{}

Let 
$$
\CD
\CY'_1  @>{g_1}>> \CY_1  \\
@V{f'}VV   @VV{f}V   \\
\CY'_2  @>{g_2}>> \CY_2
\endCD
$$
be a Cartesian diagram in $\on{PreStk}$. It gives rise to a (symmetric monoidal) functor
\begin{equation} \label{e:fiber vs ten}
\QCoh(\CY'_2)\underset{\QCoh(\CY_2)}\otimes \QCoh(\CY_1)\to \QCoh(\CY'_1).
\end{equation}

In this section we will discuss two instances in which the functor \eqref{e:fiber vs ten}
is an equivalence.

\sssec{}

We will prove:

\begin{prop} \label{p:quasi-aff}
Assume that $f$ (and hence $f'$) is quasi-affine and quasi-compact. Then \eqref{e:fiber vs ten} is an equivalence.
\end{prop} 

The proof will use the following lemma, proved in \secref{sss:q-aff alg}:

\begin{lem} \label{l:q-aff alg}
Let $f:\CY_1\to \CY_2$ be a quasi-affine quasi-compact map, and consider $f_*(\CO_{\CY_1})$
as an assciative algebra in $\QCoh(\CY_2)$. Then the functor
$$\QCoh(\CY_1)\to f_*(\CO_{\CY_1})\mod(\QCoh(\CY_2))$$
is an equivalence.
\end{lem}

\begin{proof}[Proof of \propref{p:quasi-aff}]

By \lemref{l:q-aff alg}, we have:
$$\QCoh(\CY_1)\simeq f_*(\CO_{\CY_1})\mod(\QCoh(\CY_2)) \text{ and }
\QCoh(\CY'_1)\simeq f'_*(\CO_{\CY'_1})\mod(\QCoh(\CY'_2)).$$

Hence, it is sufficient to show that the functor $g_2^*$ induces an equivalence
$$\QCoh(\CY'_2)\underset{\QCoh(\CY_2)}\otimes \left(f_*(\CO_{\CY_1})\mod(\QCoh(\CY_2))\right)\to
f'_*(\CO_{\CY'_1})\mod(\QCoh(\CY'_2)).$$

By base change, $f'_*(\CO_{\CY'_1})\simeq g^*(f_*(\CO_{\CY_1}))$. Now, the required assertion follows
from the following general lemma:

\begin{lem} \label{l:algebras}
Let $\bO$ be a monoidal DG category, $\bC$ a left $\bO$-module category, and $A\in \bO$ an algebra.
Consider $A\mod(\bO)$ as a right $\bO$-module category. Then the natural functor
$$A\mod(\bO)\underset{\bO}\otimes \bC\to A\mod(\bC)$$
is an equivalence.
\end{lem}

\end{proof}

\sssec{Proof of \lemref{l:algebras}}

We have an adjoint pair
$$\on{ind}_{A,\bO}:\bO\rightleftarrows A\mod(\bO):\on{oblv}_{A,\bO}$$
as right $\bO$-module categories. Tensoring up on the right with $\bC$ over $\bO$,
we obtain an adjoint pair
$$(\on{ind}_{A,\bO}\otimes \on{Id}_\bC):\bC\rightleftarrows A\mod(\bO)\underset{\bO}\otimes \bC:
(\on{oblv}_{A,\bO}\otimes \on{Id}_\bC).$$
The functor $(\on{oblv}_{A,\bO}\otimes \on{Id}_\bC)$ is monadic: indeed, it commutes with all colimits
and its left adjoint generates $A\mod(\bO)\underset{\bO}\otimes \bC$.

\medskip

Consider also the adjoint pair
$$\on{ind}_{A,\bC}:\bC\rightleftarrows A\mod(\bC):\on{oblv}_{A,\bC}.$$
The functor $\on{oblv}_{A,\bC}$ is tautologically monadic. 

\medskip

Hence, to prove the lemma, it suffices to show that the functor $A\mod(\bO)\underset{\bO}\otimes \bC\to A\mod(\bC)$
defines an isomorphism of the corresponding monads on $\bC$ as plain endo-functors. However, this follows
from the fact that
$$(\on{oblv}_{A,\bO}\otimes \on{Id}_\bC)\circ (\on{ind}_{A,\bO}\otimes \on{Id}_\bC)\simeq ((A\otimes -)\otimes \on{Id}_\bC)$$
identifies with the action of $A$ on $\bC$.

\qed 

\sssec{Proof of \lemref{l:q-aff alg}}  \label{sss:q-aff alg}

We shall deduce the assertion from the Barr-Beck-Lurie theorem. Since the morhism $f$
is quasi-compact and quasi-separated, it satisfies base change and the projection formula. 
and the functor $f_*$ is continuous.

\medskip

The projection formula implies that the monad $f_*\circ f^*$ is given by tensor product with $f_*(\CO_{\CY_1})$.
The continuity of $f_*$ implies that it commutes with all geometric realizations. It remains to show that $f_*$
is conservative. 

\medskip

By base change, the conservativity assertion reduces to the case when $\CY_2$ is an affine DG scheme,
in which case $\CY_1$ is quasi-affine. In this case, the conservativity of $f_*$ is equivalent to that of 
$\Gamma(\CY_2,-)$.

\medskip 

Thus, we need to show that for a quasi-affine DG scheme $Y$,
the functor $$\Gamma(Y,-):\QCoh(Y)\to \Vect$$ is conservative. Let $\jmath:Y\to \ol{Y}$ be an open embedding,
where $\ol{Y}$ is affine. Now, the functors $\jmath_*$ and $\Gamma(\ol{Y},-)$ are both conservative, and
hence so is $\Gamma(Y,-)$.

\qed

\ssec{Fiber products of passable prestacks}

\sssec{}

Recall the notion of passable prestack, see \secref{ss:passable}. We are going to prove:  

\begin{prop}  \label{p:base change pass}
Assume that $\CY_2$ is passable and that $\QCoh(\CY_1)$ is dualizable as a DG category.
Then \eqref{e:fiber vs ten} is an equivalence.
\end{prop}

The proof will use the following assertion:

\begin{lem} \label{l:pass rigid}
Let $\CY\in \on{PreStk}$ be such that the diagonal morphism is representable, quasi-compact and quasi-separated,
and $\CO_\CY\in \QCoh(\CY)$ is compact. Then the following conditions are equivalent:

\smallskip

\noindent{\em(a)} The functor $\QCoh(\CY)\otimes \QCoh(\CY')\to \QCoh(\CY\times \CY')$ is an equivalence
for any $\CY'\in \on{PreStk}$. 

\smallskip

\noindent{\em(b)} The functor $\QCoh(\CY)\otimes \QCoh(\CY)\to \QCoh(\CY\times \CY)$ is an equivalence. 

\smallskip

\noindent{\em(c)} The monoidal DG category $\QCoh(\CY)$ is rigid.

\smallskip

\noindent{\em(d)} The category $\QCoh(\CY)$ is dualizable as a DG category. 

\end{lem}

\begin{proof}[Proof of \propref{p:base change pass}]

We have
$$\QCoh(\CY'_1)\simeq \underset{S\in \affdgSch_{/\CY'_2}}{\underset{\longleftarrow}{lim}}\, \QCoh(S\underset{\CY_2}\times \CY_1).$$

\medskip

By \lemref{l:pass rigid}, the category $\QCoh(\CY_2)$ is rigid. Hence, by \lemref{l:dualizable in rigid}, the functor
$$\bC\mapsto \bC\underset{\QCoh(\CY_2)}\otimes \QCoh(\CY_1),\quad \QCoh(\CY_2)\mmod\to \StinftyCat_{\on{cont}}$$
commutes with limits. In particular, the functor
\begin{multline*}
\QCoh(\CY'_2)\underset{\QCoh(\CY_2)}\otimes \QCoh(\CY_1)=
\left(\underset{S\in \affdgSch_{/\CY'_2}}{\underset{\longleftarrow}{lim}}\, \QCoh(S)\right)
\underset{\QCoh(\CY_2)}\otimes \QCoh(\CY_1)\to \\
\to \underset{S\in \affdgSch_{/\CY'_2}}{\underset{\longleftarrow}{lim}}\,  \left(\QCoh(S)\underset{\QCoh(\CY_2)}\otimes \QCoh(\CY_1)\right)
\end{multline*}
is an equivalence.

\medskip

This reduces the assertion to the case when $\CY'_2=S\in \affdgSch$. However, in the latter case, the morphism $g_2$
is quasi-affine, and the assertion follows from \propref{p:quasi-aff}.

\end{proof}

\sssec{Proof of \lemref{l:pass rigid}}   \label{sss:pass rigid}

The implications (a) $\Rightarrow$ (b) and (c) $\Rightarrow$ (d) are tautological. The implication (d) $\Rightarrow$ (a) is
\cite[Proposition 1.4.4]{QCoh}. Thus, it remains to show that (b) implies (c). 

\medskip

Thus, we need to show that the right adjoint to the monoidal operation
\begin{equation} \label{e:mon op}
\QCoh(\CY)\otimes \QCoh(\CY)\to \QCoh(\CY)
\end{equation}
is continuous, and is a map of $\QCoh(\CY)$-bimodule categories. 

\medskip

Using the assumption in (b), we identify the functor \eqref{e:mon op} with the functor
$$\Delta_\CY^*:\QCoh(\CY\times \CY)\to \QCoh(\CY).$$

The required properties follow now from the assumptions on the morphism $\Delta_\CY$.

\qed

\section{Beck-Chevalley conditions} \label{s:Beck-Chevalley}

\ssec{Monadic and co-monadic Beck-Chevalley conditions}  \label{ss:BC}

\sssec{}

Let $\bC^\bullet$ be a co-simplicial category; for a map $\alpha:[j]\to [i]$ in $\bDelta$, let
$\sT^\alpha:\bC^j\to \bC^i$ denote the corresponding functor.

\begin{defn}
We shall say that $\bC^\bullet$ satisfies the \emph{monadic} 
Beck-Chevalley condition, if the following holds:

\begin{itemize}

\item For every $i$ and the last face map $\partial_i:[i]\to [i+1]$, the functor $\sT^{\partial_i}$
admits a left adjoint. 

\item For every $\alpha:[j]\to [i]$ (and the corresponding map $\alpha+1:[j+1]\to [i+1]$), the diagram
$$
\CD
\bC^i  @<{(\sT^{\partial_i})^L}<<  \bC^{i+1}  \\
@A{\sT^\alpha}AA   @AA{\sT^{\alpha+1}}A \\
\bC^j @<{(\sT^{\partial_j})^L}<<   \bC^{j+1},
\endCD
$$
that a priori commutes up to a natural transformation, commutes.

\end{itemize}

\end{defn}

\begin{defn} 
We say that $\bC^\bullet$ satisfies the \emph{co-monadic} Beck-Chevalley condition, when we replace 
``left adjoint" by ``right adjoint" in the above definition.
\end{defn}

\sssec{}

We also give the following definition:

\begin{defn} \label{defn:simpl mon}
Let $\bC^\bullet$ be a simplicial category. 
We shall say that $\bC^\bullet$ satisfies the \emph{monadic} Beck-Chevalley condition if for every $i$ and the last face map 
$\partial_i:[i]\to [i+1]$, the functor $\sT^{\partial_i}:\bC^{i+1}\to \bC^i$ admits a left adjoint and for every $\alpha:[j]\to [i]$ the diagram
$$
\CD
\bC^i  @>{(\sT^{\partial_i})^L}>>  \bC^{i+1}  \\
@V{\sT^\alpha}VV   @VV{\sT^{\alpha+1}}V \\
\bC^j @>{(\sT^{\partial_j})^L}>>   \bC^{j+1},
\endCD
$$
that a priori commutes up to a natural transformation, commutes.
\end{defn}

The following (tautological) observation is often useful:

\begin{lem} \label{l:double right adjoint new} 
Let $\bC^\bullet$ be a simplicial category, in which for every $([j]\overset{\alpha}\to [i]) \in \bDelta$,
the corresponding functor $\sT^\alpha:\bC^i\to \bC^j$ admits a right adjoint, and every $i$ and the last face map 
$\partial_i:[i]\to [i+1]$, the functor $\sT^{\partial_i}$ admits a left adjoint. Then the co-simplicial category 
$\bC^{\bullet,R}$, obtained by passing to the right adjoints, satisfies 
the monadic Beck-Chevalley condition if and only if $\bC^\bullet$ does.
\end{lem}

Interchanging the words ``left" and ``right" in Definition \ref{defn:simpl mon} and \lemref{l:double right adjoint new}
we obtain the dual definition and assertion for the co-monadic Beck-Chevalley conditions. 

\sssec{}

We have the following basic results (see \cite[Theorem 6.2.4.2]{Lu2}):

\begin{lem}  \label{l:monadicity}
Let $\bC^\bullet$ satisfy the monadic Beck-Chevalley condition. Then:

\smallskip

\noindent{\em(a)} 
The functor of evaluation on $0$-simplices 
$$\on{ev}^0:\on{Tot}(\bC^\bullet)\to \bC^0$$
admits a left adjoint; to be denoted $(\on{ev}^0)^L$.

\smallskip

\noindent{\em(b)} The monad $\on{ev}^0\circ (\on{ev}^0)^L$, acting on $\bC^0$, is isomorphic,
as a plain endo-functor, to $$(\sT^{\on{pr}_s})^L\circ \sT^{\on{pr}_t},$$ where
$\on{pr}_s,\on{pr}_t$ are the two maps $[0]\to [1]$\footnote{The notation ``s" is for ``source" and ``t" for ``target."}.

\smallskip

\noindent{\em(c)} The functor 
$$\on{ev}^0: \on{Tot}(\bC^\bullet)\to \bC^0$$
is monadic. 

\end{lem}

Similarly, we have:

\begin{lem}  \label{l:comonadicity}
Let $\bC^\bullet$ satisfy the co-monadic Beck-Chevalley condition. Then:

\smallskip

\noindent{\em(a)} 
The functor of evaluation on $0$-simplices 
$$\on{ev}^0:\on{Tot}(\bC^\bullet)\to \bC^0$$
admits a right adjoint; to be denoted $(\on{ev}^0)^R$.

\smallskip

\noindent{\em(b)} The co-monad $\on{ev}^0\circ (\on{ev}^0)^R$, acting on $\bC^0$, is isomorphic,
as a plain endo-functor, to $$(\sT^{\on{pr}_s})^R\circ \sT^{\on{pr}_t}.$$ 

\smallskip

\noindent{\em(c)} The functor
$$\on{ev}^0:\on{Tot}(\bC^\bullet)\to \bC^0$$
is co-monadic. 

\end{lem}

\ssec{Calculating tensor products}

For future use, here are some examples, of how the Beck-Chevalley conditions can be used
to calculate tensor products of categories.

\sssec{}

Let $\bA$ be a monoidal DG category, and let $\bC^r$ and $\bC^l$ be right and left $\bA$-module categories
respectively. 

\medskip

Consider the corresponding ``Bar" complex, i.e., simplicial category $\on{Bar}^\bullet(\bC^r,\bA,\bC^l)$,
so that
$$|\on{Bar}^\bullet(\bC^r,\bA,\bC^l)|=\bC^r\underset{\bA}\otimes \bC^l.$$

The next assertion is tautological:

\begin{lem} \label{l:tensor prod Beck-Chevalley} \hfill

\medskip

\noindent{\em(i)} Assume that the action functors $\on{act}_{\bC^r,\bA}:\bC^r\otimes \bA\to \bC^r$  and $\on{act}_{\bA,\bC^l}:\bA\otimes \bC^l\to \bC^l$ 
and the monoidal operation $\on{mult}_\bA:\bA\otimes \bA\to \bA$ all admit continuous right adjoints, and the right adjoint to $\on{act}_{\bA,\bC^l}$
is also a map of $\bA$-module categories. 

\smallskip

Then the co-simplicial category $\on{Bar}^{\bullet,R}(\bC^r,\bA,\bC^l)$,
obtained from $\on{Bar}^\bullet(\bC^r,\bA,\bC^l)$ by passage to the right adjoint functors, satisfies
the monadic Beck-Chevalley condition.

\medskip

\noindent{\em(ii)} Assume that the action functors $\on{act}_{\bC^r,\bA}$ and $\on{act}_{\bA,\bC^l}$
and the monoidal operation $\on{mult}_\bA$ all admit left adjoints, and that the 
left adjoint of the action functor $\on{act}_{\bA,\bC^l}$ is also a map of $\bA$-module categories. 

\smallskip

Then the co-simplicial category 
$\on{Bar}^{\bullet,L}(\bC^r,\bA,\bC^l)$,
obtained from $\on{Bar}^\bullet(\bC^r,\bA,\bC^l)$ by passage to the left adjoint functors, satisfies
the co-monadic Beck-Chevalley condition.

\medskip

\noindent{\em(ii')} In the situation of point (ii), the simplicial category $\on{Bar}^\bullet(\bC^r,\bA,\bC^l)$
satisfies the monadic Beck-Chevalley condition.

\end{lem}

From here we will deduce:  

\begin{cor} \label{c:tensor prod Beck-Chevalley} 
In the situation of either of the points of 
\lemref{l:tensor prod Beck-Chevalley}, the right adjoint to 
$$\bC^r\otimes \bC^l\to \bC^r\underset{\bA}\otimes \bC^l$$
is monadic, and the resulting monad on $\bC^r\otimes \bC^l$ is isomorphic,
as a plain endo-functor to
$$(\on{act}_{\bC^r,\bA}\otimes  \on{Id}_{\bC^l})\circ (\on{Id}_{\bC^r}\otimes \on{act}_{\bA,\bC^l})^R.$$
\end{cor}

\begin{proof}
In the situation of \lemref{l:tensor prod Beck-Chevalley}(i), this follows from \lemref{l:monadicity}.

\medskip

In the situation of \lemref{l:tensor prod Beck-Chevalley}(ii), this follows from Lemmas \ref{l:monadicity},
and \ref{l:double right adjoint new}, combined with \cite[Lemma 1.3.3]{DGCat}. 

\end{proof}

\sssec{}

The next lemma implies that \corref{c:tensor prod Beck-Chevalley} is applicable
to the computation of the tensor product as along as the corresponding functors admit continuous right (resp., left)
adjoints:

\medskip

Assume that $\bC^l=\Vect$, with the action map
$\on{act}_{\bA,\bC^l}$ being given by a monoidal functor $\sF:\bA\to \Vect$. We can view $\sF$ as
a datum of augmentation on $\bA$, and in this case we will write
$\on{Bar}^\bullet(\bC^r,\bA)$ instead of $\on{Bar}^\bullet(\bC^r,\bA,\Vect)$. (When $\bC^r$ is also $\Vect$
with the action given by $\sF$, we will simply write $\on{Bar}^\bullet(\bA)$.)

\medskip

We note:

\begin{lem} \label{l:verify tensor prod Beck-Chevalley} Assume that the functor $\sF$ is conservative. 

\smallskip

\noindent{\em(i)} If the functor $\bA\to \Vect$ admits a continuous right adjoint, then this right
adjoint is automatically a map of $\bA$-module categories.

\smallskip

\noindent{\em(ii)} If the functor $\bA\to \Vect$ admits a left adjoint, then this left
adjoint is automatically a map of $\bA$-module categories.

\end{lem}

\begin{proof}

Since $\sF$ is conservative, it suffices to show that the composition
$$\Vect\overset{\sF^R}\to \bA\overset{\sF}\to \Vect$$
is a map of $\bA$-module categories. However, the latter is evident:
the functor in question is given by tensor product with $\sF\circ \sF^R(k)$.

\end{proof}

\section{Rigid monoidal categories} \label{s:rigid}

\ssec{The notion of rigidity} \label{ss:rigidity}

\sssec{}   \label{sss:rigidity}

Let $\bO$ be a monoidal DG category. We shall say that $\bO$ is \emph{rigid} if the following conditions are satisfied:

\begin{itemize}

\item The right adjoint $\on{mult}_\bO^R$ of the monoidal operation $\on{mult}_\bO:\bO\otimes \bO\to \bO$ is continuous;

\item The functor $\on{mult}_\bO^R:\bO\to \bO\otimes \bO$ is strictly (rather than lax) compactible with the
action of $\bO\otimes \bO$;

\item The right adjoint $\on{unit}_\bO^R$ of the unit functor $\on{unit}_\bO:\Vect\to \bO$ is continuous.

\end{itemize}

\sssec{}   \label{sss:rigid duality}

A basic feature of rigid monoidal categories is that the monoidal structure on $\bO$
gives rise to a canonical identification 
\begin{equation}  \label{e:rigid duality}
\bO^\vee\simeq \bO
\end{equation}
as plain DG categories. 

\medskip

Namely, we define the co-unit $\bO\otimes \bO\to \Vect$ as
$$\bO\otimes \bO\overset{\on{mult}_\bO}\longrightarrow \bO \overset{\on{unit}_\bO^R}\longrightarrow \Vect,$$
and the unit $\Vect\to \bO\otimes \bO$ as
$$\Vect \overset{\on{unit}_\bO}\longrightarrow \bO \overset{\on{mult}^R_\bO}\longrightarrow \bO\otimes \bO.$$

The fact that the above maps indeed define a data of duality is immediate from the assumption on $\bO$. 

\medskip

Under the identification \eqref{e:rigid duality}, the dual of $\on{mult}_\bO$ is $\on{mult}_\bO^R$,
and the dual of $\on{unit}_\bO$ is $\on{unit}_\bO^R$.

\sssec{} \label{sss:comp gen rigid}

Assume for a moment that $\bO$ is compactly generated. Then it is easy to show that $\bO$
is rigid in the sense of \secref{sss:rigidity} if and only if every compact object of $\bO$ admits
both left and right monoidal duals.

\medskip

In this case, the equialence \eqref{e:rigid duality} is given at the level of compact objects
by
$$\bo\mapsto \bo^\vee,$$
where $\bo^\vee$ denotes the right monoidal dual.

\ssec{Modules over a rigid category}

In this subsection we let $\bO$ be a rigid monoidal DG category.

\sssec{}

We have the following key assertion: 

\begin{prop} \label{p:extsience of right adjoint for rigid}
Let $\bC$ be an $\bO$-module category. Then the right adjoint to the action functor
$$\on{act}_{\bC,\bO}:\bO\otimes \bC\to \bC$$ 
is continuous and is given by the following functor (to be denoted $\on{co-act}_{\bC,\bO}$):
\begin{equation} \label{e:right adj}
\bC\overset{\on{unit}_\bO\otimes \on{Id}_\bC}\longrightarrow \bO\otimes \bC
\overset{\on{mult}_\bO^R\otimes \on{Id}_\bC}\longrightarrow \bO\otimes \bO\otimes \bC
\overset{\on{Id}_\bO\otimes \on{act}_{\bC,\bO}}\longrightarrow \bO\otimes \bC.
\end{equation}
\end{prop}

\begin{proof}
We construct adjunction data for the functors $\on{act}_{\bC,\bO}$ and $\on{co-act}_{\bC,\bO}$ as follows.
The composition $\on{act}_{\bC,\bO}\circ \on{co-act}_{\bC,\bO}$ is isomorphic to the composition
\begin{equation} \label{e:comp 1}
\bC\overset{\on{unit}_\bO\otimes \on{Id}_\bC}\longrightarrow \bO\otimes \bC
\overset{\on{mult}_\bO^R\otimes \on{Id}_\bC}\longrightarrow \bO\otimes \bO\otimes \bC
\overset{\on{mult}_\bO\otimes \on{Id}_\bC}\longrightarrow \bO\otimes \bC 
\overset{\on{act}_{\bC,\bO}}\longrightarrow \bC.
\end{equation}
The natural transformation $\on{mult}_\bO\circ \on{mult}_\bO^R\to \on{Id}_\bO$ defines a natural transformation
from \eqref{e:comp 1} to
$$\bC\overset{\on{unit}_\bO\otimes \on{Id}_\bC}\longrightarrow \bO\otimes \bC 
\overset{\on{act}_{\bC,\bO}}\longrightarrow \bC,$$
while the latter functor is canonically isomorphic to the identity functor on $\bC$. This defines
the co-unit of the adjunction. 

\medskip

The composition $ \on{co-act}_{\bC,\bO} \circ \on{act}_{\bC,\bO}$ is isomorphic to the composition
\begin{multline} \label{e:comp 2}
\bO\otimes \bC\overset{\on{unit}_\bO\otimes \on{Id}_\bO\otimes \on{Id}_\bC}\longrightarrow \bO\otimes \bO\otimes \bC
\overset{\on{mult}_\bO^R\otimes \on{Id}_\bO \otimes \on{Id}_\bC}\longrightarrow \bO\otimes \bO\otimes \bO\otimes \bC \to \\
\overset{\on{Id}_\bO \otimes \on{mult}_\bO\otimes \on{Id}_\bC}\longrightarrow 
\bO\otimes \bO\otimes \bC  \overset{\on{Id}_\bO \otimes\on{act}_{\bC,\bO}}\longrightarrow \bO\otimes \bC.
\end{multline}

The condition on $\bO$ implies that the diagram
$$
\CD
\bO\otimes \bO @>{\on{mult}_\bO^R\otimes \on{Id}_\bO}>>   \bO\otimes \bO\otimes \bO \\
@V{\on{mult}_\bO}VV    @VV{\on{Id}_\bO \otimes \on{mult}_\bO}V  \\
\bO  @>{\on{mult}_\bO^R}>>  \bO\otimes \bO
\endCD
$$
commutes. 

\medskip

Hence, the functor in \eqref{e:comp 2} can be rewritten as 
\begin{equation} \label{e:comp 3}
\bO\otimes \bC \overset{\on{mult}_\bO^R\otimes \on{Id}_\bO}\longrightarrow \bO\otimes \bO\otimes \bC 
 \overset{\on{Id}_\bO \otimes\on{act}_{\bC,\bO}}\longrightarrow \bO\otimes \bC.
\end{equation}
 
Now, the isomorphism
$$\on{mult}_\bO\circ (\on{Id}_\bO\otimes \on{unit}_\bO)\simeq \on{Id}_\bO$$
gives rise to a natural transformation
$$\on{Id}_\bO\otimes \on{unit}_\bO\to \on{mult}_\bO^R.$$

Hence, the functor in \eqref{e:comp 3} receives a natural transformation from
$$\bO\otimes \bC \overset{\on{Id}_\bO\otimes \on{unit}_\bO\otimes \on{Id}_\bO}\longrightarrow \bO\otimes \bO\otimes \bC 
\overset{\on{Id}_\bO \otimes\on{act}_{\bC,\bO}}\longrightarrow \bO\otimes \bC,$$
whereas the latter is the identity functor on $\bO\otimes \bC$. 

\medskip

This defines the unit for the 
$(\on{act}_{\bC,\bO},\on{co-act}_{\bC,\bO})$-adjunction
The verification that the above unit and co-unit satisfy the adjunction requirements is a straightforward verification.

\end{proof}

\sssec{}

As a formal corollary by diagram chase we obtain: 

\begin{cor}  \label{c:dual and rigid module}
For an $\bO$-module $\bC$, the right adjoint of the action map is a map of $\bO$-module categories.
\end{cor} 

\ssec{The dual co-monoidal category}  \label{ss:dual and adj rigid modules}

Let $\bO$ be a monoidal DG category, dualizable as a plain DG category. 

\sssec{}

The monoidal structure on $\bO$ gives rise to a co-monoidal structure on $\bO^\vee$, so that
we have a canonical equivalence
\begin{equation} \label{e:modules and comodules 1}
\bO\mmod\simeq \bO^\vee\commod,
\end{equation}
commuting with the forgetful functor to $\StinftyCat_{\on{cont}}$.

\sssec{}

From now in this subsection, let us assume that $\bO$ is rigid. 

\medskip

The assumption on $\bO$ implies that the \emph{right adjoint} of the monoidal structure
on $\bO$ defines a co-monoidal structure on $\bO$. We shall denote $\bO$, equipped with
this co-monoidal structure by $\bO_{\on{co}}$. 

\medskip

Furthermore, \propref{p:extsience of right adjoint for rigid} implies that the procedure of taking
the right adjoint of the action defines a functor
\begin{equation} \label{e:modules and comodules 2}
\bO\mmod\to  \bO_{\on{co}}\commod,
\end{equation}
which commutes with the forgetful functor to $\StinftyCat_{\on{cont}}$.

\sssec{}

Combining \eqref{e:modules and comodules 1} to \eqref{e:modules and comodules 2}, we obtain that there exists
a canonically defined functor
\begin{equation} \label{e:comodules and comodules}
\bO^\vee\commod\to \bO_{\on{co}}\commod,
\end{equation}
that commutes with the forgetful functor to $\StinftyCat_{\on{cont}}$.

\medskip

Hence, the functor \eqref{e:comodules and comodules} comes from a homomorphism of co-monoidal categories
\begin{equation} \label{e:another ident}
\phi_{\bO}:\bO^\vee\to \bO_{\on{co}}.
\end{equation}

\begin{lem}
The homomorphism \eqref{e:another ident} is an isomorphism.
\end{lem}

\begin{proof}
It follows from the construction that at the level of plain DG categories, the functor \eqref{e:another ident}
equals that of \eqref{e:rigid duality}.
\end{proof}

\sssec{Homomorphisms between rigid categories}  \label{sss:right adj and dual homomorphisms}

In a similar way to the construction of \eqref{e:another ident} we obtain:

\begin{prop} \label{p:right adj and dual homomorphisms}
Let $\sF:\bO_1\to \bO_2$ be a homomorphism between rigid monoidal categories. Then
the following diagram of homomorphisms of co-monoidal categories commutes:
$$
\CD
(\bO_2)_{\on{co}} @>{\sF^R}>>  (\bO_1)_{\on{co}} \\
@VVV    @VVV  \\
\bO_2^\vee   @>{\sF^\vee}>> \bO_1^\vee.
\endCD
$$
\end{prop}

\sssec{Left modules vs. right modules}  \label{sss:left vs right}

Note that if $\bO$ is rigid, then so is the category $\bO^{\on{o}}$ with the opposite monoidal structure.
Applying the construction of \secref{ss:dual and adj rigid modules}, we obtain an isomorphism
$$\phi_{\bO^{\on{o}}}:(\bO^{\on{o}})^\vee\to (\bO^{\on{o}})_{\on{co}}.$$

\medskip

We obtain that there exists a canonically defined monoidal automorphism $\psi_\bO$ of $\bO$, which 
intertwines the isomorphisms $\phi_{\bO^{\on{o}}}$ and 
$$(\bO^{\on{o}})^\vee\simeq (\bO^\vee)^{\on{o}}\overset{(\phi_{\bO})^{\on{o}}}\longrightarrow 
(\bO_{\on{co}})^{\on{o}}\simeq (\bO^{\on{o}})_{\on{co}}.$$

\medskip

We note that the automorphism $\psi_\bO$ is trivial when the monoidal struture on $\bO$ is commutative. 

\begin{rem}
When $\bO$ is compactly generated, at the level of compact objects, the automorphism $\psi_\bO$
acts as 
$$\bo\mapsto (\bo^\vee)^\vee.$$
\end{rem}

\ssec{Hochschild homology vs cohomology}

\sssec{}

Let $\bO$ be a monoidal DG category, and  
let $\bC^l$ and $\bC^r$ be a left and right $\bO$-module categories. We can form their
tensor product 
$$\bC^r\underset{\bO}\otimes \bC^l\in \StinftyCat_{\on{cont}},$$
which is computed as  
$$|\on{Bar}^\bullet(\bC^r,\bO,\bC^l)|.$$

\medskip

Assume that $\bO$ is dualizable as a plain DG category, and consider $\bO^\vee$ as a monoidal DG category.
Consider the co-tensor product
$$\bC^r\overset{\bO}\otimes \bC^l\in \StinftyCat_{\on{cont}},$$
defined as 
$$\on{Tot}(\on{co-Bar}^\bullet(\bC^r,\bO^\vee,\bC^l)).$$

\sssec{}

We now claim:

\begin{prop} \label{p:rigid Hochschild}
Assume that $\bO$ is rigid. Then 
there exists a canonical isomorphism in $\StinftyCat_{\on{cont}}$
$$\bC^r\underset{\bO}\otimes \bC^l\simeq (\bC^r)_{\psi_\bO}\overset{\bO}\otimes \bC^l,$$
where $(\bC^r)_{\psi_\bO}$ is the right $\bO$-module category, obtained from $\bC^r$
by twisting the action by the automorhism $\psi_\bO$ of \secref{sss:left vs right}.
\end{prop}

\begin{proof}

By \cite[Lemma 1.3.3]{DGCat}, the tensor product $\bC^r\underset{\bO}\otimes \bC^l$ can
be computed as the totalization of the co-simplicial category $\on{Bar}^{\bullet,R}(\bC^r,\bO,\bC^l)$,
obtained from $\on{Bar}^\bullet(\bC^r,\bO,\bC^l)$ by passing to the right adjoint functors.

\medskip

Now, by the construction of \secref{ss:dual and adj rigid modules}, the co-simplicial categories
$$\on{Bar}^{\bullet,R}(\bC^r,\bO,\bC^l) \text{ and } \on{co-Bar}^\bullet((\bC^r)_{\psi_\bO},\bO^\vee,\bC^l)$$
are canonically equivalent.

\end{proof}

\sssec{}

Let $\bC_1$ and $\bC_2$ be two left $\bO$-module categories, and assume that $\bC_1$ is dualizable
as a plain category. Consider $\bC_1^\vee$ as a right $\bO$-module category. Then we have
$$\uHom_{\bO}(\bC_1,\bC_2)\simeq \bC_1^\vee \overset{\bO}\otimes \bC^l.$$

\medskip

Hence, from \propref{p:rigid Hochschild} we obtain:

\medskip

\begin{cor} \label{c:rigid Hochschild}
Assume that $\bO$ is rigid. Then for $\bC_1$ and $\bC_2$ as above, there exists a canonical isomorphism
$$\uHom_{\bO}((\bC_1)_{\psi_\bO},\bC_2)\simeq \bC_1^\vee \underset{\bO}\otimes \bC^l.$$
\end{cor}

\sssec{}  

As another corollary of \propref{p:rigid Hochschild}, we obtain:

\begin{cor}  \label{c:any object as limit}
Let $\bO$ be rigid. Then any $\bO$-module category can be obtained as a totalization of a co-simplicial
object, whose terms are of the form $\bO\otimes \bD$ with $\bD\in \StinftyCat_{\on{cont}}$.
\end{cor}

\begin{proof}
For $\bC\in \bO\mod$, we have
$$\bC\simeq \bO\underset{\bO}\otimes \bC,$$
(where the right-hand side is regarded as a left $\bO$-module category via the left action of $\bO$
on itself). Now, by  \propref{p:rigid Hochschild}
$$\bO\underset{\bO}\otimes \bC\simeq \on{Tot}\left(\on{co-Bar}^\bullet(\bO_{\psi_\bO},\bO,\bC)\right)
\simeq \on{Tot}\left(\on{co-Bar}^\bullet(\bO,\bO,(\bC)_{\psi^{-1}_\bO})\right).$$

Now, the terms of $\on{co-Bar}^\bullet(\bO,\bO,(\bC)_{\psi^{-1}_\bO})$, when regarded as left 
$\bO$-modules, have the required form.

\end{proof}

\sssec{}

Finally, combining Corollaries \ref{c:dual and rigid module} and \ref{c:tensor prod Beck-Chevalley}, we obtain:

\begin{cor}  \label{c:dual and rigid module bis}
Let $\bO$ be rigid. 

\smallskip

\noindent{\em(a)} The co-simplicial category $\on{co-Bar}^\bullet(\bC^r,\bO^\vee,\bC^l)$ satisfies
the monadic Beck-Chevalley condition. 

\smallskip

\noindent{\em(b)} The functor of evaluation on $0$-simplices
$$\on{Tot}\left(\on{co-Bar}^\bullet(\bC^r,\bO^\vee,\bC^l)\right)\to \bC^r\otimes \bC^l$$
admits a left adjoint and 
is monadic. The resulting monad, viewed
as a plain endo-functor of $\bC^r\otimes \bC^l$, identifies with the composition
$$\bC^r\otimes \bC^l\overset{\on{Id}_{\bC^r}\otimes \on{co-act}_{\bC^l,\bO}}\longrightarrow 
\bC^r\otimes \bO\otimes \bC^l \overset{\on{act}_{\bC^r,\bO}\otimes \on{Id}_{\bC^l}}\longrightarrow \bC^r\otimes \bC^l.$$
\end{cor}

\ssec{Dualizability of modules over a rigid category}

\sssec{}

Let $\bO$ be a monoidal DG category, and let $\bC^r$ and $\bC^l$ be a right and left $\bO$-module categories,
respectively. 

\medskip

Recall that a data of duality between $\bC^r$ and $\bC^l$ as $\bO$-module categories consists of a unit map
$$\Vect\to \bC^r\underset{\bO}\otimes \bC^l,$$
which is a map in $\StinftyCat_{\on{cont}}$, and a co-unit map
$$\bC^l\otimes \bC^r\to \bO,$$
which is a map on $(\bO\otimes \bO)\mmod$, which satisfy the usual axioms. 

\medskip

Equivalently, the datum of duality between $\bC^r$ and $\bC^l$ as $\bO$-module categories is 
a functorial equivalence
$$\uHom_{\bO}(\bC^l,\bC)\simeq \bC^r\underset{\bO}\otimes \bC,\quad \bC\in \bO\mmod.$$

\begin{rem}
Assume for a moment that $\bO$ is symmetric monoidal. Then it is easy to see that
a duality data between two $\bO$-module categories is equivalent to a duality data
inside the symmetric monoidal DG category $\bO\mmod$.
\end{rem}

\sssec{}  \label{sss:dualizable in rigid}

We claim:

\begin{prop}
Assume that $\bO$ is rigid. Then $\bC^l\in \bO\mmod$ is dualizable if and only if it is dualiable as a plain 
DG category. The DG category underlying the $\bO$-module dual of $\bO$ is canonically equivalent to
$(\bC^l)^\vee$.
\end{prop}

\begin{proof}

Suppose first being given a duality data between $\bC^r$ and $\bC^l$ as $\bO$-module categories.
We define a duality data between $\bC^r$ and $\bC^l$ as plain DG categories by taking the unit to be 
$$\Vect\to \bC^r\underset{\bO}\otimes \bC^l\to \bC^r\otimes \bC^l,$$
where the second arrow is the right adjoint to the tautological functor $\bC^r\otimes \bC^l\to \bC^r\underset{\bO}\otimes \bC^l$
(it is continuous, e.g., by \corref{c:dual and rigid module bis}(b)). We take the co-unit to be
$$\bC^l\otimes \bC^r\to \bO\overset{\on{unit}_\bO^R}\longrightarrow \Vect.$$

The fact that the duality axioms hold is straightforward. 

\medskip

Vice versa, let $\bC^l$ be dualizable as a plain DG category. Set $\bC^r:=((\bC^l)^\vee)_{\psi_\bO}$. Now, the functorial
equivalence
$$\uHom_{\bO}(\bC^l,\bC)\simeq \bC^r\underset{\bO}\otimes \bC$$
follows from \corref{c:rigid Hochschild}.

\end{proof}

\section{Commutative Hopf algebras}  \label{s:Hopf}

\ssec{The setting}

\sssec{}

Let $\bO$ be a symmetric monoidal category\footnote{In this section $\bO$ is not necessarily stable, e.g.,
$\bO=\StinftyCat_{\on{cont}}$.}. 

\medskip

Consider the category $\on{co-Alg}(\bO)$ of co-algebras in $\bO$. We regard it as a symmetric monoidal
category under the operation of tensor product.

\medskip

By a (commutative) bi-algebra in $\bO$ we will mean a (commutative) algebra in $\on{co-Alg}(\bO)$. We shall say that
a (commutative) bi-algebra is a (commutative) Hopf algebra if it is such at the level of the underlying ordinary categories
(i.e., if it admits a homotopy antipode).

\sssec{}

Recall that if $A$ is an augmented co-algebra object in a monoidal category $\bO$, we can canonically attach to it a
co-simplicial object $$\on{co-Bar}^\bullet(A).$$

\medskip 

If $A$ is a bi-algebra, the object $\on{co-Bar}^\bullet(A)\in \bO^{\bDelta}$ naturally lifts to one in 
$$\on{Alg}(\bO^{\bDelta})\simeq (\on{Alg}(\bO))^{\bDelta},$$
i.e., $\on{co-Bar}^\bullet(A)$ is a co-simplicial algebra in $\bO$, or equivalently, 
a co-simplicial object of $\bO$ endowed with a compatible family of simplex-wise monoidal structures. 

\sssec{}

Consider the corresponding co-simplicial category 
$$\on{co-Bar}^\bullet(A)\mod$$
(where the transition functors are given by tensoring up along the maps in $\on{co-Bar}^\bullet(A)$).

\medskip

Consider the totalization
$$\on{Tot}\left(\on{co-Bar}^\bullet(A)\mod\right).$$

\medskip

The goal of this Appendix is to prove the following:

\begin{propconstr}  \label{p:Hopf}
Let $A$ be a Hopf algebra in $\bO$. Then there exists a canonical equivalence of categories
$$A\comod\to \on{Tot}\left(\on{co-Bar}^\bullet(A)\mod\right).$$
\end{propconstr}

\ssec{Construction of the functor}

To construct the sought-for functor in \propref{p:Hopf}
we proceed as follows.

\sssec{}

Let 
$$\on{coAlg+comod}(\bO)$$ be the category of pairs $$(A\in \on{co-Alg}^{\on{aug}}(\bO), M\in A\comod(\bO)).$$ Note that 
the assignment $A\rightsquigarrow \on{co-Bar}^\bullet(A)$ can be extended to a functor 
\begin{equation} \label{e:coBar modules}
\on{coAlg+comod}(\bO) \rightsquigarrow \on{co-Bar}^\bullet(A,M')\in \bO^{\bDelta}.
\end{equation}

Moreover, this functor is (symmetric) monoidal, when on $\on{coAlg+comod}(\bO)$ we consider the (symmetric) monoidal
structure
$$(A_1,M_1)\otimes (A_2,M_2):=(A_1\otimes A_2,M_1\otimes M_2),$$
and on $\bO^{\bDelta}$ the component-wise (symmetric) monoidal structure.

\sssec{}

Note that for a bi-algebra $A$, the pair $(A,\one_\bO)$ is naturally an algebra object in the category $\on{coAlg+comod}(\bO)$,
and we have a canonically defined functor
\begin{equation}    \label{e:comodules pairs}
A\comod\to (A,\one_\bO)\mod(\on{coAlg+comod}(\bO)).
\end{equation}

\medskip

We note that the value of the functor \eqref{e:coBar modules} on $(A,\one_\bO)$ is 
$\on{co-Bar}^\bullet(A)\in \on{Alg}(\bO^{\bDelta})$. Now composing the functor \eqref{e:comodules pairs}
and \eqref{e:coBar modules}, we obtain a functor
\begin{equation} \label{e:pre Hopf}
A\comod \to \on{co-Bar}^\bullet(A)\mod(\bO^{\bDelta}).
\end{equation}

\sssec{}

Now, it is easy to see that if $A$ is a Hopf algebra, then  
for $M\in A\comod$ and a map $[j]\to [i]$ in $\bDelta$, for the
corresponding map of algebras and modules 
$$\on{co-Bar}^j(A)\to \on{co-Bar}^i(A), \quad \on{co-Bar}^j(A,M)\to \on{co-Bar}^i(A,M),$$
the resulting map
$$ \on{co-Bar}^i(A)\underset{\on{co-Bar}^j(A)}\otimes  \on{co-Bar}^j(A,M)\to \on{co-Bar}^i(A,M)$$
is an isomorphism. 

\medskip

Hence, the functor \eqref{e:pre Hopf} defines a functor 

\begin{equation} \label{e:Hopf}
A\comod\to \on{Tot}\left(\on{co-Bar}^\bullet(A)\mod\right).
\end{equation}

\ssec{Proof of the equivalence}

\sssec{}

The fact that $A$ is a Hopf algebra implies that the co-simplicial 
category $\on{co-Bar}^\bullet(A)\mod$ satisfies the co-monadic 
Beck-Chevalley condition. Hence, by \lemref{l:comonadicity}, the functor $\on{ev}^0$ of evaluation
on $0$-simplices is co-monadic, and the resulting co-monad on $\bO$ is given by tensor 
product with $A$. 

\sssec{}

Consider now the composition 
$$A\comod\to \on{Tot}\left(\on{co-Bar}^\bullet(A)\mod\right)\overset{\on{ev}^0}\longrightarrow \bO.$$

It identifies with the forgetful functor $\on{oblv}_A:A\comod\to \bO$. In particular, it is co-monadic,
and the resulting monad on $\bO$ being the tensor product with $A$.

\sssec{}

It remains to show that the map of co-monads, induced by the functor \eqref{e:Hopf}, is an isomorphism
as plain endo-functors of $\bO$. However, it is easy to see that the natural
transformation in question is the identity map on the functor of tensor product by $A$.

\end{document}